\documentclass[10pt]{amsart}

\usepackage{amsfonts,amsmath,latexsym,amssymb,verbatim,amsbsy,amsthm}
\usepackage{dsfont,bm}
\usepackage{enumitem}
\usepackage{multirow}

\usepackage[top=1in, bottom=1in, left=1in, right=1in]{geometry}

\usepackage[dvipsnames]{xcolor}

\usepackage[colorlinks=true, pdfstartview=FitV, linkcolor=RoyalBlue,citecolor=ForestGreen, urlcolor=blue]{hyperref}

\theoremstyle{plain}
\newtheorem{THEOREM}{Theorem}[section]

\newtheorem{theorem}[THEOREM]{Theorem}
\newtheorem{corollary}[THEOREM]{Corollary}
\newtheorem{lemma}[THEOREM]{Lemma}
\newtheorem{proposition}[THEOREM]{Proposition}

\theoremstyle{definition}

\newtheorem{definition}[THEOREM]{Definition}

\theoremstyle{remark}

\newcommand{\thm}[1]{Theorem~\ref{#1}}
\newcommand{\lem}[1]{Lemma~\ref{#1}}
\newcommand{\defin}[1]{Definition~\ref{#1}}

\newcommand{\cor}[1]{Corollary~\ref{#1}}
\newcommand{\prop}[1]{Proposition~\ref{#1}}

\newcommand{\sect}[1]{Section~\ref{#1}}

\def \b {\beta}
\def \g {\gamma}
\def \d {\delta}
\def \e {\varepsilon}
\def \f {\varphi}
\def \k {\kappa}
\def \l {\lambda}
\def \n {\nabla}
\def \s {\sigma}
\def \t {\tau}
\def \th {\theta}
\def \o {\omega}
\def \w {\omega}

\def \D {\Delta}

\def \L {\Lambda}
\def \Th {\Theta}
\def \O {\Omega}

\def \bk {{\bf k}}
\def \bl {{\bf l}}

\def \oW { \overline{W} }

\def \oS { \overline{S} }

\def \oSr { \overline{Sr} }

 \def \st {\mathrm{s}}
  
 \def \rmb {\mathrm{b}}
  \def \rmc {\mathrm{c}}
  \def \rmw {\mathrm{w}}
    \def \rmd {\mathrm{d}}
      \def \rmh {\mathrm{h}}
      \def \rmm {\mathrm{m}}
      \def \rmx {\mathrm{x}}
       
       \def \rmv {\mathrm{v}}
       \def \rmn {\mathrm{n}}

\def \rmA {\mathrm{A}}

\def \cA {\mathcal{A}}

\def \cC {\mathcal{C}}
\def \cD {\mathcal{D}}
\def \cE {\mathcal{E}}

\def \cH {\mathcal{H}}
\def \cI {\mathcal{I}}

\def \cO {\mathcal{O}}
\def \cP {\mathcal{P}}

\def \oW { \overline{W} }

\def \oS { \overline{S} }

\def \oSr { \overline{Sr} }
\def \oTh {\overline{\Th}}

\def \oTh { \overline{\Th} }

\newcommand{\N}{\ensuremath{\mathbb{N}}}   
\newcommand{\Z}{\ensuremath{\mathbb{Z}}}   
\newcommand{\R}{\ensuremath{\mathbb{R}}}   
\newcommand{\T}{\ensuremath{\mathbb{T}}}   
\newcommand{\I}{\ensuremath{\mathbb{I}}}   

\def \loc {\mathrm{loc}}

\def \one {{\mathds{1}}}

\def \Lip {\mathrm{Lip}}

\newcommand{\jap}[1]{\left\langle #1 \right\rangle}
\newcommand{\ave}[1]{ \left[ #1 \right]}

\DeclareMathOperator{\supp}{supp} %

\def \p {\partial}

\def \ss {\subset}

\def \GL {Gr\"onwall's Lemma}
\def \HI {H\"older inequality}

\def \CK{ Csisz\'ar-Kullback inequality}

\renewcommand{\geq}{\geqslant}

\renewcommand{\leq}{\leqslant}

\def \dx  {\, \mbox{d}x}

\def \dxi  {\, \mbox{d}\xi}
\def \dt  {\, \mbox{d}t}

\def \dy  {\, \mbox{d}y}
\def \dz  {\, \mbox{d}z}

\def \dr  {\, \mbox{d}r}

\def \ds  {\, \mbox{d}s}
\def \dw  {\, \mbox{d}w}

\def \dB  {\, \mbox{d}B}

\def \dmu  {\, \mbox{d}\mu}

\def \dk  {\, \mbox{d}\kappa}

\def \deta  {\, \mbox{d}\eta}
\def \dtau  {\, \mbox{d}\tau}
\def \drho  {\, \mbox{d}\rho}

\def \dv  {\, \mbox{d} v}

\def \ddt  {\frac{\mbox{d\,\,}}{\mbox{d}t}}

\def \domain {{\O \times \R^n}}

\begin{document}

\title[Fokker-Planck-Alignment equations]{Global well-posedness and relaxation for solutions of the Fokker-Planck-Alignment equations}

\author{Roman Shvydkoy}

\address{851 S Morgan St, M/C 249, Department of Mathematics, Statistics and Computer Science, University of Illinois at Chicago, Chicago, IL 60607}

\email{shvydkoy@uic.edu}


\date{\today}


\thanks{\textbf{Acknowledgment.}  
	This work was  supported in part by NSF
	grant  DMS-2405326 and Simons Foundation.}

\begin{abstract} In this paper we prove global existence of weak solutions, their regularization, and  relaxation for large data  for a broad class of Fokker-Planck-Alignment models which appear in collective dynamics. The main feature of these results, as opposed to previously known ones, is the lack of regularity or no-vacuum requirements on the initial data. With a particular application to the classical kinetic Cucker-Smale model, we demonstrate that any bounded data with finite higher moment, $f_0 \in L^1(1+ |v|^q) \cap L^\infty$, $q \geq n+4$, gives rise to a global instantly smooth solution, satisfying entropy equality and relaxing exponentially fast.

The results are achieved through the use of  a new thickness-based renormalization procedure, which circumvents the problem of degenerate diffusion in non-perturbative regime.
\end{abstract}

\maketitle 
 \setcounter{tocdepth}{1}
\tableofcontents

\section{Introduction}

In this paper we study regularity and long time behavior of solutions to a class of kinetic models of Fokker-Planck type arising in collective dynamics,
\begin{equation}\label{e:FPA}
\p_t f + v \cdot \n_x f = \s \D_v f + \n_v \cdot ( \cA(f) f ).
\end{equation}
Here we assume the flock is confined to a periodic-in-$x$ environment $x\in \O= L \T^n$, and $v\in \R^n$. The dynamics is driven by an alignment force  $\cA(f)$ given by
\begin{equation}\label{ }
\cA(f)(x,v) = \int_{\domain} \phi_\rho(x,y) (v - w) f(y,w) \dw \dy,
\end{equation}
where $\phi_\rho \in L^1(\rho \times \rho)$ is a non-negative communication kernel that protocols interactions between agents, and $\s \geq 0$ is a diffusion coefficient chosen to be proportional to the total strength of influence of the flock on location $x$,  
\begin{equation}\label{e:sigma}
\s(x) = \s_0  \st_\rho(x), \quad \st_\rho(x) = \int_\O \phi_\rho(x,y) \rho(y) \dy, \ \s_0 >0.
\end{equation}

Due to insufficient smoothness of the coefficients, only $\cA(f)\in L^2$, $\s\in L^\infty$, and degenerate diffusion $\s\geq 0$ the well-posedness results for \eqref{e:FPA} have been scarce, see \cite{BCC2011,DFT2010,KMT2013,S-EA,S-hypo} and discussion below. Those mostly depend on thickness assumption on data or solutions are constructed in weak finite moment classes. The main purpose of this paper is to establish well-posedness, regularization, and relaxation of solutions to \eqref{e:FPA} starting from arbitrary large data. 

Models of type \eqref{e:FPA} first appeared in Ha and Tadmor \cite{HT2008} and later Ha and Liu \cite{HL2009} as a mean-field limit of the celebrated discrete Cucker-Smale system, see \cite{CS2007a,CS2007b}
\begin{equation}\label{e:CS}
\dot{x}_i = v_i, \quad \dot{v}_i =\sum_{j=1}^N m_j \phi(x_i-x_j) (v_j - v_i). 
\end{equation}
where $\phi$ is a smooth radially decreasing density-independent kernel. We refer to these surveys \cite{ABFHKPPS,Axel97,Ben2005,Edel2001,Darwin,VZ2012,MT2014,MP2018,S-book,Tadmor-notices} for the broad range of applications of the agent-based  alignment systems. 

The inclusion of noise $\sqrt{2\s} \dB_i$ in the momentum equation has become a standard attribute of many agent-based models of swarming to account for random fluctuations caused by either internal or external factors, see \cite{BCC2011,Choi2016,DFT2010, HaXZh2018,KMT2013,Shvart2019}. While formally, the probability distribution of a random process corresponding to such a system solves the Fokker-Planck-Alignment equation \eqref{e:FPA}, the rigorous analysis of stochastic mean-field limit requires scrutiny.  The work of Bolley,  Ca\~{n}izo, and Carrillo  \cite{BCC2011} initially treated the case of a constant noise $\s = \s_0>0$ and  sub-Maxwellian data  $\int_\domain e^{a|v|^p} f \dv\dx <\infty$. Subsequently, the result has been extended in  \cite{S-EA} to include density-dependent parameters as in \eqref{e:FPA}.

The renewed interest in alignment models of type \eqref{e:FPA} is motivated by its connection to  \emph{emergence}, which is a cumulative term for a range of self-organization phenomena observed in systems driven by purely local interaction rules.  In the context of \eqref{e:CS} emergence presents itself as the asymptotic alignment $\max |v_i - v_j| \to 0$ as $t\to \infty$. It has been established under the fat tail condition $\int_0^\infty \phi(r) \dr = \infty$ on $\O = \R^n$  in the aforementioned works \cite{CS2007a,CS2007b,HT2008,HL2009}, and this condition is known to be sharp, see \cite{S-book}. For purely local kernels $\phi(r) \sim \one_{r<r_0}$ one obvious example of desynchronization is given by two radially divergent agents starting from a distance $r>r_0$. Such kind of dispersion can be controlled by an additional confinement force \cite{ShuT2019}, or without any force on the periodic domain $\O = L \T^n$ where confinement is automatic. In this case examples of desynchronized agents become more rare -- those are agents traveling along periodic orbits with mutually rational velocity directions -- and so the asymptotic alignment can only be expected for generic data. So far this problem has been solved in limited cases: for the sticky particle model, for $n=1, N\in \N$, and for $n\in \N, N=2$, see \cite{S-GAC,DS2021}. 

Stochastic forces preclude solutions from locking themselves in a periodic motion. The expected asymptotic behavior in this case is relaxation to an equilibrium $f \to \mu$, which belongs to the kernel of the Fokker-Planck operator. For the multiplicative noise such as \eqref{e:sigma} we have a family of equilibria given by global Maxwellians,
\begin{equation}\label{e:Max}
\mu_{\s_0,\bar{u}} = \frac{1}{|\O|(2\pi \s_0)^{n/2}} e^{- \frac{|v - \bar{u}|^2}{2\s_0}}, \quad \bar{u}\in \R^n.
\end{equation}
So, if relaxation to one of these equilibria is established (for conservative systems $\bar{u}$ is determined by the initial momentum) then the alignment can be recovered in the vanishing noise limit, $f \to \mu_{\s_0,\bar{u}} \to \d_0(v-\bar{u})$.  

Research in this direction started with the work of  Duan et al \cite{DFT2010}, which established relaxation for perturbative data $g = (f - \mu_{\s,\bar{u}}) / \sqrt{\mu_{\s,\bar{u}}}$ small in a higher order Sobolev metric. For the local alignment force $\n_v \cdot ((v-u) f)$ (formally corresponding to $\phi_\rho = \frac{1}{\rho(x)} \d_0(x-y)$) the same result was proved by Choi \cite{Choi2016}. Both works use stability analysis inspired by new developments for collisional models at the time. 
The first result for large data was established in \cite{S-hypo} for communication kernel given by $\phi_\rho = \int_\O \frac{\psi(x-z) \psi(y-z)}{\rho \ast \psi(z)} \dz$, see \ref{Cg+}-protocol below, and under uniform bound on the macroscopic field $\sup_{t>0} \|u(t)\|_\infty <\infty$. The technique uses an extension of the linear hypercoercivity method of Villani  \cite{Villani-hypo} that fits due to a special cancelation in the nonlinear alignment force on the Fisher information level. Finally, an unconditional relaxation for large data for the Cucker-Smale-based model with a Bochner-positive kernel $\phi = \psi \ast \psi$ was established in \cite{S-EA}.  All large data results, however, have been proved for a given classical solution, which is only known to exist starting from thick data and for locally bounded alignment forces, see \cite{S-EA}. 

In light of these developments it has become imperative to develop a more inclusive well-posedness theory for the Fokker-Planck-Alignment models \eqref{e:FPA} that incorporates thin data away from equilibrium. Generally, weak solutions in Lebesgue classes can be constructed either from the mean-field limit for sub-Maxwellian data \cite{BCC2011}, or through the use the Averaging Lemma as in Karper, Mellet, and Trivisa \cite{KMT2013} for any $f\in L^1(1+|v|^2+|x|^2) \cap L^\infty$. Such solutions are only shown to satisfy entropy \emph{inequality} in the integral form, which is not sufficient for the relaxation analysis (in view of the need to use \GL\ in the ultimate construction of the Lyapunov function and access to the higher order Fisher regularity). On a related note,  regularization  for models similar to \eqref{e:FPA} with no alignment and noise depending on the density $\rho$ have been established in \cite{AZ2024,IM2021}. Their approach relies on an a priori $L^\infty$-control for the distribution $h = f/\mu \in L^\infty$, which is unavailable in the presence of alignment force.
 
Let us state our main result as it pertains to the most classical Cucker-Smale based model \eqref{e:FPA}. We refer to \sect{ss:notation} for the relevant notation.
 
\begin{theorem}\label{t:mainrelaxCS}
Consider the Fokker-Planck-Alignment equation \eqref{e:FPA} based on the Cucker-Smale alignment protocol $\phi_\rho(x,y) = \phi(x-y)$, with a smooth local kernel $\phi(r) \geq c_0 \one_{|r|<r_0}$. For any data $f_0 \in  L^1(1+|v|^q) \cap L^\infty$, $q \geq n+4$, there exists a global weak solution in the same class and with finite Fisher information $\st_\rho \frac{|\n_v f|^2}{f} \in L^1_{t,x,v}$, satisfying the entropy law \eqref{e:entropylaw}. Every such solution gains uniform Gaussian tails 
\begin{equation}\label{e:GaussIntro}
f(t,x,v) \geq b e^{-a |v|^2} , \qquad \forall (t,x,v) \in [\e,\infty) \times \T^n \times\R^n .
\end{equation}
 and regularizes into $f\in C([\e,T]; H^m_q) \cap C^1([\e,T]; H^{m-1}_{q-3})$ for any $m\in \N$ instantly. Solutions are unique from non-vacuous data, and every solution relaxes to the Maxwellian exponentially  fast,
\begin{equation}\label{e:introrelax}
\|f(t) - \mu_{\s_0,\bar{u}_0}\|_{L^1(\domain)} \leq c  e^{- c t},
\end{equation}
where $c>0$ depending only on the initial condition and fixed parameters of the model.
\end{theorem}

The existence and uniqueness are a consequence of our main \thm{t:weak}. Gain of positivity and instant regularization are addressed in \prop{p:Gauss} and \prop{p:instantreg}, respectively.  Finally, the relaxation is a consequence of \cite[Proposition 8.1]{S-EA} which is stated in \prop{p:relax}.  All these component are proved for a much broader class of alignment models. We note that the exact same statement of \thm{t:mainrelaxCS} holds for all communication protocols that are symmetric, local, and of mapping type $(2,\infty)$, see \sect{s:basic} for the definitions.  

 Let us point out the main technical issues. First and foremost, thin data produces a degenerate strength $\st_\rho$ which meddles both with the ellipticity of the equation and a priori Fisher regularity, $\st_\rho \frac{|\n_v f|^2}{f} \in L^1_{t,x,v}$. The latter is critical for the entropy law in the same way as the Onsager-$1/3$ regularity is critical  for the energy conservation in incompressible fluids, see \cite{CCFS2008,cet,Isett2018,ds-1}. To prove the entropy law we resort to DiPerna-Lions renormalization as a tool to establish a distributional equation for the Boltzmann functional $\b(f) = f \log f$. This implies the entropy {\em equality}  starting from time $t=0$ as well as other useful tools such as propagation of moments and weak maximum principle.  The degeneracy in  Fisher information, however, stands in the way of controlling the DiPerna-Lions commutators. To circumvent this difficulty we implement a new approach by adding an extra step - a ``pre-renormalization" using thickness-weighted 
mollification, $\frac{(\th f)_{\e_1}}{\th + \e_2}$, where $\th$ is the thickness of the flock, and $\e_1,\e_2 > 0$ are small parameters, see \defin{d:th} and  \eqref{e:FPAmoll}. The pre-renormalization allows to gain direct access to the Fisher-regularity of $\th f$ and use it in the estimates on commutators  encountered in the main renormalization procedure.

With the  entropy law at hand, further analysis can be essentially embedded into the framework developed in \cite{S-EA}. We first obtain the spread of positivity \eqref{e:GaussIntro}, and consequently, restore the thickness and non-degeneracy of the diffusion $\s>0$. This further implies three facts -- regularity of coefficients in the equation \eqref{e:FPA}, hypoelliptic regularization of solution into any $H^m_q$, and gain of spectral gap uniformly in time for the alignment operator, see \lem{l:spgap}. At this point we have all the ingredients for the application of the general relaxation result of \cite[Proposition 8.1]{S-EA} to prove \eqref{e:introrelax}.

\subsection{Notation}\label{ss:notation} We denote by $\cP(\O)$ the set of Borel probability measures on $\O$. Given a measure $\mu$ on $\O$ we denote $(u,v)_\mu = \int_\O u(x)v(x) \dmu(x)$, and $L^p(\mu)=\{f: \int_{\O} |f|^p \dmu <\infty\}$. The classical space of tempered distributions is denoted $\cD'(\domain)$. For a probability density $f$ over $\domain$ the energy and entropy are denoted, respectively,
\begin{align}
\cE & = \frac12  \int_\domain |v|^2 f  \dv \dx \label{e:energy}\\
\cH & =  \s_0 \int_\domain f \log f  \dv \dx + \frac12  \int_\domain |v|^2 f  \dv \dx. \label{e:entropy}
\end{align}
Functionally, the wellposedness of our kinetic models requires the use of weighted Lebesgue and Sobolev spaces.  Denoting the Japanese brackets by $\jap{v} = (1+|v|^2)^{\frac12}$, we define
\begin{equation}\label{}
L^p_q(\domain) = \left\{ f:  \int_\domain \jap{v}^q |f|^p \dv \dx <\infty  \right\},
\end{equation}
and 
\begin{equation}\label{e:Sobdef}
H^{m}_q(\domain) =  \left\{ f :  \sum_{ 2|\bk| + | \bl | \leq 2m}   \int_\domain  \jap{v}^{q - 2|\bk| - | \bl |  } | \p^{\bk}_{x} \p_v^{\bl} f |^2 \dv\dx <\infty \right\}.
\end{equation}
Note that the higher weights are placed on lower order derivatives and if $2m > q$, then the top derivatives allow to grow at infinity. 

We adopt  the classical definition of the Kantorovich-Rubinstein distance between probability measures $\rho',\rho''\in \cP(\O)$:
\begin{equation}\label{e:KR}
	W_1(\rho',\rho'') = \sup_{ \Lip(h) \leq 1} \left| \int_{\O} h(x) [ \drho'(x) - \drho''(x) ]\right|.
\end{equation}

\section{Communication protocols and description of the results}\label{s:basic}

In this section we give detailed description of  our results. Following \cite{S-EA}, we adopt the framework  based on the concept of a regular communication protocol, which we present in a local format in order to establish most inclusive results.

\subsection{Communication protocols}\label{ss:cp}

By the {\em communication kernel} of a flocking model we understand a family of non-negative functions $\phi_\rho: \O \times \O \to [0,\infty]$, $\rho \in \cP(\O)$,  which are uniformly in  $L^1(\rho\otimes\rho)$,
\begin{equation}\label{e:Fint}
\sup_{\rho \in \cP(\O)} \int_{\O \times \O} \phi_\rho(x,y) \drho(y) \drho(x) <\infty.
\end{equation}
The  corresponding integral operators are denoted
\begin{equation}\label{e:warep}
\rmw_\rho(u)(x) =  \int_\O \phi_\rho(x,y) u(y) \drho(y).
\end{equation}
A priori from \eqref{e:Fint} such operators are only bounded on $L^\infty(\rho) \to L^1(\rho)$, however further assumptions will widen the range of mapping properties considerably. 
It is crucial to  have kernels defined outside the support of the flock in order to have a capability to compare them for distinct flocks $\rho',\rho''$. In particular, the integral formula \eqref{e:warep} can be considered for $x$ outside the support of $\rho$, although a priori we do not guarantee that such integrals will be finite.

Another fundamental component of a flocking model is called thickness, which is a measure of presence of the flock at a given location.

\newlist{thickness}{enumerate}{1}

\setlist[thickness, 1]
{label=($\th$\arabic{thicknessi}), 
leftmargin=30pt,
}

\begin{definition}\label{d:th}
A {\em thickness} is a uniformly bounded family of functions $\th_\rho : \O \to [0,\oTh]$ satisfying the following conditions:
\begin{thickness}
\item \label{i:th1} $\th_\rho$ is lower semi-continuous;
\item \label{i:th2} $\rho( \{\th_\rho = 0 \} ) = 0$, for all $\rho \in L^1(\O)$;
\item  \label{i:th3} there exists a positive function $c(r)>0$, $r>0$,  such that $ \th_\rho(x) \geq c(\min \rho)$, for all $\rho\in \O$;
\item  \label{i:th4} Continuity-in-$\rho$: there exists a $C>0$ such that for all $\rho',\rho''\in \cP(\O)$,
\[
\sup_\O| \th_{\rho'} (x) - \th_{\rho''}(x)| \leq C W_1(\rho',\rho'').
\]
\end{thickness}
The thickness of a set $B \ss \O$ is defined by 
\begin{equation}\label{e:thS}
\th_\rho(B)  = \inf_{x\in B} \th_\rho(x) ,
\end{equation}
and the  {\em global thickness} is defined by $\th_\rho(\O)$.
\end{definition}

One particular example of thickness, which was traditionally viewed as the universal choice for many alignment models is the {\em ball-thickness}. The ball-thickness represents a tapered  mass of a ball around a given location. To define it let $\chi \in C^\infty_0$ be a standard radially symmetric mollifier supported on the unit ball,  and fix some $r_0>0$.  Then
\begin{equation}\label{e:ballth}
\th_{\rho,r_0}(x) = \int_\O \rho(x -r_0 y) \chi(y) \dy.
\end{equation}
A distinct feature of the ball-thickness is locality -- the mass of the flock is measured in the immediate neighborhood of the agent $x$.  This may not always be the case as for instance in segregation models \ref{Segg} we present later.

\begin{lemma}\label{l:contth}
Suppose $ \rho \in \cP(\O)$ satisfies the continuity equation
\[
\p_t \rho + \n_x \cdot (u\rho) = 0,
\]
weakly with $u\in L^\infty([0,T); L^1(\rho))$. Then  for every point $x\in \O$, the function $t \to \th_{\rho(t)}(x)$ is Lipschitz on $[0,T)$ with uniformly bounded distributional derivative, 
\[
\sup_{x\in \O}| \p_t \th_\rho(x)| \leq C \| u\|_{L^1(\rho)}.
\]
\end{lemma}
\begin{proof}
For every point $x\in \O$, by \ref{i:th4}, we have 
\[
|\th_{\rho(t)}(x) - \th_{\rho(s)}(x)| \leq C W_1(\rho(t), \rho(s)).
\]
If $h: |\n h|\leq 1$, we have
\[
\int_\O h \drho(t) - \int_\O h \drho(s) = \int_s^t \int_\O \n h(x) u(x) \drho(\t,x) \dtau \leq \int_s^t \| u\|_{L^1(\rho)} \dt \leq C (t-s).
\]
So, $W_1(\rho(t), \rho(s)) \leq C (t-s)$, and the conclusions follow from the Rademacher Theorem.
\end{proof}

From analytical standpoint the role of thickness is to underline regions of space where the communication kernel is guaranteed to be smooth and continuous as a function of $\rho$. Namely, we adopt the following local regularity assumptions:

\newlist{regularity}{enumerate}{1}

\setlist[regularity, 1]
{label=(r\arabic{regularityi}), 
leftmargin=30pt,
}

\begin{regularity}
\item \label{i:r1}  for any $k = 0,1,\dots$ there is a positive function $C_k: (0,\infty) \to (0,\infty)$ such that for all $\rho\in \cP(\O)$ 
\[
 |\p^k_{x,y} \phi_\rho(x,y)|   \leq C_k(\th_\rho(x)), \ \forall (x,y)\in \O \times \O,
\]
\item \label{i:r2} the following continuity-in-$\rho$ holds: for any  $\rho,\rho' \in \cP(\O)$  we have
\[
|\phi_{\rho} (x,y) -\phi_{\rho'}(x,y)| \leq C(\th_{\rho}(x),\th_{\rho'}(x)) W_1( \rho, \rho'),\ \forall (x,y)\in \O \times \O,
\]
for some $C: (0,\infty)\times(0,\infty) \to (0,\infty)$.
\end{regularity}

Note that at those points where thickness vanishes the above assumptions give no information, in other words, $C$ and $C_k$'s can blow up at $0$. Also,  by the lower semicontinuity \ref{i:th1}  thick regions $\th_\rho > 0$ are open, and therefore the regularity conditions hold on open subsets. 

Lastly, we introduce the marginals of $\phi_\rho$ which are called {\em communication strength-functions}. Strength is a measure of  total influence of the flock on an agent $x\in \O$  defined by
\begin{equation}\label{e:s}
\st_\rho(x) = \int_\O \phi_\rho(x,y) \drho(y): \O \to [0,\infty].
\end{equation}
The other marginal  
\begin{equation}\label{e:sadj}
\st^*_\rho(x) = \int_\O \phi_\rho(y,x) \drho(y): \O \to [0,\infty]
\end{equation}
will be called the {\em adjoint strength} and it measures the opposite influence of the agent $x$ on the flock.  
 
A priori both strengths are finite $\rho$-almost everywhere as follows from the integrability condition \eqref{e:Fint}. However, we assume much more than that:

\newlist{strength}{enumerate}{1}

\setlist[strength, 1]
{label=(s\arabic{strengthi}), 
leftmargin=30pt,
}

\begin{strength}
\item \label{i:s1} $\sup_{\rho \in \cP(\O)}\sup_{x\in \O} \st_\rho(x)  = \oS<\infty$;
 \item  \label{i:s2} there exists a universal constant $C>0$  such that 
\begin{equation*}\label{e:s3}
\st^*_\rho(x) \leq C \st_\rho(x), \quad \rho\text{-almost everywhere}.
\end{equation*}
\end{strength}
The latter condition is somewhat less intuitive. Its phenomenological explanation is to rule out  extreme influencers that affect the flock disproportionally stronger than being affected by it.  For example, it rules out a ``dictator" $x$ which affects the flock positively $\st^*_\rho(x)>0$ while being completely insensitive to feedback $ \st_\rho(x)=0$. The condition \ref{i:s2} is automatically satisfied for a class of communication kernels we call {\em conservative} 
\begin{equation}\label{e:consdef}
\st_\rho = \st^*_\rho, \quad \rho\text{-almost everywhere},
\end{equation}
which includes all  {\em symmetric} kernels
\begin{equation}\label{e:symm}
\phi_\rho(x,y) = \phi_\rho(y,x), \quad \rho \otimes \rho \text{-almost everywhere.}
\end{equation}

To fix the terminology we make the following definition.

\begin{definition}\label{}
A family $\cC = \{ \phi_\rho, \th_\rho\}_{\rho \in \cP(\O)}$ satisfying \eqref{e:Fint}, \ref{i:th1}-\ref{i:th4}, \ref{i:r1}-\ref{i:r2}, \ref{i:s1}-\ref{i:s2} is said to define a {\em  regular communication protocol}.
\end{definition}

All the models we consider naturally facilitate communication between nearby agents provided the flock is present at a given location. 
\begin{definition}\label{d:loc}
The communication kernel is {\em locally supported} (or local, for short) if there exists a $r_0, c_0 >0$  such that for all $\rho\in \cP(\O)$ and all $x,y\in \O$, we have
\begin{equation}\label{e:loc}
\phi_\rho(x,y) \geq c_0 \one_{|x-y| <r_0} \one_{\th_\rho >0}(x).
\end{equation}
\end{definition}
Taking the fist marginal we readily obtain a lower bound on the strength in terms of ball-thickness:
\begin{equation}\label{e:stbelow}
\st_\rho(x) \geq c_0 \th_{\rho,r_0}(x)\one_{\th_\rho >0}(x).
\end{equation}

\subsection{Core examples}

Let us highlight a few examples of communication protocols, starting with the classical one that defines the Cucker-Smale model \eqref{e:CS}: for $\phi\in C^\infty(\O)$,
\begin{equation}\label{CS}
\phi_\rho(x,y) = \phi(x-y), \quad \th_\rho = \rho \ast \phi.
\tag{CS}
\end{equation}
The model is clearly symmetric with the strength given by $\st_\rho = \rho \ast \phi$, and regularity conditions  \ref{i:r1}-\ref{i:r2} independent of thickness (we call such protocols {\em uniformly regular}, see \cite{S-EA}). The protocol is local, with $\t = 0$ provided the defining kernel itself is local
\begin{equation}\label{e:locf}
\phi(r) \geq c_0 \one_{|r| <r_0}.
\end{equation}

The Motsch-Tadmor model introduced in \cite{MT2011,MT2014} corresponds to the averaged (or right-stochastic) version of the 
\ref{CS}-protocol\footnote{From here on we adopts a convention  that quotients involving density in the denominator are zero outside the support of the denominator.}, 
\begin{equation}\label{MT}
\phi_\rho(x,y) = \frac{\phi(x-y)}{\rho \ast \phi(x)}, \quad \th_\rho = \rho \ast \phi. \tag{MT}
\end{equation}
Here the strength is given by $\st_\rho(x) = \one_{\rho \ast \phi > 0}(x)$, and in fact $\st_\rho = 1$, $\rho$-almost everywhere on the support of $\rho$. The adjoint strength-function takes form
\[
\st_\rho^* = \left(  \frac{\rho }{\rho \ast \phi }\right) \ast \phi.
\]
The condition \ref{i:s2} becomes somewhat non-trivial to verify. It amounts to establishing 
\begin{equation}\label{e:KMTst*}
 \left(\frac{\rho }{\rho \ast \phi }\right) \ast \phi \leq C,
\end{equation}
on the set where $\rho \ast \phi \neq 0$ (otherwise both sides vanish). The inequality \eqref{e:KMTst*} indeed holds for local kernels \eqref{e:locf} and was first proved in Karper, Mellet, Trivisa \cite{KMT2013} exactly for the purpose of showing that the averaging $\rmw_\rho(u) =\frac{(u\rho)\ast \phi}{\rho \ast \phi}$ (also known as the Favre filtration) is bounded on $L^2(\rho)$. On the torus, it follows from a slightly stronger bound that we will come back to repeatedly in the sequel. Let us record it for future reference.
\begin{lemma}\label{e:KMT} Suppose \eqref{e:locf}.
There is an absolute constant $C$ depending only on $\phi$ such that 
\begin{equation}\label{e:KMT0}
\sup_{\rho\in \cP(\O)} \int_{\O} \frac{\drho(x)}{\rho\ast \phi(x) } < C.
\end{equation}
\end{lemma}
\begin{proof}
By locality \eqref{e:locf}, $\rho\ast \phi(x) \gtrsim \rho(B_{r_0}(x))$. Let us cover $\O$ with $J$ balls, where $J$ depends only on $n$ and $r_0$,  $\O \ss \cup_{j=1}^J B_{r_0/2}(x_j)$. Since for all $x\in B_{r_0/2}(x_j)$, we have $B_{r_0/2}(x_j) \ss B_{r_0}(x)$, and hence (with the standard $0$-convention if the denominators vanish)
\begin{equation}
\int_{\rho \ast \phi > 0 } \frac{\drho(x)}{\rho\ast \phi(x) }  \lesssim \sum_{j=1}^J \int_{B_{r_0/2}(x_j)} \frac{\drho(x)}{  \int_{B_{r_0}(x)} \rho(y)\dy } \leq  \sum_{j=1}^J \int_{B_{r_0/2}(x_j)} \frac{\drho(x)}{  \int_{B_{r_0/2}(x_j)} \rho(y)\dy }  \leq J.
\end{equation} 
\end{proof}

In order to achieve better mapping properties of the models like \ref{MT}, yet still retaining a degree of the \ref{MT}-normalization one can interpolate between the \ref{CS} and \ref{MT} and define a range of protocols, we call $\g$-protocols. Let us fix a parameter $0 \leq \g  \leq 1$, and consider
\begin{equation}\label{Cg}
\phi_\rho(x,y) =  \frac{\phi(x-y)}{(\rho \ast \phi(x))^\g}, \quad \th_\rho = \rho \ast \phi. \tag{$\cC_\g$}
\end{equation}
So, the limiting case $\g =1$ corresponds to Motsch-Tadmor $\cC_1 =$\ \ref{MT}, while $\g=0$ corresponds to Cucker-Smale, $\cC_0 =$\ \ref{CS}.   Just like for the Motsch-Tadmor model, all the properties are verified similarly, except for \ref{i:s2}. With the strength-function given by $\st_\rho =(\rho \ast \phi)^{1-\g}$ and the adjoint strength given by $\st_\rho^* = \left( \frac{\rho}{(\rho \ast \phi)^\g} \right) \ast \phi$, the inequality \ref{i:s2} becomes
\begin{equation}\label{e:KMTb}
 \left( \frac{\rho}{(\rho \ast \phi)^\g} \right) \ast \phi \leq C (\rho \ast \phi)^{1-\g}.
\end{equation}
We refer to \cite{S-EA} for the proof. 

Note that the \ref{Cg}-protocols are local in the sense of \defin{d:loc} provided the defining kernel itself is local \eqref{e:locf}.

Most of technical issues related to analysis of \eqref{e:FPA} based on protocols like \ref{Cg} are attributed to the lack to symmetry (or even conservative property).  There are modifications to these models that retain the same scaling in terms of density while making the kernels symmetric and even positive semi-definite. One such modification was proposed and studied in \cite{S-hypo,S-EA}:
\begin{equation}\label{Cg+}\tag{$\cC_{\g}^+$}
\phi_\rho(x,y) = \int_{\O} \frac{\psi(x-z)\psi(y-z)}{(\rho \ast \psi(z))^\g}\dz, \quad \th_\rho(x)  = \inf_{z \in \supp \psi(x-\cdot)} \rho \ast \psi(z),
\end{equation}
where $\psi\in C^\infty(\O)$ is a local kernel with $\int \psi \dx = 1$. The corresponding communication map is given by the overmollified Favre filtration $\rmw_\rho(u) =\left(\frac{(u\rho)\ast \psi}{(\rho \ast \psi)^\g}\right)\ast \psi$. The strength is given by $\st_\rho = ((\rho \ast \psi)^{1-\g})\ast \psi$, and all the axioms of a local regular protocol are rather obvious. We call such protocols {\em Bochner-positive}, by analogy with the classical Bochner and Mercer's Theorems, because they define positive semi-definite communication operators:
\[
(u , \rmw_\rho(u))_\rho = \int_\O \frac{((u\rho) \ast \psi)^2}{(\rho \ast \psi)^\g} \dx \geq 0.
\]

Another symmetric variant of \ref{Cg} can be defined as follows
\begin{equation}\label{Cg/2}
\phi_\rho(x,y) =  \frac{ \phi(x-y) }{ (\rho \ast \phi(x))^{\frac{\g}{2}} (\rho \ast \phi(y))^{\frac{\g}{2}} }, \quad  \th_\rho(x)  = \inf_{y \in \supp \phi(x-\cdot)} \rho \ast \phi(y). \tag{$\cC_{\frac{\g}{2},\frac{\g}{2}}$}
\end{equation}
Here as always we assume that $\phi$ is local. It implies that first, $x \in \supp \phi(x-\cdot)$ and so we capture all the support of the denominator in the verification of the regularity axioms; and second, that \ref{Cg/2} is a local protocol.

Another distinctly different and non-Galilean invariant protocol was introduced in \cite{S-EA} as a consensus model in segregated neighborhoods.  Consider any smooth partition of unity $g_l \in C^\infty(\O)$, $g_l \geq 0$, and $\sum_{l=1}^L g_l = 1$ subordinated to an open cover $\{ \cO_l \}_{l=1}^L$ of $\O$,  so that $\supp g_l \ss \cO_l$. For a $\rho\in \cP(\O)$ we denote $\rho(g_l) = \int_\O g_l \drho$. One defines the model by   
\begin{equation}\label{Segg}
\phi_\rho(x,y) =  \sum_{l = 1}^L \frac{ g_l(x) g_l(y)}{(\rho(g_l))^\g}, \quad  \th_\rho(x) = \min_{l: x\in \supp g_l} \rho(g_l).
 \tag{Seg$_\g$}
\end{equation}
Here, the strength is given by $ \st_\rho(x) = \sum_{l=1 }^L g_l(x) \rho^{1-\g}(g_l)$ and since the model is symmetric, condition \ref{i:s2} is automatically verified. To check locality, let us observe that for any point $x$ where $\th_\rho(x)>0$, we have $\rho(g_l)>0$ whenever $x\in \supp g_l$. Since the sum in \ref{Segg} is taken only over such indexes $l$, and clearly $\rho(g_l)\leq 1$, we have 
\[
\phi_\rho(x,y) \geq   \sum_{l : x\in \supp g_l} g_l(x) g_l(y)
\]
Since $g_l$ is a partition of unity there is at least one $l$ for which $g_l(x) \geq \frac1L$. By uniform continuity, there is a radius $r_0>0$ such that $g_l(y) > \frac{1}{2L}$ for all $|x-y|<r_0$, which implies \eqref{e:loc} as desired.

\subsection{Mapping properties}

Communication protocols give rise to families of integral operators such as $\rmw_\rho$  defined in \eqref{e:warep} and its averaged version, called in \cite{S-EA} {\em environmental averaging}:
\begin{equation}\label{e:ave}
\ave{u}_\rho(x) =  \frac{\rmw_\rho(u)(x)}{\st_\rho(x)}.
\end{equation}
Both families are important in the study of alignment models, so let us discuss their immediate mapping properties.

Conditions \ref{i:s1}-\ref{i:s2} readily imply that the marginals \eqref{e:s} and \eqref{e:sadj} are uniformly bounded, which puts our kernels $\phi_\rho$ uniformly  into Schur's class. As a consequence of the classical Schur  theory this implies  $L^p$-boundedness:  for all $1\leq p \leq \infty$,
\begin{equation}\label{ }
\sup_{\rho \in \cP(\O)} \| \rmw_\rho\|_{L^p(\rho) \to L^p(\rho)} \leq C \oS.
\end{equation}

The mapping properties of environmental averages, on the other hand, are more naturally associated with the $L^p$-spaces relative to what we call a {\em strength-measure} defined by $\dk_\rho = \st_\rho \drho$, and the corresponding adjoint counterpart, $\dk^*_\rho = \st^*_\rho \drho$. 
\begin{lemma}\label{l:k-to-k*}
For any $1\leq p\leq \infty$, any $\rho \in \cP(\O)$, and $u\in L^p(\k^*_\rho)$ we have
\begin{equation}\label{e:k-to-k*}
\|[u]_\rho\|_{L^p(\k_\rho)} \leq \|u\|_{L^p(\k^*_\rho)}.
\end{equation}
\end{lemma}
\begin{proof}
Indeed, for $p=\infty$ it is obvious from the defining formula \eqref{e:ave} itself. For $p<\infty$, by  the \HI,
\begin{equation*}\label{}
\begin{split}
\| [u]_\rho \|_{L^p(\k_\rho)}^p =\int_\O \left| \int_\O u(y) \frac{\phi_\rho(x,y) \drho(y)}{\st_\rho(x)} \right|^p \st_\rho(x) \drho(x) & \leq \int_\O\int_\O |u(y)|^p \frac{\phi_\rho(x,y) \drho(y)}{\st_\rho(x)} \st_\rho(x) \drho(x) \\
& = \int_\O\int_\O |u(y)|^p \phi_\rho(x,y) \drho(y) \drho(x)\\
& = \int_\O\int_\O |u(y)|^p \st^*_\rho(y) \drho(y) = \| u\|_{ L^p(\k^*_\rho)}^p.
\end{split}
\end{equation*}
\end{proof}

In view of \ref{i:s2}, \eqref{e:k-to-k*}  implies that the averages map $L^p(\k_\rho) \to L^p(\k_\rho)$ uniformly,
\begin{equation}\label{e:LpLp}
\sup_{\rho \in \cP(\O)} \| \ave{u}_\rho\|_{L^p(\k_\rho)} \leq C \|u\|_{L^p(\k_\rho)}.
\end{equation}
Clearly, for conservative models, $C=1$, which implies the contraction mapping 
\begin{equation}\label{e:contraction}
\sup_{\rho \in \cP(\O)} \| \ave{u}_\rho\|_{L^p(\k_\rho)} \leq \|u\|_{L^p(\k_\rho)}.
\end{equation}
This has multiple fundamental implications which we discuss in next section.  

Coming back to the communication operator $\rmw_\rho$, we distinguish a special of class of protocols for which $\rmw_\rho(u)$ improves regularity from $L^2(\rho)$ to uniformly bounded, which is going to play a major role in the relaxation. 

\begin{definition}\label{d:type2}
We say that a protocol $\cC$ is of mapping type $(2,\infty)$ if there exists a $\oW>0$, such that 
\begin{equation}\label{e:type2}
\sup_{\rho \in \cP} \sup_{\|u\|_{L^2(\rho)} \leq 1} \sup_{x\in \O} |\rmw_\rho(u)(x)| \leq \oW.
\end{equation}
\end{definition}

Among $\g$-protocols, \ref{Cg} or \ref{Cg+}, those of type-$(2,\infty)$ correspond to $0 \leq \g \leq \frac12$. Indeed, it comes down to the estimate 
\[
\left| \frac{(u\rho)\ast \phi}{(\rho \ast \phi)^\g} \right| \leq \frac{((|u|^2\rho)\ast \phi)^{1/2} (\rho \ast \phi)^{1/2}}{(\rho \ast \phi)^\g} \leq \|u\|_2 (\rho \ast \phi)^{\frac12 - \g} \lesssim \|u\|_2.
\]
The computation is entirely similar for \ref{Segg}. For the \ref{Cg/2}-protocol, we have
\[
\rmw_\rho(u) = \frac{1}{ (\rho \ast \phi)^{\g/2}} \left( \frac{ u\rho }{ (\rho \ast \phi)^{\g/2} } \right) \ast \phi \leq \frac{\|u\|_2}{ (\rho \ast \phi)^{\g/2}} \sqrt{ \frac{\rho }{ (\rho \ast \phi)^{\g} }  \ast \phi },
\]
using \eqref{e:KMTb},
\[
\lesssim \|u\|_2 (\rho \ast \phi)^{\frac12 - \g} \lesssim \|u\|_2,
\]
provided $\g \leq \frac12$. Let us record our findings for future reference.

\begin{lemma}\label{ }
All the protocols  \ref{Cg}, \ref{Cg+}, \ref{Cg/2}, \ref{Segg} are of type-$(2,\infty)$ provided $0 \leq \g \leq \frac12$.
\end{lemma}

\subsection{Momentum, entropy, and $H$-theorem}\label{ss:en} 

For a probability density $f$ solving \eqref{e:FPA}, let us denote  the total momentum by $\bar{u} = \int_\domain v f(t,x,v)\dv\dx$. Generally $\bar{u}$ is a function of time solving
\begin{equation}\label{e:momentum}
\ddt \bar{u} = \int_\O (\st_\rho^*(x) - \st_\rho(x)) u(x) \rho(x)\dx.
\end{equation}
We can see that for conservative protocols \eqref{e:consdef}  the  momentum is conserved, $\bar{u}=\bar{u}_0$. 

Formally, the entropy of $f$ defined in \eqref{e:entropy} satisfies the equation
\begin{equation}\label{e:entropylaw}
\ddt  \cH = - \int_\domain \st_\rho \frac{|\s_0 \n_v f + v f |^2}{f} \dv\dx   +  (u,\ave{u}_\rho)_{\k_\rho}.
\end{equation}
It highlights two competing mechanisms of entropy depletion / production, where the first term -- Fisher information -- relaxes the solution towards a Maxwellian centered at a hydrostatic state $(\rho, \bar{u} = 0)$, while the alignment production term inhibits such relaxation.  This is clearly not a dissipative law for the entropy. If we rewrite it differently by incorporating macroscopic shift in the Fisher term, then the two mechanisms start to agree on the direction, at least for conservative models. Indeed, \eqref{e:entropylaw} is equivalent to  
\begin{equation}\label{e:elaw1}
\ddt \cH = -  \int_{\domain}   \st_\rho \frac{|\s_0 \n_v f + (v - u) f |^2}{f}  \dv\dx  - \|u\|_{L^2(\k_\rho)}^2 + (u,\ave{u}_\rho)_{\k_\rho}.
\end{equation}
If the protocol $\cC$ is conservative, then by the contraction property \eqref{e:contraction},  $ - \|u\|_{L^2(\k_\rho)}^2 + (u,\ave{u}_\rho)_{\k_\rho} \leq 0$, and we recover the classical  $H$-theorem.
\begin{lemma}[$H$-theorem]\label{e:H}
If $\cC$ is a conservative protocol, then the entropy $\cH$ is a Lyapunov function of the equation \eqref{e:FPA},
\begin{equation}\label{ }
\ddt \cH \leq 0.
\end{equation}
\end{lemma}

\subsection{Existence of weak solutions}
Let us  state our first main result -- global existence of weak solutions.
\begin{definition}\label{d:weak} 
We say that 
\begin{equation}\label{d:regclass}
\begin{split}
f  \in L^\infty([0,T);L^1_2) \cap C([0,T); \cD') 
\end{split}
\end{equation}
 is a weak solution to \eqref{e:FPA} with initial condition $f_0$ if for every $\f \in C_0^\infty([0,T) \times \O \times \R^n)$ one has
\begin{equation}\label{e:weak}
\begin{split}
 \int_{\domain} \f(t,x,v) f(t,x,v) \dv \dx - \int_0^t\int_{\domain} \p_t \f f \dv \dx \ds -  \int_{\domain} \f(0,x,v) f_0 \dv \dx \\
  - \int_0^t\int_{\domain} v \cdot \n_x \f  f \dv \dx \ds 
 = \int_0^t\int_{\domain} \s f \D_v \f \dv \dx \ds  + \int_0^t\int_{\domain} \cA(f) f  \cdot\n_v \f \dv \dx \ds.
\end{split}
\end{equation}
\end{definition}

To see that the  last term makes distributional sense, let us observe 
\begin{equation}\label{e:A}
\cA (f)  =   \rmw_\rho(u)  -  v \st_\rho.
\end{equation}
Clearly, $v \st_\rho f \in L^1$, while $\rmw_\rho(u) f = \rmw_\rho(u) \sqrt{f} \sqrt{f} \in L^2 \cdot L^2 \ss L^1$. So, the basic energy class is sufficient for all the terms in \eqref{e:weak} to be finite, however, for the purposes of obtaining weak solutions, specifically to enable the application of the Averaging Lemma,  further assumptions on initial data will be necessary such as finite higher moments and boundedness.

\begin{theorem}\label{t:weak}
Let $\cC$ be a local regular protocol.

\underline{Existence}: For any initial condition $f_0 \in L^\infty \cap L^1_q$, $q \geq 2$,  there exists a global weak solution to \eqref{e:FPA} in the class 
\begin{equation}\label{e:regweak}
\begin{split}
f  \in L^\infty([0,T);L^\infty \cap L^1_q) \cap C([0,T); \cD'), \\
 \st_\rho |\n_v f|^2 f^{p-2} \jap{v}^{q'} \in L^1([0,T) \times \domain),
 \end{split}
\end{equation}
 for any  $T>0$, where $q' = q$ if $p>1$, and $q'<q$ if $p=1$.  
 
\underline{Entropy}: Any weak solution in the above class has absolutely continuous entropy  in time $t\to \cH(t)$ satisfying the entropy equality \eqref{e:entropylaw} on $[0,\infty)$.

\underline{Renormalization}: any weak solution is automatically renormalized in the sense of DiPerna-Lions, see \prop{p:renorm}.

\underline{Uniqueness}: Let $f',f''$ be two weak solutions in class \eqref{e:regweak} with, $q \geq n+4$, starting from same non-vacuous data $\min \rho_0 >0$. Then $f' = f''$ on some interval $[0,t_0)$ and for as long as they both remain non-vacuous. 
\end{theorem}

\sect{s:exist} will be entirely devoted to the existence and uniqueness statements of the theorem. The renormalization is  established in \prop{p:renorm} in the classical sense and in broader sense to include Boltzmann entropy in \prop{p:b-law}. In particular, this implies the entropy law \eqref{e:entropylaw}.

\subsection{Gain of positivity and hypoellipticity}  The next series of results will detail properties of solutions constructed in  \thm{t:weak} under some additional assumptions on the protocol $\cC$. First, as a consequence of the hypoelliptic structure of the Fokker-Planck operator, solutions to \eqref{e:FPA} are expected to gain regularity instantly. In our situation this statement stumbles upon two issues -- a possibly degenerate diffusion coefficient $\st_\rho$, and the lack of regularity of the drift $\rmw_\rho$, which a priori belongs only to $L^2(\rho)$. For a local regular communication protocol $\cC$, both issues are resolved if the flock is known to be  thick $\th_\rho(\O)>0$ as readily follows from the definitions of \sect{ss:cp}. One way to achieve thickness is to establish spread of the support of $f$ that fills the entire domain $\domain$. Such a spread can be established through a gain of Gaussian tails at $v$-infinity and is a known effect for many kinetic models, see \cite{AZ2024,DesVill2000,FranPoli2006,GIMV2019,HST2020,Hormander1967,IMS2020,Kolm1934,Mouhot2005}. For flocking models of type \eqref{e:FPA} the result was established for classical solutions in \cite[Proposition 7.3]{S-EA}, and  we will extend it to the class of weak solutions constructed in \thm{t:weak}. 

\begin{proposition}\label{p:Gauss} Let $\cC$ be a local regular conservative protocol  of type $(2,\infty)$. 
Then for any given weak solution $f$ in class \eqref{e:regweak} and for any $t_0>0$, there exist constants $a,b>0$, which depend on the initial entropy $\cH_0$, $t_0$, and other fixed parameters of the model, such that
\begin{equation}\label{e:Gauss}
f(t,x,v) \geq b e^{-a |v|^2} , \qquad \forall (t,x,v) \in [t_0,\infty) \times \O \times\R^n .
\end{equation}
\end{proposition}

The main technical load of this proposition is already contained in the aforementioned \cite[Proposition 7.3]{S-EA}. However, in order to relax it to the class of weak solutions new tools are necessary including the renormalization and the Weak Comparison Principle we establish in \prop{p:renorm} and \cor{c:wmp}. Details will be presented in \sect{ss:propGauss}.

Clearly the gain of positivity implies that $\min \rho(t) >c_0$ for any $t>t_0$ where $c_0$ is independent of time.  In view of 
\ref{i:th3}, the flock becomes instantaneously thick, which implies uniform in time spacial regularity of the coefficients of \eqref{e:FPA} as well as uniform ellipticity 
\begin{equation}\label{e:unifell}
\s >c_1,
\end{equation}
due to locality \eqref{e:stbelow}.  This puts the model \eqref{e:FPA} within reach of global hypoellipticity methods, see \cite{Villani-hypo}, which we properly adapt in \sect{s:hypoell} to  drift-diffusion  models such as \eqref{e:FPA}. The following is a direct consequence of \prop{p:reg} and \prop{p:Gauss}.
\begin{proposition}\label{p:instantreg}
 Let $\cC$ be a local regular conservative protocol  of type $(2,\infty)$. 
 Then for any given weak solution $f$ in class \eqref{e:regweak} and for any $T>t_0>0$, and $m\in \N$ there exists a constant $C_{t_0,T,m}>0$ such that 
\begin{equation}\label{e:mqinstant}
\|f(t)\|_{H^m_q} + \|\p_t f(t)\|_{H^{m-1}_{q-3}} \leq C_{t_0,T,m}, \quad \forall t\in (t_0,T).
\end{equation}
\end{proposition}

\subsection{Relaxation for large data}  Let us address now the question of long time behavior when in addition to locality and type $(2,\infty)$ we assume that the protocol $\cC$ is symmetric \eqref{e:symm}.  The important feature of the estimate \eqref{e:Gauss} is that the coefficients $a,b$ emerge independent of time for $t>1$. As we noted above this implies that the flock gets uniformly thick with regard to any thickness, $\th_\rho$, including the native one associated with the protocol $\cC$ as well as the ball-thickness \eqref{e:ballth}. Consequently, according to the regularity assumption \ref{i:r1} the kernel $\phi_\rho$ becomes uniformly smooth for all $t>1$ which implies that the communication operator $\rmw_\rho$ is regularizing
\begin{equation}\label{e:H1L2}
\| \n \rmw_\rho(u) \|_{L^2(\rho)} \leq c_2 \| u\|_{L^2(\rho)}, \quad t>1.
\end{equation}
Due to locality  \eqref{e:stbelow}, we also have  \eqref{e:unifell} uniformly for all $t>1$. For $q>n+4$ the Sobolev norms  $H^m_q$ control Fisher information, according \cite[Lemma 1]{TV2000},
\begin{equation*}\label{}
\begin{split}
\cI & = \cI_{vv} + \cI_{xx}, \\
\cI_{vv} &=  \int_{\domain}  \frac{|\s_0 \n_v f + v f |^2}{f}  \dv\dx, \quad \cI_{xx} = \s_0 \int_\domain \frac{|\n_x f|^2}{f}  \dv \dx
\end{split}
\end{equation*}
Therefore, for $t>1$ we can treat the solution as classical and perform the hypocoercivity analysis of \cite[Proposition 8.1]{S-EA} to obtain relaxation to the global Maxwellian \eqref{e:Max}. In order to apply that result three conditions must be met uniformly in time:  the regularization of the communication map \eqref{e:H1L2}, the uniform bound from the below on diffusion  \eqref{e:unifell}, and lastly, a uniform in time spectral gap condition
\begin{equation}\label{e:sgap}
 \|u\|_{L^2(\k_\rho)}^2 - (u,\ave{u}_\rho)_{\k_\rho} \geq \e_0 \|u\|_{L^2(\k_\rho)}^2,
\end{equation}
for all $t>1$ and some $\e_0 >0$. The latter implies that the entropy depletion due to the alignment, given in \eqref{e:elaw1} takes a coercive form,
which combined with \eqref{e:entropylaw} implies
\begin{equation}\label{e:HIvv}
\ddt  \cH(f | \mu_{\s_0,\bar{u}}) \lesssim  - \cI_{vv}  -  \|u\|_{L^2(\k_\rho)}^2,
\end{equation}
where 
\[
\cH(f | \mu_{\s_0,\bar{u}}) = \int_\domain f \log \frac{f}{\mu_{\s_0,\bar{u}}} \dv\dx
\]
is the relative entropy -- a functional that  differs from $\cH$ by a constant. 
Since $ \cI_{vv} $ is not a full Fisher information it does not control the relative entropy $ \cH(f | \mu_{\s_0,\bar{u}})$ as is the case with the full information via the log-Sobolev inequality
\begin{equation}\label{}
\cH(f | \mu_{\s_0,\bar{u}}) \leq C \cI.
\end{equation}
So, in order to close the entropic inequality \eqref{e:HIvv} one has to complement it with a series of equations on each component of the Fisher information including the mixed one:
\[
\cI_{xv} =  \int_{\domain}  \frac{(\s_0 \n_v f + v f ) \cdot \n_x f}{f}  \dv\dx.
\]
Thanks to special cancellations that occur in the alignment force one can control all the residual terms in the computation and obtain a differential inequality of type
\begin{equation}\label{e:HILyap}
\ddt  [ \cH(f | \mu_{\s_0,\bar{u}}) + a \cI_{vv} + b \cI_{xv} + c \cI_{xx} ] \lesssim - \cI,
\end{equation}
for an appropriate choice of $a,b,c>0$ in dependent of time, we refer to \cite{S-EA} for details. This implies exponential relaxation. Since the relative entropy controls the $L^1$-norm by the\CK 
\[
\|f(t) - \mu_{\s_0,\bar{u}}\|^2_{L^1(\domain)}  \lesssim \cH(f | \mu_{\s_0,\bar{u}}),
\]
the convergence in $L^1$ metric follows as well. So, we obtain
\begin{equation}\label{e:relaxexp}
\|f(t) - \mu_{\s_0,\bar{u}}\|_{L^1(\domain)} \leq c_3 \sqrt{\cI(f_0)}\ e^{- c_4 t},
\end{equation}
where $c_3,c_4>0$ depend on the initial entropy $\cH_0$ (as defined in \eqref{e:entropy})  as well as fixed parameters of the model such as $\s_0, C_k, \oS, \oW$, etc. 

It remains to verify the spectral gap condition \eqref{e:sgap}. In the following lemma we obtain an estimate $\e_0$ in terms of minimum of the density. 
\begin{lemma}\label{l:spgap}
Suppose $\cC$ is a local regular symmetric protocol. Then
\begin{equation}\label{e:spgapsym}
\e_0 \geq c_5 (\min \rho )^4,
\end{equation}
for some $c_5>0$ depending only on the fixed parameters of the protocol. 
\end{lemma}
The proof also gives an estimate in terms of ball-thickness $\th_{\rho,r_1}$, for some $r_1>0$, depending on other parameters of the model, but it will not be needed thanks to the already available uniform no-vacuum condition. 

\begin{proof} We may assume that   $\th_\rho(\O) >0$, for otherwise $\min \rho = 0$ by \ref{i:th2}, and there is nothing to prove. In this case we have the locality bound on the kernel \eqref{e:loc} available on the entire domain $\O$. So,
using symmetry and \eqref{e:loc},  we have
\begin{equation*}\label{}
\begin{split}
( u, u)_{\kappa_\rho} - ( u , [u]_\rho)_{\kappa_\rho}&  = \frac12  \int_{\O\times \O} \phi_\rho(x,y) | u(y) - u(x) |^2 \rho(y) \rho(x) \dy \dx \\
& \gtrsim   \int_{|x-y| <r_0}  | u(y) - u(x) |^2 \rho(y) \rho(x) \dy \dx.
\end{split}
\end{equation*}
Let us fix a local radially symmetric kernel $\psi \geq 0$ supported on the ball of radius $r_0/2$. Then $\psi \ast \psi \leq \one_{|r| <r_0}$. Consequently, we obtain
\[
( u, u)_{\kappa_\rho} - ( u , [u]_\rho)_{\kappa_\rho}  \gtrsim \frac12  \int_{\O\times \O}\psi \ast \psi(x-y)  | u(y) - u(x) |^2 \rho(y) \rho(x) \dy \dx.
\]
The latter is equal to the difference of inner products $ ( u, u)_{\kappa_\rho} - ( u , [u]_\rho)_{\kappa_\rho}$ corresponding to the Bochner-positive Cucker-Smale model \ref{CS} with kernel $\phi = \psi \ast \psi$. According to \cite[Proposition 4.16]{S-EA} the spectral gap for such protocols is estimated at $\e \sim \th_{\rho,r'}^3$ for some $r'>0$, which in turn is bounded from below by $\sim (\min \rho)^3$. So, we continue
\[
( u, u)_{\kappa_\rho} - ( u , [u]_\rho)_{\kappa_\rho}  \geq c (\min \rho)^3 \int_\O  |u(x)|^2  \rho \ast \psi \ast \psi(x) \rho(x) \dx\geq c (\min \rho)^4 \| u\|^2_{L^2(\rho)} \geq \frac{c} {\oS^{2}} (\min \rho)^4 \| u\|^2_{L^2(\k_\rho)}.
\]
\end{proof}

We now have all the ingredients for the direct application of \cite[Proposition 8.1]{S-EA} to state the general relaxation result for large data.

\begin{proposition}\label{p:relax}
Let $\cC$ be a local regular symmetric protocol of type $(2,\infty)$. Then any weak solution $f$ to \eqref{e:FPA} in class \eqref{e:regweak}, for $q > n+4$, relaxes to the global Maxwellian exponentially fast
\begin{equation}\label{}
\|f(t) - \mu_{\s_0,\bar{u}}\|_{L^1(\domain)} \leq c \sqrt{\cI(f_0)}\ e^{- c t},
\end{equation}
for some $c>0$ depending on the initial entropy $\cH_0$ and the fixed parameters of the protocol.
\end{proposition}

As we noted in the introduction, these results imply all the statements of \thm{t:mainrelaxCS} as it pertains to the Cucker-Smale protocol. But the exact same statement applies to cover all symmetric local protocols of type $(2,\infty)$, which include  \ref{Cg+}, \ref{Cg/2}, and \ref{Segg} for $0 \leq \g \leq \frac12$, i.e. half-way between the Cucker-Smale and Motsch-Tadmor scaling.

\section{Construction of Weak solutions}\label{s:exist}
This section is devoted to proving the existence and uniqueness components of \thm{t:weak}. From here on we assume that $\s_0 =1$ for notational brevity. 

We also show that every weak solution satisfies the energy equality. Namely, the energy $\cE(t)$, see \eqref{e:energy},  is an absolutely continuous in time function satisfying
\begin{equation}\label{e:eneq}
\ddt \cE = -  \int_{\domain} \st_\rho |v|^2 f \dv \dx + n  \int_{\domain} \st_\rho f  \dv \dx + (\ave{u}_\rho, u)_{\k_\rho}. 
\end{equation}
 This will one part of the entropy law \eqref{e:entropylaw} that will be fully established after \prop{p:b-law}.

\subsection{A priori estimates} Let us assume that we have a classical smooth solution $f$ to \eqref{e:FPA}. Any such solution has an a priori bounded $L^\infty$-norm. Indeed, evaluated at a point of maximum $(x,v)$, we have
$\n_v (f \cA(f)) = f \n_v \cA(f) = n f \st_\rho$, which is bounded by $n \oS \|f\|_\infty$. As to the dissipation, at a point of maximum, $\st_\rho \D_v f \leq c \|f\|_\infty$. So, 
\[
\ddt \|f\|_\infty \leq c \|f\|_\infty.
\]

Now let us consider the moments
\[
j_{p,q} = \int_\domain f^p \jap{v}^q \dv\dx.
\]

Writing the equation for the basic energy moment $j_{1,2}$ we easily obtain the following inequality 
\[
\ddt j_{1,2} \leq j_{1,2} + \jap{\bar{u}}^2.
\]
Let us  pair it with the momentum equation \eqref{e:momentum}, which thanks to \ref{i:s1}-\ref{i:s2} satisfies the bound 
\begin{equation}\label{ }
\ddt \bar{u} = \int_{\O} (\st^*_\rho(x) - \st_\rho(x)) u(t,x) \rho(t,x) \dx \leq C \sqrt{j_{1,2}}.
\end{equation}
 So,
\[
\ddt |\bar{u} |^2 \lesssim j_{1,2} + \jap{\bar{u}}^2.
\]
Putting the two together,
\[
\ddt (j_{1,2} + |\bar{u} |^2) \lesssim j_{1,2} + |\bar{u}|^2 + C,
\]
which makes both quantities bounded on for any finite time interval $[0,T)$:
\begin{equation}\label{e:j12}
j_{1,2} + |\bar{u} |^2 \leq C_T, \quad t<T.
\end{equation}

Let us now examine higher moments by testing the equation with $f^{p-1} \jap{v}^q$. Every estimate below is done up to a constant multiple.  So, for the time derivative we get $\ddt j_{p,q}$ as expected, and the free transport term vanishes. The dissipation term becomes
\begin{equation}\label{e:dissapriori}
\begin{split}
&\int_\domain \st_\rho \D_v f f^{p-1} \jap{v}^q \dv\dx  = - \int_\domain \st_\rho \n_v f \n_v(f^{p-1} \jap{v}^q) \dv\dx \\
& =  - (p-1) \int_\domain \st_\rho |\n_v f|^2 f^{p-2} \jap{v}^q \dv\dx -  \int_\domain  \st_\rho \n_v f f^{p-1} \n_v(\jap{v}^q) \dv\dx\\
& = - (p-1) \int_\domain \st_\rho |\n_v f|^2 f^{p-2} \jap{v}^q \dv\dx+ \frac{1}{p} \int_\domain  \st_\rho  f^{p} \D_v\jap{v}^q\dv\dx\\
& \lesssim -  (p-1) \int_\domain \s |\n_v f|^2 f^{p-2} \jap{v}^q \dv\dx + j_{p,q}.
\end{split}
\end{equation}

For the alignment term, using the representation \eqref{e:A}, we have
\begin{equation*}\label{}
\begin{split}
&\int_\domain f \cA(f) \n_v(f^{p-1} \jap{v}^q) \dv\dx\\
 & = (p-1) \int_\domain f (\rmw_\rho(u) - v \st_\rho ) \n_vf  f^{p-2} \jap{v}^q \dv\dx  + \int_\domain f (\rmw_\rho(u) - v \st_\rho )  f^{p-1} \n_v (\jap{v}^q) \dv\dx \\
& = \frac{p-1}{p} \int_\domain  (\rmw_\rho(u) - v \st_\rho ) \n_v(f^p) \jap{v}^q \dv\dx +  \int_\domain (\rmw_\rho(u) - v \st_\rho )  f^p \n_v (\jap{v}^q) \dv\dx\\
& = - \frac{p-1}{p} \int_\domain \n_v (\rmw_\rho(u) - v \st_\rho ) f^p \jap{v}^q \dv\dx +\int_\domain  (\rmw_\rho(u) - v \st_\rho ) f^p \n_v(\jap{v}^q) \dv\dx \\
& =  n \frac{p-1}{p} \int_\domain  \st_\rho  f^p \jap{v}^q \dv\dx +\int_\domain  (\rmw_\rho(u) - v \st_\rho ) f^p \n_v(\jap{v}^q) \dv\dx \\
& \lesssim  j_{p,q} + q \int_\domain  |\rmw_\rho(u)| f^p \jap{v}^{q-1} \dv\dx.
\end{split}
\end{equation*}

If $p=1$,  we obtain 
\begin{equation*}\label{}
\begin{split}
\int_\domain  |\rmw_\rho(u)| f \jap{v}^{q-1} \dv\dx& \leq \left( \int_\domain  |\rmw_\rho(u)|^q f \dv\dx \right)^{1/q} \left( \int_\domain   f \jap{v}^{q} \dv\dx \right)^{1/q^*}  \\
& \leq j_{1,q}^{1/q^*}  \| \rmw_\rho(u) \|_{L^q(\rho)} \leq c  j_{1,q}^{1/q^*}  \|u  \|_{L^q(\rho)} \leq c  j_{1,q}^{1/q^*}  j_{1,q}^{1/q} = c  j_{1,q}. 
\end{split}
\end{equation*}
In this case all the estimates close on the momentum, and we obtain
\[
\ddt j_{1,q} \leq c_1 j_{1,q},
\]
which implies a priori  bound on $ j_{1,q}$ if the initial momentum is finite. However, in this case we lose the dissipation term from \eqref{e:dissapriori} entirely. So,  coming back to the $j_{p,q}$-momentum, for $p>1$, and with the a priori bounds on $\|f\|_\infty$ and $j_{1,q}$ at hand, we have 
\begin{equation}\label{e:wpq}
\int_\domain  |\rmw_\rho(u)| f^p \jap{v}^{q-1} \dv\dx \leq \|f\|_\infty^{p-1}  j_{1,q} <C.
\end{equation}
The estimates close on the $j_{p,q}$-momentum,
\[
\ddt j_{p,q} \leq c_1 j_{p,q} + c_2 - c_3 \int_\domain  \st_\rho |\n_v f|^2 f^{p-2} \jap{v}^q \dv\dx,
\]
and we obtain  another  a priori bound on the weighted Fisher information,
\begin{equation}\label{ }
\int_0^T \int_\domain  \st_\rho |\n_v f|^2 f^{p-2} \jap{v}^q \dv\dx <C.
\end{equation}

To recover the classical Fisher information corresponding  to $p=1$ we use the entropy equation. But first let us establish  a general bound on weighted $L\log L$-norms. 

\begin{lemma}\label{l:LlogL} If $f\in L^1_q \cap L^p$, $p>1$, then for any 
$1\leq r <\infty$ and $0\leq \g < q$, we have 
\begin{equation}\label{ }
\int_\domain \jap{v}^\g f |\log f|^r \dv\dx \leq c'_{\g,p,q,r} + c''_{\g,p,q,r} \|f\|_{L^1_q} + c'''_{\g,p,q,r} \|f\|_{L^p}.
\end{equation}
\end{lemma}
\begin{proof}
Indeed, for $\epsilon = (p-1) \frac{\g}{q}$, we have
\[
\jap{v}^\g \one_{f \geq 1} f \log^r f   \lesssim \jap{v}^\g f^{1+ \epsilon} = \jap{v}^q f+ f^{(1+\epsilon - \frac{q-\g}{q}) \frac{q}{\g}} =  \jap{v}^q f + f^{p},
\]
while
\begin{multline*}
\jap{v}^\g \one_{f \leq 1}  f (-\log f)^r  = \jap{v}^\g \one_{f \leq \exp(-\jap{v}^{\frac{q-\g}{r}})} f (-\log f)^r + \jap{v}^\g \one_{\exp(-\jap{v}^{\frac{q-\g}{r}}) < f \leq 1} f (-\log f)^r \\
  \leq c \jap{v}^\g \one_{f \leq \exp(-\jap{v}^{\frac{q-\g}{r}})} \sqrt{f} +   f \jap{v}^q   \leq  \jap{v}^\g  \exp(-\jap{v}^{\frac{q-\g}{2r}}) +   f \jap{v}^q.
\end{multline*}
\end{proof}

Testing \eqref{e:FPA} with $(\log f + 1)\jap{v}^{q'}$, for $q'<q$,
\begin{equation*}\label{}
\begin{split}
\ddt \int_\domain f \log f \jap{v}^{q'} \dv \dx = &-  \int_\domain \st_\rho \frac{|\n_v f|^2}{f} \jap{v}^{q'} \dv\dx  - \int_\domain  \st_\rho \n_v f \n_v \jap{v}^{q'} (\log f +1)\dv\dx\\
& + \int_\domain \cA(f) \n_v f \jap{v}^{q'} \dv\dx .
\end{split}
\end{equation*}
We have by \lem{l:LlogL},
\begin{equation*}\label{}
\begin{split}
- \int_\domain \st_\rho \n_v f \n_v \jap{v}^{q'} (\log f +1)\dv\dx & = \int_\domain \st_\rho f  \D_v \jap{v}^{q'} (\log f +1)\dv\dx + \int_\domain \st_\rho   \n_v \jap{v}^{q'} \n_v f \dv\dx \\
& \lesssim j_{1,q} +  \int_\domain f |\log f| \jap{v}^{q'} \dv\dx - \int_\domain  \st_\rho   \D_v \jap{v}^{q'}  f \dv\dx\\
& \lesssim j_{1,q} + \|f\|_p \lesssim j_{1,q} + \|f\|_1 + \|f\|_\infty \leq C.
\end{split}
\end{equation*}

As to the alignment term,  we obtain
\[
\int_\domain \cA(f) \n_v f \jap{v}^{q'} \dv\dx = -\int_\domain \rmw_\rho(u) \n_v  \jap{v}^{q'} f\dv\dx + \int_\domain \st_\rho f \n_v(v  \jap{v}^{q'})\dv\dx ,
\]
and since $| \n_v  \jap{v}^{q'}| \leq \jap{v}^{q-1}$, by \eqref{e:wpq}, this expression is bounded.
So, we obtain the following equation on the weighted entropy
\[
\ddt \int_\domain f \log f \jap{v}^{q'} \dv \dx + \int_\domain \st_\rho \frac{|\n_v f|^2}{f} \jap{v}^{q'} \dv\dx \leq C.
\]
Again, by \lem{l:LlogL}, $\int_\domain f \log f \jap{v}^{q'} \dv \dx\in L^\infty_t$, so integrating the above we conclude 
\begin{equation}\label{e:Fisherapriori}
\int_0^T  \int_\domain \st_\rho \frac{|\n_v f|^2}{f}  \jap{v}^{q'}\dv\dx \dt \leq C \|f\|_{L^\infty([0,T); L^1_q \cap L^\infty)}+ C.
\end{equation}

Let us now turn to a priori estimates on the forces to make them ready for the application of the averaging lemma. Since  $q >2$,
\begin{align} 
\int_\domain |\st_\rho f|^2 \dv\dx & \lesssim \int_\domain f^2  \dv\dx  = \|f\|_\infty j_{1,q}   \in L^\infty_t, \label{e:fdiss}\\
\int_\domain |f \cA(f)|^2 \dv\dx& \lesssim \int_\domain (|v|^2 + |\rmw_\rho(u)|^2) f^2 \dv\dx  \lesssim \|f\|_\infty (j_{1,2} + \|\rmw_\rho(u)\|_2^2) \label{e:fA} \\
& \lesssim  \|f\|_\infty (j_{1,2} + j_{1,2}) \in L^\infty_t \notag.
\end{align}
Let us summarize the obtained estimates for future reference.
\begin{lemma}\label{l:apriori} Solutions to the Fokker-Planck-Alignment equation \eqref{e:FPA} are a priori bounded in $L^1_q \cap L^\infty$ if initially so, and $ \s |\n_v f|^2 f^{p-2} \jap{v}^{q'} \in L^1_{t,x,v}$, for all $1 < p < \infty$ and $q'=q$, and $p=1$ with $q'<q$. We also have all the forces in $L^\infty_t L^2_{x,v}$.
\end{lemma}

Let us now recall a version of the averaging lemma that will be suitable for our purposes.
\begin{lemma}[\cite{Glassey}] \label{l:al}
Suppose $f$ solves to the following equation on $\O \times \R^n \times \R$
\begin{equation}\label{ }
\p_t f + v \cdot \n_x f = \sum_{l=0}^L \n_v^l G_l,
\end{equation}
where $f, G_l\in L^2_{t,x,v}$. Let $\f \in C^\infty_0( B_R)$, the macroscopic quantity
\[
\rho_{\f} (t,x) = \int_{\R^{n}} \f(v) f(t,x,v) \dv
\]
belongs to $H^{1/6}_{t,x}$ with
\[
\|  \rho_{\f} \|_{H^{1/6}_{t,x}} \leq C_R ( \|  f \|_{L^2_{t,x,v}} + \sum_l \| G_l \|_{L^2_{t,x,v}}).
\]
\end{lemma}

\subsection{Approximation scheme}

We now introduce a proper approximation scheme that  respects all the a priori estimates stated in the previous section, and at the same time is classically well-posed in higher regularity Sobolev spaces.

First we consider a smooth truncation of the mollified initial condition, for example
\[
f_0^\e(x,v) = \chi(x,\e v) f_0 \ast \chi_\e (x,v),
\]
where $\chi_\e$ is any standard smooth mollifier. 
Note that $\| f_0^\e \|_{L^p_q} \leq \| f_0 \|_{L^p_q}$, for any $p,q$, which means that $f_0^\e$ belong uniformly to those classes where $f_0$ initially belongs, and $f^\e_0 \in H^{M}_Q(\domain)$ for any $M,Q\in \N$.

Next, we remove vacuum from the density as a parameter of the kernel by elevating it slightly,
\[
\bar{\rho}^{\e}  = \frac{\rho^\e + \e}{1+\e} \in \cP(\O).
\]
We use this elevated density to approximate the kernel and corresponding operators
\[
\rmw^\e   = \int_\O \phi_{\bar{\rho}^{\e}}(x,y) u^\e(y) \rho^\e(y) \dy.
\]
Note that we only use the elevated density in the kernel, but not in the momentum $u^\e \rho^\e$ for the quantity $u^\e \bar{\rho}^\e$ does not make a distributional sense.  

Let us consider the following approximation
\begin{equation}\label{e:FPAee}
\p_t f^\e + v\cdot \n_x f^\e + \n_v \cdot (  f^\e (\rmw^\e - v \st_{\bar{\rho}^\e}) )  = \st_{\bar{\rho}^\e} \D_v f^\e.
\end{equation}
Observe that all the coefficients that appear in \eqref{e:FPAee} are infinitely smooth and bounded, and $\st_{\bar{\rho}^\e} \geq c_\e$ thanks to the locality \eqref{e:stbelow}, and \ref{i:th3}.
Note that the thickness of initial condition is not necessary here to prove local well-posendness in $H^M_Q$ as in \cite{S-EA} since thickness was only required to ensure sufficient regularity of the protocols, which in our case is provided by design. The a priori bounds stated in \lem{l:apriori}, including the energy, extend to the local solution of the approximated model \eqref{e:FPAee} directly, using that $u^\e \to \rmw^\e$ is bounded on $L^q(\rho^\e)$. Indeed,

\begin{equation}\label{e:wEnergy}
\| \rmw^\e\|^q_{L^q(\rho^\e)} = \int_{\O } \left| \int_{\O } \phi_{\bar{\rho}^{\e}}(x,y) u^\e(y) \rho^\e(y) \dy \right|^q  \rho^\e(x) \dx,
\end{equation}
using that 
\[
\int_{\O } \phi_{\bar{\rho}^{\e}}(x,y)\rho^\e(y) \dy \leq (1+\e) \int_{\O } \phi_{\bar{\rho}^{\e}}(x,y)\bar{\rho}^{\e}(y) \dy  = (1+\e) \st_{\bar{\rho}^{\e}}(x) \leq 2 \oS,
\]
and 
\[
\int_{\O } \phi_{\bar{\rho}^{\e}}(x,y)\rho^\e(x) \dx \leq (1+\e) \int_{\O } \phi_{\bar{\rho}^{\e}}(x,y)\bar{\rho}^{\e}(x) \dx  = (1+\e) \st^*_{\bar{\rho}^{\e}}(x) \leq C \oS.
\]
by the \HI,
\begin{equation}\label{e:L2e}
\| \rmw^\e\|^q_{L^q(\rho^\e)} \leq 2 \oS \int_{\O }  \int_{\O } \phi_{\bar{\rho}^{\e}}(x,y) |u^\e(y)|^2 \rho^\e(y) \dy   \rho^\e(x) \dx \lesssim  \| u^\e\|^q_{L^q(\rho^\e)} \leq j_{1,q}.
\end{equation}
 
As the moments propagate so does regularity of the coefficients on any finite time interval, and thus the solution can be extended indefinitely in $H^M_Q$, proving global wellposedness of the approximated model, again as in  \cite{S-EA}. At the same time, we have determined that the respective moments of the family  as stated in \lem{l:apriori} are preserved on any finite time interval uniformly in $\e$.  It follows from \eqref{e:fdiss}, \eqref{e:fA} that  all the forces are uniformly in $L^\infty([0,T);L^2_{x,v})$, so we have 
\[
\p_t f^\e \in L^\infty([0,T); H^{-2}(\domain)), \quad  f^\e \in L^\infty([0,T); L^2_q(\domain)),
\]
uniformly in $\e$, and also according to \eqref{e:Fisherapriori} $\st_{\bar{\rho}^\e}  \frac{|\n_v f^\e|^2}{f^\e} \in L^\infty_t L^1_{x,v}$ uniformly in $\e$.  Since $L^2_q \hookrightarrow H^{-1}$ compactly, and $H^{-1} \ss H^{-2}$, by the Aubin-Lions Lemma, $\{ f^\e \}_\e \ss C([0,T); H^{-1}(\domain))$ compactly. All the above implies that there exists a subsequence $f^{\e_n} \rightharpoonup f$  strongly in $C([0,T); H^{-1}(\domain))$ and also weakly in momenta space $L^\infty([0,T); L^1_q \cap L^\infty)$. Thus, the limiting function belongs to 
\begin{equation}\label{ }
 f \in L^\infty([0,T);L^1_q \cap L^\infty)\cap C([0,T); \cD').
\end{equation}
The Fisher regularity of the limiting solution is not essential at this point and will be discussed in \sect{ss:Fisher}.

In order to show that the limiting function is indeed a weak solution of \eqref{e:FPA} we need extra compactness provided by the averaging \lem{l:al}.  According to \eqref{e:fdiss} -- \eqref{e:fA} all the forces fulfill the assumptions of the lemma uniformly in $\e$. To make the solution defined on full time line $\R$  let us fix a time-cutoff function $h\in C^\infty_0(0,T)$, and rewrite \eqref{e:FPAee} as follows
\begin{equation}\label{}
\p_t( h f^\e )+ v\cdot \n_x (h f^\e)  =  \sum_l \n^l_v ( h G^\e_l) +  f^\e \p_t h, 
\end{equation}
Now, $h f^\e$ solves the equation on the entire time line $t\in \R$. 

Let $\f \in C^\infty_0( B_R)$, and denote
\[
\rho^\e_{\f} (t,x) = \int_{\R^{n}} \f(v) f^\e(t,x,v) \dv.
\]
Then by \lem{l:al}, $\| h \rho^\e_{\f} \|_{H^{1/6}_{t,x}} \leq C$ uniformly in $\e$.

Let us recall that for a subsequence $\e_n \to 0$ we already have $f^{\e_n} \rightharpoonup f$ weakly. By  compactness of embedding $H^{1/6}_0( [0,T] \times \O) \hookrightarrow L^2$ and uniqueness of the limit all the macroscopic quantities at hand converge in $L^2$-metric and for any test-function $\f\in C^\infty_0$. Specifically, let us denote for short
\[
f^n = f^{\e_n} , \quad \rho^n = \rho^{\e_n}, \quad \rho^n_\f = \rho^{\e_n}_\f, \quad \bar{\rho}^{n} = \bar{\rho}^{\e_n}, \quad \text{etc},
\]
We have
 \begin{equation}\label{e:macroL2}
\| h \rho_{\f}^n -  h \rho_{\f} \|_{L^2_{t,x}} \to 0,
\end{equation}
for any $h\in C_0^\infty(0,T)$.  
The convergence \eqref{e:macroL2} also extends to $[0,T)$ by observing that 
\[
\sup_{[0,T)} \| \rho_{\f}^n (t) \|_{L^2_{x}}+ \| \rho_{\f}(t) \|_{L^2_{x}} < C,
\]
and introducing a cut-off $h$ supported on $[\d/2, T+1)$ and $h = 1$ on $[\d,T]$,
\[
\int_0^T \| \rho_{\f}^n(t)  -  \rho_{\f} (t) \|^2_{L^2_{x}} \dt \leq \| h \rho_{\f}^n - h  \rho_{\f} \|^2_{L^2_{t,x}} + \int_0^\d \| \rho_{\f}^n(t)  -  \rho_{\f} (t) \|^2_{L^2_{x}} \dt \leq o(1) + C\d.
\]
Thus,
 \begin{equation}\label{e:macroL3}
\| \rho_{\f}^n -  \rho_{\f} \|_{L^2([0,T) \times \O)} \to 0.
\end{equation}
We also have trivial convergence of macroscopic quantities initially
 \begin{equation}\label{e:macroLini}
\| \rho_{\f}^n(0) -  \rho_{\f}(0) \|_{L^2(\O)} \to 0.
\end{equation}

Further consequences of  \eqref{e:macroL3} is the convergence of the macroscopic densities on $[0,T)$,
 \begin{equation}\label{e:macroL1}
\| \rho^n -  \rho \|_{L^1([0,T) \times \O)} \to 0.
\end{equation}
Indeed, denoting by $\f_R(v) = \f(v/R)$, where $R>0$ is large and $\f$ is a cut-off function, we have 
\begin{equation*}\label{}
\begin{split}
\| \rho^n -  \rho \|_{L^1([0,T) \times \O)} & \leq \| \rho^n_{\f_R} -  \rho_{\f_R} \|_{L^1([0,T) \times \O)} + \int_{(0,T) \times \O} \int_{|v| > R} |f^n - f| \dv \dx \dt  \\
& \leq  \| \rho^n_{\f_R} -  \rho_{\f_R} \|_{L^1([0,T) \times \O)} + \frac{1}{R^2} ( \|f^n\|_{L^\infty_t L^1_2} + \|f\|_{L^\infty_t L^1_2})
\end{split}
\end{equation*}
Sending $n\to \infty$ and then $R\to \infty$ proves the convergence.

Similarly,  we have convergence of momenta
 \begin{equation}\label{e:macromoment}
\| \rho^n u^n -  \rho u \|_{L^1([0,T) \times \O)}  \to 0.
\end{equation}
Indeed,  
\begin{equation*}\label{}
\begin{split}
\| \rho^n u^n -  \rho u \|_{L^1} & \leq \| \rho^n_{v \f_R} -  \rho_{v \f_R} \|_{L^1} + \int_{(0,T) \times \O} \int_{|v| > R} |v| |f^n - f| \dv \dx \dt  \\
& \leq  \| \rho^n_{v \f_R} -  \rho_{v \f_R} \|_{L^1} + \frac{1}{R}  ( \|f^n\|_{L^\infty_t L^1_2} + \|f\|_{L^\infty_t L^1_2})
\end{split}
\end{equation*}
and we conclude as before.

\subsection{Obtaining a weak solution}
It remains to demonstrate that $f$ is a weak solution on $[0,T) \times \O \times \R^n$.  To control the convergence we first establish auxiliary continuity properties of the protocol that follow directly from the regularity axioms \ref{i:r1}-\ref{i:r2}, and demonstrate that in an averaged sense those regularity conditions depend only on the thickness of one of the densities.

 \begin{lemma}\label{l:semi-unc}
For any regular protocol, we have
\begin{align}
|\p^k_x \st_\rho(x)| & \leq C_k(\th_\rho(x)), \label{e:sreg1}\\
\int_{\O} | \phi_{\rho}(x,y) - \phi_{\rho'}(x,y)| \drho'(y) & \leq C(\th_{\rho}(x)) W_1(\rho,\rho'), \label{e:semiF}\\
| \st_{\rho}(x)-\st_{\rho'}(x) | & \leq C(\th_{\rho}(x) \vee \th_{\rho'}(x) )W_1( \rho, \rho') .
\label{e:semiS}
\end{align}
\end{lemma}
 \begin{proof} \eqref{e:sreg1} is an immediate consequence of \ref{i:r1}. 
 
According to \ref{i:th4}, if $W_1(\rho,\rho')< \frac{\th_{\rho}(x)}{2C}$, then $\th_{\rho'}(x)> \frac12\th_{\rho}(x)$, and so \eqref{e:semiF} is a consequence of \ref{i:r2}.  If, however, $W_1(\rho,\rho') \geq  \frac{\th_{\rho}(x)}{2C}$, then
 \begin{multline*}
 \int_{\O} | \phi_{\rho}(x,y) - \phi_{\rho'}(x,y)| \drho'(y) \leq  \int_{\O}  \phi_{\rho'}(x,y) \drho'(y) + \int_\O \phi_{\rho}(x,y) \drho'(y) \leq  \st_{\rho'}(x) + C(\th_{\rho}(x)) \\
 \leq 2C \frac{\oS + C(\th_{\rho}(x))}{\th_{\rho}(x)}W_1(\rho,\rho') \leq C'(\th_{\rho}(x)) W_1(\rho,\rho') ,
\end{multline*}
 as desired. 
 
Let us verify \eqref{e:semiS}: by \eqref{e:semiF},
\begin{equation}\label{e:ssemi}
\begin{split}
\st_{\rho}(x)-\st_{\rho'}(x) & = \int_\O \phi_{\rho}(x,y) \drho(y) - \int_\O \phi_{\rho'}(x,y) \drho'(y) \\
&=\int_\O [ \phi_{\rho}(x,y)-\phi_{\rho'}(x,y)] \drho(y) + \int_\O \phi_{\rho'}(x,y) [\drho(y)- \drho'(y)]\\
& \leq C(\th_{\rho'}(x)) W_1( \rho, \rho' ) + \| \n_y \phi_{\rho'}(x,\cdot) \|_\infty W_1( \rho, \rho') \leq C(\th_{\rho'}(x)) W_1( \rho, \rho') .
\end{split}
\end{equation}
By symmetry we can switch $\th_{\rho'}(x)$ and $\th_{\rho}(x)$ and take the maximum.
 \end{proof}

By approximation we can assume that a test function  is given by $ \chi(t,x) \f(v) \in C_0^\infty([0,T) \times \O \times \R^n)$. Plugging it into \eqref{e:weak} we have

\begin{equation}\label{e:weakn}
\begin{split}
& \int_0^T\int_{\O}  \rho_{\f}^n(t,x) \p_t \chi(t, x) \dx \dt  +  \int_{\O} \rho^n_\f(0,x) \chi(0,x) \dx+ \int_0^T\int_{\O}  \rho_{v \f}^n(t,x) \cdot \n_x \chi(t,x) \dx \dt \\
 = - & \int_0^T\int_{\O} \st_{\bar{\rho}^n} \rho^n_{ \D_v\f}(t,x)  \chi(t,x) \dx \dt  + \int_0^T\int_{\O} \st_{\bar{\rho}^n}(x) \rho^n_{v \cdot \n \f}(t,x) \chi(t,x)  \dx \dt \\
  - &  \int_0^T\int_{\O} \rmw^n \cdot \rho^n_{\n \f}(t,x) \chi(t,x) \dx \dt : = I + II + III.
\end{split}
\end{equation}

Convergence of the left hand side is trivial. Let us focus on the right hand side.

The strategy will be to split analysis to thin and thick regions and to use regularity of the protocol available on the thick regions. To this end, by \ref{i:th2}, for any $\e>0$ we can find a $\d>0$ such that  $\int_0^T \int_{\th_\rho <2 \d} \rho \dx \dt <\e$. Then due to \eqref{e:macroL1} starting from some large $n$, we will also have $\int_0^T \int_{\th_\rho <2 \d} \rho^n \dx \dt <2 \e$.  Furthermore, 
\[
\int_0^T \int_{\th_{\bar{\rho}^n} <  \d} \rho \dx = \int_0^T \int_{ \{\th_{\bar{\rho}^n} <\d\} \cap \{\th_{\rho} <  2\d\} } \rho \dx +  \int_0^T \int_{\{\th_{\bar{\rho}^n} <\d\} \cap \{\th_{\rho} \geq  2 \d\} } \rho \dx.
\]
The first integral does not exceed $\e$, while for the second we have the measure of the underlying set is vanishing. Indeed, 
\[
 | \{\th_{\bar{\rho}^n} <\d\} \cap \{\th_{\rho} \geq  2 \d\}  |_{\O \times [0,T]} \leq \frac{1}{\d }\int_0^T \int_{\{\th_{\bar{\rho}^n} <\d\} \cap \{\th_{\rho} \geq  2 \d\} } | \th_{\bar{\rho}^n} - \th_{\rho} | \dx \dt
\]
which by \ref{i:th4} and \eqref{e:macroL1},
\[
\lesssim \frac{1}{\d } \int_0^T \int_\O  |\bar{\rho}^n - {\rho} | \dx \dt \to 0.
\]
Thus,
\[
\int_0^T \int_{\{\th_{\bar{\rho}^n} <\d\} \cap \{\th_{\rho} \geq  2 \d\} } \rho \dx \to 0.
\]
So, for $n$ large,
\begin{equation}\label{e:denssmall}
 \int_0^T \int_{ \{\th_{\bar{\rho}^n} <\d\} \cup \{\th_{\rho} <  \d\} } \rho \dx \leq 2\e.
\end{equation}
Due to \eqref{e:macroL1}, for $n$ large, we have a similar bound on $\rho^n$,
\begin{equation}\label{e:denssmalln}
 \int_0^T \int_{ \{\th_{\bar{\rho}^n} <\d\} \cup \{\th_{\rho} <  \d\} } \rho^n \dx \leq 3\e.
\end{equation}

Let us go back to \eqref{e:weakn}, the first two terms on the right hand side are similar. So, let us focus on the first one only. 
 Since we have a pointwise bound $|\rho_{\D_v\f}(t,x)| \leq c \rho(t,x)$, it implies $\int_0^T \int_{\th_\rho <\d} |\rho_{\D_v\f}(t,x)| \dx \dt <c \e$. We therefore, obtain by uniform boundedness of $\st$'s, and \eqref{e:denssmall} - \eqref{e:denssmalln},
 \begin{equation*}\label{}
\begin{split}
 \int_0^T\int_{\th_\rho <\d} | \st_{\rho} \rho_{\D_v\f}(t,x)  \chi(t,x) | \dx \dt & \leq c \e \\
 \int_0^T \int_{\th_\rho <\d} | \st_{\bar{\rho}^n}  \rho^n_{\D_v\f}(t,x)  \chi(t,x) | \dx \dt &\leq c \e,
\end{split}
\end{equation*}
 for $n$ large enough. Whereas on the complement, by \eqref{e:semiS} and \eqref{e:macroL3},
 \begin{equation*}\label{}
\begin{split}
 \int_0^T \int_{\th_\rho  \geq \d} | \st_{\bar{\rho}^n}  \rho^n_{\D_v\f}(t,x) - \st_{\rho} \rho_{\D_v\f}(t,x)||  \chi(t,x)| \dx \dt & \leq  \int_0^T \int_{\th_\rho  \geq \d} | \st_{\bar{\rho}^n} - \st_{\rho}| |\rho^n_{\D_v\f}(t,x)  \chi(t,x)| \dx \dt \\
 & + \oS \int_0^T \int_{\th_\rho  \geq \d} |\rho^n_{\D_v\f}(t,x) -\rho_{\D_v\f}(t,x)||  \chi(t,x)| \dx \dt \to 0.
\end{split}
\end{equation*}
 
Lastly, let us discuss the convergence
\begin{equation}\label{e:convw}
\int_0^T\int_{\O} \rmw^n \cdot \rho^n_{\n \f}(t,x) \chi(t,x) \dx \dt \to \int_0^T\int_{\O} \rmw_\rho(u) \cdot \rho_{\n \f}(t,x) \chi(t,x) \dx \dt.
\end{equation}
The strategy will be similar -- working on the thick and thin sets separately.

For those $n$'s we obtain due to uniform bound on the $L^2$-norm of the averaging, and \eqref{e:denssmall} - \eqref{e:denssmalln},
\begin{equation*}\label{}
\begin{split}
 \int_0^T \int_{ \{\th_{\bar{\rho}^n} <\d\} \cup \{\th_{\rho} <  \d\} } \rmw_\rho(u) \cdot \rho_{\n \f}(t,x) \chi(t,x) \dx \dt & \leq \| \rmw_\rho \|_{L^2(\rho)} \left( \int_0^T \int_{ \{\th_{\bar{\rho}^n} <\d\} \cup \{\th_{\rho} <  \d\} } \rho \dx \dt  \right)^{1/2} \lesssim \e^{1/2} \\
 \int_0^T \int_{ \{\th_{\bar{\rho}^n} <\d\} \cup \{\th_{\rho} <  \d\} } \rmw^n \cdot \rho^n_{\n \f}(t,x) \chi(t,x) \dx \dt & \leq \| \rmw^n \|_{L^2(\rho^n)} \left( \int_0^T \int_{ \{\th_{\bar{\rho}^n} <\d\} \cup \{\th_{\rho} <  \d\} }\rho^n \dx \dt  \right)^{1/2}\lesssim \e^{1/2}.
\end{split}
\end{equation*}
On the other hand, on the complement, we have
\begin{equation*}\label{}
\begin{split}
& \int_0^T\int_{ \{\th_{\bar{\rho}^n}  \geq \d\} \cap \{\th_{\rho}  \geq  \d\} } (\rmw^n  \cdot \rho^n_{\n \f}(t,x) - \rmw_\rho(u) \cdot \rho_{\n \f}(t,x) ) \chi(t,x) \dx \dt  \\
= &  \int_0^T\int_{ \{\th_{\bar{\rho}^n}  \geq \d\} \cap \{\th_{\rho}  \geq  \d\} } \int_\O [\phi_{\bar{\rho}^n}(x,y) - \phi_\rho(x,y)] u^n(y) \rho^n(y)\dy\rho^n_{\n \f}(t,x) \chi(t,x) \dx \dt \\
+ &  \int_0^T\int_{ \{\th_{\bar{\rho}^n}  \geq \d\} \cap \{\th_{\rho}  \geq  \d\} } \int_\O \phi_{\rho}(x,y)(u^n(y) \rho^n(y) - u(y)\rho(y)) \dy  \rho^n_{\n \f}(t,x) \chi(t,x) \dx \dt \\
+ &  \int_0^T\int_{ \{\th_{\bar{\rho}^n}  \geq \d\} \cap \{\th_{\rho}  \geq  \d\} } \int_\O \phi_{\rho}(x,y) u(y)\rho(y) \dy [ \rho^n_{\n \f}(t,x)-\rho_{\n \f}(t,x)] \chi(t,x) \dx \dt : = I + II + III.
\end{split}
\end{equation*}
In the last two integrals the kernel is uniformly bounded, and so the convergence $II, III \to 0$ follows immediately from  \eqref{e:macroL3} and \eqref{e:macromoment}. For $I$ we use the continuity of the kernel \ref{i:r2} to estimate
\[
|I| \lesssim \| u^n \|_{L^2(\rho^n)} \int_0^T \int_\O  |\bar{\rho}^n - {\rho} | \dx \dt \to 0.
\]
This establishes  the convergence \eqref{e:convw}.

\subsection{Fisher Information}\label{ss:Fisher}
Let us now turn to the Fisher regularity of the limiting solution.  We have a uniform bound from the approximation by \lem{l:apriori}, 
\begin{equation}\label{e:Fn}
  \int_0^T \int_\domain \st_{\bar{\rho}^n} |\n_v f^n |^2(f^n)^{p-2} \jap{v}^{q'} dv\dx \ds \leq C,
\end{equation}
where $q' = q$ if $p>1$, and $q'<q$ if $p=1$. It suffices to show that 
\begin{equation}\label{e:limitF0}
 \int_0^T \int_{\R^n}   \int_{\{ \th_\rho>0\} \cap \{ \th_{\rho,r_0}>0\} \ }   \st_\rho |\n_v f |^2 f^{p-2}  \jap{v}^{q'} \dv\dx \ds <\infty,
\end{equation}
where $r_0>0$ is coming from the locality assumption \eqref{e:loc}. Indeed,  on the set $ \{\th_\rho= 0\}\cup \{ \th_{\rho,r_0}=0\}$ we have $\rho =  0$ a.e. by \ref{i:th2}, which implies that $f(x,v) = 0$ a.e. on $ \{\th_\rho= 0\}\cup \{ \th_{\rho,r_0}=0\} \times \R^n$.  It is known that $\n_v f^{p/2} =0$ a.e. on that set too, and therefore the integral vanishes.

To prove \eqref{e:limitF0} we fix a $\d>0$, $R>0$, and prove instead
\begin{equation}\label{e:limitF0}
 \int_0^T \int_{|v|<R}   \int_{ \{\th_\rho>\d\} \cap \{ \th_{\rho,r_0}>\d\} }  \st_\rho  |\n_v f |^2 f^{p-2}  \jap{v}^{q'} \dv\dx \ds <\infty,
\end{equation}
uniformly in $\d$ and $R>0$.  But on the set $\{ \th_\rho>\d \}$ as we argued by \eqref{e:semiS} that  $\st_{\bar{\rho}^n} \to  \st_\rho$ uniformly. Since by the locality \eqref{e:stbelow}, $\st_\rho >c \d$ on $\{ \th_{\rho,r_0}>\d\}$ we conclude that starting from large $n$, $\st_{\bar{\rho}^n} >c\d/2$ uniformly on the set. This implies by \eqref{e:Fn} that $|\n_v f^n |^2 (f^n)^{p-2} \jap{v}^{q'}\in L^1$ uniformly on the set $\{\th_\rho>\d\} \cap \{ \th_{\rho,r_0}>\d\} $ for large $n$.  We now appeal to the classical lower semi-continuity of the Fisher information, see \cite{Villani-optimal}, to conclude that 
\[
 \int_0^T \int_{|v|<R}   \int_{ \{\th_\rho>\d\} \cap \{ \th_{\rho,r_0}>\d\} }  \st_\rho |\n_v f |^2 f^{p-2} \jap{v}^{q'}  \dv\dx \ds \leq C
 \]
 as desired. Passing to the limit as $\d \to 0$ and $R\to \infty$ we prove the result.

\subsection{Energy law} \label{s:enlaw}
Let us derive the energy law \eqref{e:eneq} from the weak formulation for a given weak solution in the regularity class \eqref{e:regweak}. Since we are only allowed to plug in test-functions with compact support to do that rigorously  we approximate $\frac12 |v|^2$ by a cut-off function $\f_R(v) = \f(v/R)$, where $\f =1$ on the ball $|v|<1$ and $\f\in C^\infty_0(|v|<2)$. Notice that in the diffusion term we can move one $\n_v$ back on $\st_\rho f$ since $\st_\rho \n_v f \in L^1$.  One obtains
\begin{equation}\label{e:enR}
\begin{split}
\frac12 \int_{\domain \times \{t\} } |v|^2 \f_R(v)  f(t,x,v) \dv \dx - \frac12 \int_{\domain  \times \{0\} } |v|^2 \f_R(v)  f_0 \dv \dx \\
 = \frac12 \int_0^t\int_{\domain} \st_\rho f \D_v (|v|^2 \f_R(v)) \dv \dx \ds  + \frac12 \int_0^t\int_{\domain} \cA(f) f  \cdot \n_v (|v|^2 \f_R(v)) \dv \dx \ds.
\end{split}
\end{equation}
Since the energy is finite for all $t$, we trivially have
\begin{equation*}\label{}
\begin{split}
\frac12 \int_{\domain \times \{t\} } |v|^2 \f_R(v)  f(t,x,v) \dv \dx&  \to \cE(t) \\
\frac12 \int_{\domain \times \{t\} } |v|^2 \f_R(v)  f(0,x,v) \dv \dx&  \to \cE(0) 
\end{split}
\end{equation*}
As to the diffusion term, we notice that $| \D_v (|v|^2 \f_R(v)) | \leq C$ uniformly in $R$, and $ \D_v (|v|^2 \f_R(v)) \to n $ pointwise, while $| \st_\rho f | \lesssim  f \in L^1$. So, by the dominated convergence,
\[
\frac12 \int_0^t\int_{\domain} \st_\rho f \D_v (|v|^2 \f_R(v)) \dv \dx \ds \to n  \int_0^t\int_{\domain} \st_\rho f  \dv \dx \ds.
\]
Lastly, the alignment term is given by
\begin{equation*}\label{}
\begin{split}
 \frac12 \int_0^t\int_{\domain} \cA(f) f  \cdot \n_v (|v|^2 \f_R(v)) \dv \dx \ds& = \int_0^t\int_{\domain} \cA(f) f  \cdot v \f_R(v) \dv \dx \ds \\
& + \frac12  \int_0^t\int_{\domain} \cA(f) f  \cdot  |v|^2 \n_v \f_R(v) \dv \dx \ds.
\end{split}
\end{equation*}
For the first integral we have $|\cA(f) f  \cdot v| \lesssim |v|^2 f + |\rmw_\rho(u)| |v| f$, and $ |\rmw_\rho(u)| |v| f \leq  |\rmw_\rho(u)|^2  f + |v|^2 f$, all in $L^1$. So, 
\begin{multline*}
\int_0^t\int_{\domain} \cA(f) f  \cdot v \f_R(v) \dv \dx \ds \to \int_0^t\int_{\domain} \cA(f) f  \cdot v \dv \dx \ds \\
=\int_0^t (\rmw_\rho(u), u)_\rho \ds -  \int_0^t\int_{\domain} \st_\rho |v|^2 f \dv \dx \ds.
\end{multline*}
And
\begin{equation*}\label{}
\begin{split}
\left|  \int_0^t\int_{\domain} \cA(f) f  \cdot  |v|^2 \n_v \f_R(v) \dv \dx \ds \right| \leq  \int_0^t\int_{R < |v|<2R} |v|^2 f +|\rmw_\rho(u)|  f   |v| \dv \dx \ds.
\end{split}
\end{equation*}
We already concluded that the function inside is integrable, therefore this term vanishes as $R\to \infty$.

We conclude that $t \to \cE(t)$ is  absolutely continuous  and 
\begin{equation}\label{ }
\ddt \cE =(\rmw_\rho(u), u)_\rho -  \int_{\domain} \st_\rho |v|^2 f \dv \dx + n  \int_{\domain} \st_\rho f  \dv \dx.
\end{equation}
Noting that $(\rmw_\rho(u), u)_\rho = (\ave{u}_\rho, u)_{\k_\rho}$ proves \eqref{e:eneq}.

\subsection{Uniqueness from non-vacuous data} Let us address the uniqueness part of  \thm{t:weak}.

As a consequence of \lem{l:contth} and \ref{i:th3}, both native and ball-thicknesses will remain positive 
\[
\th_{\rho'}(\O), \th_{\rho''}(\O), \th_{\rho',r_0}(\O), \th_{\rho'',r_0}(\O) \geq c_0
\]
 on some interval of time $[0,t^*)$.  We will prove the uniqueness on any such time interval. So, the first moment of time  uniqueness fails is necessarily when one of these thicknesses vanishes, which implies $\min \rho = 0$ at that time by \ref{i:th3}. So, as long as solutions remain non-vacuous, or even thick in the sense of both native and ball-thicknesses, they will remain the same. 

Thanks to the regularity and locality of the protocol $\cC$, the strength-functions $\st'= \st_{\rho'}$, $\st''= \st_{\rho''}$ remain $C^\infty$ and bounded from below, and the kernels $\phi', \phi''$ are also in $C^\infty_{x,y}$. So, the formal computations below are justified classically. 
For notational brevity we relabel $q \to q+2$, and assume that $q \geq n+2$, while knowing that the $(q+2)$-moments are bounded.

We have
\begin{equation*}\label{}
\begin{split}
\ddt \| g\|^2_{L^2_q}  = & - \int_\domain \st' |\n_v g|^2 \jap{v}^q  \dv \dx -  \int_\domain \st' g \n_v g \n_v \jap{v}^q  \dv \dx \\
&  - \int_\domain (\st'-\st'') \n_v f'' \cdot \n_v g \jap{v}^q  \dv \dx -  \int_\domain (\st'-\st'') \n_v f''  g \n_v \jap{v}^q  \dv \dx\\
& - \frac12 \int_\domain \st' v g^2 \n_v \jap{v}^q  \dv \dx \\
& - \int_\domain (\st'- \st'')v f'' \n_v g  \jap{v}^q \dv \dx - \int_\domain (\st'- \st'')v f'' g \n_v  \jap{v}^q \dv \dx\\
& - \int_\domain \rmw' \cdot \n_v \jap{v}^q g^2  \dv \dx -  \int_\domain ( \rmw' - \rmw'') \cdot \n_v  f'' g \jap{v}^q  \dv \dx
\end{split}
\end{equation*}

The first term gives diffusion supported from below $c \int_\domain |\n_v g|^2 \jap{v}^q  \dv \dx$. Therefore, the second term gets absorbed leaving $\| g\|^2_{L^2_q}$. For the third term we use \eqref{e:semiS}, 
\[
\int_\domain (\st'-\st'') \n_v f'' \cdot \n_v g \jap{v}^q  \dv \dx \lesssim \| \rho' - \rho''\|_1^2 \| \n_v f''\|^2_{L^2_q} + \e  \| \n_v g\|_{L^2_q}^2
\]
denoting by $L(t)$ a generic integrable in time function, we continue,
\[
\lesssim L \| \rho' - \rho''\|_1^2  + \e  \| \n_v g\|_{L^2_q}^2.
\]
Now, 
\begin{equation}\label{e:rhodiff}
 \| \rho' - \rho''\|_1 \leq \int_\domain |g|  \dv \dx\leq \int_\domain |g| \jap{v}^{q/2} \frac{ \dv \dx}{\jap{v}^{q/2}} \lesssim  \| g\|_{L^2_q}. 
\end{equation}
So, 
\[
\int_\domain (\st'-\st'') \n_v f'' \cdot \n_v g \jap{v}^q  \dv \dx \lesssim L \| g\|_{L^2_q}^2 +  \e  \| \n_v g\|_{L^2_q}^2.
\]

For the fourth term we obtain similarly,
\[
 \int_\domain (\st'-\st'') \n_v f''  g \n_v \jap{v}^q  \dv \dx \lesssim L \| g\|_{L^2_q}^2.
\]
The fifth is trivially bounded by $\lesssim  \| g\|_{L^2_q}^2$. For the sixth, we have
\begin{equation*}\label{}
\int_\domain (\st'- \st'')v f'' \n_v g  \jap{v}^q \dv \dx \leq \| \rho' - \rho''\|_1 \int_\domain \jap{v}^{1+q/2} f''  |\n_v g|  \jap{v}^{q/2} \dv \dx \lesssim \| g\|_{L^2_q}^2 \|f''\|_{L^2_{q+2}}  \lesssim \| g\|_{L^2_q}^2.
\end{equation*}
For the seventh, we trivially have $\lesssim \| g\|_{L^2_q}^2$. Finally, the eighth is bounded by $\lesssim \| g\|_{L^2_q}^2$ thanks to uniform boundedness of the weighted average $\rmw'$.  And the ninth term is bounded by
\[
 \int_\domain ( \rmw' - \rmw'') \cdot \n_v  f'' g \jap{v}^q  \dv \dx \leq \| \rmw' - \rmw''\|_\infty  \|\n_v f''\|_{L^2_{q}}\| g\|_{L^2_q} \leq L \| \rmw' - \rmw''\|_\infty \| g\|_{L^2_q}.
\]
We have
\[
\rmw' - \rmw'' = \int_\O \phi' (u' \rho' - u'' \rho'') \dy + \int_\O (\phi' - \phi'') u'' \rho'' \dy.
\]
The second integral is bounded by, using \ref{i:r2} and \eqref{e:rhodiff},
\[
\lesssim \| g\|_{L^2_q} \cE \lesssim \| g\|_{L^2_q}.
\]
And as to the first,
\[
\int_\T^n | u' \rho' - u'' \rho'' |  \dy \leq \int_\domain |v| |g|  \dv \dy \leq  \int_\domain \jap{v}^{q/2} |g|  \frac{\dv \dy}{\jap{v}^{q/2 - 1}}
\lesssim \| g\|_{L^2_q}.
\]
Thus,
\[
\| \rmw' - \rmw''\|_\infty \lesssim  \| g\|_{L^2_q}.
\]
We conclude that the ninth term is bounded by $L  \| g\|_{L^2_q}^2$.

Putting the estimates together we obtain
\[
\ddt \| g\|^2_{L^2_q} \leq L  \| g\|^2_{L^2_q},
\]
and the uniqueness follows.

\section{Renormalization}\label{s:renormFP}
 
In this section we discuss renormalization of weak solutions in the sense of  DiPerna-Lions, see \cite{DiPL1989}.  The renormalization  allows to justify equations for compositions with other sufficiently regular functions $\b : \R \to \R$, for example $f^p$, $f\log f$ or $(f - \chi)_+$. The utility of such compositions is broad -- from maximum and comparison principles to entropy law and propagation of higher moments.

For $\b \in C^2(\R)$,  we can formally write the equation for $\b(f)$ as follows
\begin{equation}\label{e:renorm}
\p_t \b(f)  + v \cdot \n_x \b(f) + \n_v( \cA(f) \b(f))  =   n \st_\rho (f \b'(f) - \b(f)) + \st_\rho (\D_v  \b(f) -  |\n_v f|^2 \b''(f)).
\end{equation}
Here, we encounter the first issue: $\cA(f) \b(f)$ is not a distribution unless $\b(0) =0$. Indeed, if say $\b \equiv 1$, then $\rmw_\rho(u)$ is not guaranteed to be integrable with respect to the Lebesgue measure (recall that $\rmw_\rho(u)\in L^2(\rho)$, i.e. relative to the density only). So, to make the renormalization well-defined, we add and subtract $\b(0)$:
\begin{multline}\label{e:renorm-0}
\p_t \b(f)  + v \cdot \n_x \b(f) + \n_v( \cA(f)( \b(f) - \b(0)))  =  n \st_\rho (f \b'(f) + \b(0) - \b(f)) \\
+ \st_\rho (\D_v  \b(f) -  |\n_v f|^2 \b''(f)).
\end{multline}
Since $| \b(f) - \b(0)| \lesssim f$, this formulation makes distributional sense.

The second issue becomes apparent if we recall the classical DiPerna-Lions approach to proving renormalization, particularly for the transport term. One attempts to reach \eqref{e:renorm-0} by mollifying the equation, which for the transport term results in $(\cA(f) f)_\e$. Testing with $\b'(f_\e)$, we attempt to use the classical chain rule. So, we add and subtract $\cA(f) f_\e$. The proof reduces to controlling the commutator $(\cA(f) f)_\e -\cA(f) f_\e$. This classical approach, however, fails from the start because the term  $\cA(f) f_\e$ does not make distributional sense.  This basic strategy still works if extra scrutiny is given by splitting the commutator estimates between bulk regions $\rho >\d$, and peripheral regions $\rho <\d$ where the flock is thin and the alignment force is small. 

The third issue is in order to control commutators the minimal regularity requirement on $f$  is  Fisher  $\n_v \sqrt{f} \in L^1$, which in our case comes with degenerate weight $\st_\rho$. So, to access the weighted Fisher regularity,  we implement an extra ``pre-renormalization" step by looking at the equation for a properly weighted  and mollified $f$.

\subsection{Main result} We first state and prove a classical version of renormalization valid for smooth $\b$'s before addressing the sharp version in the next subsection.

\begin{proposition}\label{p:renorm} Let $\cC$ be a local regular protocol. Then any weak solution $f$ to \eqref{e:FPA} in  class \eqref{e:regweak} is renormalized \eqref{e:renorm-0} for any $\b \in C^{2}$.
\end{proposition}

\begin{proof}

First step will be to pre-renormalize the equation with a weight $\w$ that is subordinate to $\s$ and that gives us access to the Fisher integrability of  $\s |\n_v \sqrt{f}|^2$, as opposed to simply $f$.  

Let us fix two small parameters $\e_1, \e_2>0$ and a standard compactly supported mollifier $\chi_\e\in C^\infty_0(\domain)$.  Denote  $\w  = \frac{\th_{\rho,r_0}}{\th_{\rho,r_0} + \e_2}$, 
 and observe that $\w \in C^{1,1}_{x,t} \cap L^\infty$, which is a sufficient regularity for $\w$ to use as a test-function. We observe that $\w f \in L^1 \cap L^\infty$ and $\w \n_v f \in L^2_{t,x,v}$, the latter following from the 
  locality \eqref{e:stbelow} and the Fisher regularity of $f$. The Fisher norm of course deteriorates as $\e_2 \to 0$, but for every fixed $\e_2$ it is finite.
   
 So, let us write the equation for $ (\w f)_{\e_1}$:
\begin{equation}\label{e:FPAmoll}
\p_t (\w f)_{\e_1} + (v \cdot \n_x(\w f))_{\e_1}  + ( \w \n_v(\cA(f) f ))_{\e_1} =   (\w \st_\rho \D_v f )_{\e_1} 
+ R_{\e_1, \e_2},
\end{equation}
where  $R$ is the residual term
\[
R_{\e_1, \e_2} = (\p_t \w f)_{\e_1} +( f v \cdot \n_x \w)_{\e_1} .
\]
Next, we write test with $\b'( (\w f)_{\e_1})$ and write the equation for $\b( (\w f)_{\e_1})$ by the classical chain rule:
\begin{multline}\label{e:e1e2}
\p_t  \b((\w f)_{\e_1})  + (v \cdot \n_x(\w f))_{\e_1} \b'( (\w f)_{\e_1}) +   ( \w \n_v(\cA(f) f ))_{\e_1}  \b'( (\w f)_{\e_1}) \\
=    (\w  \st_\rho \D_v f)_{\e_1}  \b'( (\w f)_{\e_1})
+ R_{\e_1, \e_2}  \b'( (\w f)_{\e_1}).
\end{multline}
This equation is transplation invariant  $\b \to \b + c$. So, we can assume without loss of generality that $\b(0) = 0$.

We will now take consecutive limits as $\e_1 \to 0$ and $\e_2 \to 0$, starting with $\e_1 \to 0$.

\noindent\underline{{\it Free transport.}} We have
\begin{equation*}\label{}
\begin{split}
\p_t  \b((\w f)_{\e_1})  + (v \cdot \n_x(\w f))_{\e_1} \b'( (\w f)_{\e_1}) & =(\p_t +  v \cdot \n_x) \b( (\w f)_{\e_1}) + C_{\e_1,\e_2}\b'( (\w f)_{\e_1}) \\
C_{\e_1,\e_2} & = (v \cdot \n_x(\w f))_{\e_1} - v \cdot \n_x(\w f)_{\e_1} .
\end{split}
\end{equation*}
 We keep in mind that $\w f \in L^1 \cap L^\infty$, then $(\w f)_{\e_1} \to \w f$ in any $L^p$, $p<\infty$, and up to a subsequence pointwise almost everywhere. This implies $\b((\w f)_{\e_1}) \to \b(\w f)$ in $L^1$, and as a consequence, 
\[
(\p_t +  v \cdot \n_x)  \b((\w f)_{\e_1}) \to (\p_t +  v \cdot \n_x) \b(\w f ) 
\]
distributionally. As to the commutator we have
\begin{equation}\label{e:C1}
\begin{split}
C_{\e_1,\e_2}  &= \int_\domain \w(x-y) f(x-y,v-w) w \n_y \f_{\e_1}(y,w) \dy \dw \\
&=  \int_\domain \w(x-\e_1 y) f(x-\e_1 y,v-\e_1 w) w \n_y \f(y,w) \dy \dw \\
&=  \int_\domain \w(x-\e_1 y) (f(x-\e_1 y,v-\e_1 w) - f(x,v)) w\n_y \f(y,w) \dy \dw.
\end{split}
\end{equation}
Since $w \n_y \f(y,w)$ has a compact support and $f(x-\e_1 y,v-\e_1 w) - f(x,v) \to 0$ in $L^1(\dv\dx)$, the convergence $C_{\e_1,\e_2}\to 0$ as $\e_1 \to 0$ in $L^1_\loc$ follows.

 So, for the free transport term we have
\[
\p_t  \b((\w f)_{\e_1})  + (v \cdot \n_x(\w f))_{\e_1} \b'( (\w f)_{\e_1}) \underset{\e_1 \to 0}{ \longrightarrow }(\p_t +  v \cdot \n_x)  \b(\w f).
\]
Now, to take the next limit as $\e_2 \to 0$, notice that pointwise, $\w \to \one_{\th_{\rho,r_0} >0}$. On the set $\{\th_{\rho,r_0} = 0\} \times \R^n$ the distribution $f$ vanishes almost everywhere. So, $f \w \to f$ pointwise almost everywhere, and hence distributionally,
\[
(\p_t +  v \cdot \n_x)  \b(\w f)  \underset{\e_2 \to 0}{  \longrightarrow } (\p_t +  v \cdot \n_x)  \b( f) .
\]

\noindent\underline{{\it Alignment term.}}  First, we have 
\[
\w \n_v(\cA(f) f ) = \n_v(\w \cA(f) f ).
\]
 Let us fix an arbitrary $\d>0$ and observe that
\[
\w \cA(f) f  = \w \cA(f) f \one_{\rho>0} = \w \cA(f) f \one_{\rho >\d} + \w \cA(f) f \one_{0< \rho \leq \d},
\]
almost everywhere, and hence as distributions. Now the crucial observation is that $\cA(f)  \one_{\rho >\d} \in L^2_\loc(\dx\dv)$. In other words, on the set where the density is bounded away from zero the alignment drift is in $L^2$ relative to the Lebesgue measure, not the density-weighted measure. Denote $g = \cA(f)  \one_{\rho >\d}$. Then since $g\in C^\infty_v \cap L^2$, and $\n_v( \w f) \in L^2$, by the classical approximation, we have the product rule
\[
\n_v( g \w f) = -n\st_\rho \w f \one_{\rho >\d}+ g \cdot \n_v (\w f).
\]
So, let us mollify,
\[
(\n_v( g \w f))_{\e_1}  \b'((\w f)_{\e_1}) = -n( \st_\rho \w f \one_{\rho >\d})_{\e_1} \b'((\w f)_{\e_1})+( g \cdot \n_v (\w f))_{\e_1}  \b'((\w f)_{\e_1}).
\]
The first term converges classically
\[
 -n( \st_\rho \w f \one_{\rho >\d})_{\e_1} \b'((\w f)_{\e_1}) \to  -n \st_\rho \w f \one_{\rho >\d} \b'(\w f),
\]
while the second 
\begin{equation*}\label{}
\begin{split}
( g \cdot \n_v (\w f))_{\e_1}  \b'((\w f)_{\e_1}) & =g \cdot \n_v (\w f)_{\e_1}  \b'((\w f)_{\e_1}) + [( g \cdot \n_v (\w f))_{\e_1} - g \cdot \n_v (\w f)_{\e_1} ] \b'((\w f)_{\e_1}) \\
& = g \cdot \n_v  \b((\w f)_{\e_1}) + [( g \cdot \n_v (\w f))_{\e_1} - g \cdot \n_v (\w f)_{\e_1} ] \b'((\w f)_{\e_1}).
\end{split}
\end{equation*}
The commutator between two $L^2$-functions converges to zero classically. As to the main term $g \cdot \n_v  \b((\w f)_{\e_1})$, by the product rule between a $C^\infty$ and $L^2$-function
\[
g \cdot \n_v  \b((\w f)_{\e_1}) = \n_v( g  \b((\w f)_{\e_1})  ) + n \st_\rho \one_{\rho >\d} \b((\w f)_{\e_1}),
\]
which distributionally converges to 
\[
 \to  \n_v( g  \b(\w f) ) + n \st_\rho \one_{\rho >\d}  \b(\w f). 
\]
We thus have proved that 
\[
(\n_v( g \w f))_{\e_1}  \b'((\w f)_{\e_1}) \to \n_v( g  \b(\w f) ) + n \st_\rho \one_{\rho >\d}[ \b(\w f)  - \w f \b'(\w f)]
\]

Let us now consider the complementary  term: 
\begin{equation*}\label{}
\begin{split}
(\n_v( \w \cA(f) f \one_{0< \rho \leq \d}   ))_{\e_1}  \b'((\w f)_{\e_1}) & = \n_v(( \w \cA(f) f \one_{0< \rho \leq \d}   )_{\e_1}  \b'((\w f)_{\e_1}) )\\
& - (\w \cA(f) f \one_{0< \rho \leq \d})_{\e_1}  \b''((\w f)_{\e_1}) \n_v (\w f)_{\e_1} .
\end{split}
\end{equation*}
The first term is under the full derivative, so distributionally it converges to where the expression inside does, which is 
\[
\n_v(( \w \cA(f) f \one_{0< \rho \leq \d}   )_{\e_1}  \b'((\w f)_{\e_1}) ) \to \n_v(\w \cA(f) f \one_{0< \rho \leq \d}   \b'(\w f) ).
\]
The second term converges classically since both $\cA(f)f$ and $ \n_v (\w f)$ are in $L^2$:
\[
- (\w \cA(f) f \one_{0< \rho \leq \d})_{\e_1}  \b''((\w f)_{\e_1}) \n_v (\w f)_{\e_1}  \to - \w \cA(f) f \one_{0< \rho \leq \d} \b''(\w f) \n_v (\w f).
\]

To summarize we have proved so far that for any $\d>0$, the following limit holds:
\begin{equation*}\label{}
\begin{split}
(\n_v( \cA(f) \w f))_{\e_1}  \b'((\w f)_{\e_1})  \to &  \n_v( \cA(f)   \b(\w f)  \one_{\rho >\d}) + n \st_\rho \one_{\rho >\d}[ \b(\w f)   - \w f \b'(\w f)]\\
+ & \n_v(\w \cA(f) f \one_{0< \rho \leq \d}   \b'(\w f) ) -  \w \cA(f) f \one_{0< \rho \leq \d} \b''(\w f) \n_v (\w f).
\end{split}
\end{equation*}
The above identity holds for any $\d \to 0$ on the right hand side, but the left hand side is independent of $\d$. So, it remains to be seen where the right hand side converges  distributionally as $\d\to 0$.  For the first term, since $ |\cA(f)   \b(\w f)  | \lesssim |\cA(f)|f \in L^1$,
\[
 \n_v( \cA(f)   \b(\w f)  \one_{\rho >\d}) \to  \n_v( \cA(f)   \b(\w f)  \one_{\rho >0})=\n_v( \cA(f)   \b(\w f)  ).
 \]
 For the second term, classically, and since it vanishes a.e. on $\rho= 0$,
 \[
  n \st_\rho \one_{\rho >\d}[ \b(\w f)  - \w f \b'(\w f)] \to  n \st_\rho [ \b(\w f)   - \w f \b'(\w f)].
  \]
For the third term, distributionally 
\[
\n_v(\w \cA(f) f \one_{0< \rho \leq \d}   \b'(\w f) ) \to 0.
\]
And lastly, for the fourth term in $L^1$,
\[
 \w \cA(f) f \one_{0< \rho \leq \d} \b''(\w f) \n_v (\w f) \to 0.
\]

We have proved the following limit for the main alignment term  
\[
(\n_v( \cA(f) \w f))_{\e_1}  \b'((\w f)_{\e_1})  \underset{\e_1 \to 0}{ \longrightarrow }  \n_v( \cA(f)   \b(\w f)  ) + n \st_\rho [ \b(\w f)  - \w f \b'(\w f)],
\]
and since $\w f \to f$ pointwise and $|\w f| \leq f$, we further obtain
\[
 \underset{\e_2 \to 0}{ \longrightarrow }  \n_v( \cA(f)  \b(f)  ) + n \st_\rho [ \b(f)  - f \b'(f)].
\]

In conclusion, the alignment term converges to
\begin{equation}\label{ }
 ( \w \n_v(\cA(f) f ))_{\e_1}  \b'( (\w f)_{\e_1}) \to  \n_v( \cA(f)  \b(f)  ) + n \st_\rho [ \b(f)  - f \b'(f)],
\end{equation}
exactly as appears in \eqref{e:renorm-0}.

\noindent\underline{{\it Dissipation term.}}  Let us continue to the dissipation term in \eqref{e:e1e2} given by $(\w  \st_\rho \D_v f)_{\e_1}  \b'( (\w f)_{\e_1})$. By the product rule,
\begin{equation*}\label{}
(\w  \st_\rho \D_v  f)_{\e_1}  \b'( (\w f)_{\e_1})  =  \n_v(  (\w \st_\rho \n_v  f )_{\e_1}  \b'( (\w f)_{\e_1})) - (\w  \st_\rho \n_v f)_{\e_1} \cdot \n_v (\w f)_{\e_1} \b''( (\w f)_{\e_1}).
\end{equation*}
The second term converges classically as pairing between two $L^2$-functions,
\[
(\w  \st_\rho \n_v f)_{\e_1} \cdot \n_v (\w f)_{\e_1} \b''( (\w f)_{\e_1}) \to \w^2  \st_\rho |\n_v f|^2 \b''( \w f).
\]
Let us take a closer look at the first term. It is a full derivative of 
\[
(\w \st_\rho \n_v  f )_{\e_1}  \b'( (\w f)_{\e_1}) =  ( \st_\rho \n_v (\w f) )_{\e_1}  \b'( (\w f)_{\e_1}),
\]
which we write as 
\begin{equation*}\label{}
\begin{split}
( \st_\rho \n_v (\w f) )_{\e_1}  \b'( (\w f)_{\e_1}) & =  \st_\rho \n_v (\w f)_{\e_1}  \b'( (\w f)_{\e_1}) + [ ( \st_\rho \n_v (\w f) )_{\e_1} -  \st_\rho \n_v (\w f)_{\e_1} ] \b'( (\w f)_{\e_1}) \\
& = \st_\rho \n_v   \b( (\w f)_{\e_1}) + [ ( \st_\rho \n_v (\w f) )_{\e_1} -  \st_\rho \n_v (\w f)_{\e_1} ] \b'( (\w f)_{\e_1}).
\end{split}
\end{equation*}
Inside the commutator we have a pairing between $\st_\rho \in L^\infty \ss L^2_\loc(\domain)$ and $g = \n_v (\w f) \in L^2(\domain)$. So, $( \st_\rho g )_{\e_1} - \st_\rho g _{\e_1} \to 0 $ in $L^1_\loc(\domain)$ classically, and therefore distributionally, 
\[
 [ ( \st_\rho \n_v (\w f) )_{\e_1} -  \st_\rho \n_v (\w f)_{\e_1} ] \b'( (\w f)_{\e_1}) \to 0.
 \]
At the same time, $ \b( (\w f)_{\e_1}) \to  \b( \w f)$ pointwise boundedly, so 
\[
 \st_\rho \n_v   \b( (\w f)_{\e_1}) \to  \st_\rho \n_v   \b( \w f) ,
 \]
distributionally.  We have established that the dissipation term converges distributionally to 
\[
(\w  \st_\rho \D_v f)_{\e_1}  \b'( (\w f)_{\e_1})  \underset{\e_1 \to 0}{ \longrightarrow }   \st_\rho \D_v   \b( \w f) - \w^2  \st_\rho |\n_v f|^2 \b''( \w f).
\]

Next, we take the limit as $\e_2 \to 0$. Since $\w f \to f$ pointwise as we argued earlier, while $\st_\rho |\n_v f|^2\in L^1$,  we trivially have
\[
 \st_\rho \D_v   \b( \w f) \to  \st_\rho \D_v   \b(f) 
\]
as a distribution, and 
\[
\w^2  \st_\rho |\n_v f|^2 \b''( \w f) \to \st_\rho |\n_v f|^2 \b''( f),
\]
by the dominated convergence.  This recovers the last term in \eqref{e:renorm-0} as desired.

\noindent\underline{{\it Residual term.}}  Finally, the residual term converges to its natural limit as follows from the regularity of $\w$ and $f$:
\[
R_{\e_1, \e_2}  \b'( (\w f)_{\e_1}) \to R_{\e_2}  \b'( \w f),
\]
where 
\[
R_{\e_2} = \p_t \w f  +  f v \cdot \n_x \w .
\]
It remains to show that $R_{\e_2} \to 0$ in $L^1$. Let us start with the first term,
\[
\p_t \w  = \e_2 \frac{ \p_t \th_{\rho,r_0} }{ (  \th_{\rho,r_0}  + \e_2)^2} .
\]
If $  \th_{\rho,r_0}  \neq 0$, then $\p_t \w \to 0$. If $  \th_{\rho,r_0}  =0$, then $\p_t \w f = 0$ a.e. So, pointwise, we have $\p_t \w f \to 0$. The time derivative of the thickness is uniformly bounded by \lem{l:contth}, so 
\[
|\p_t \w  f| \lesssim \frac{f}{\th_{\rho,r_0}} \in L^1(\domain),
\]
by \eqref{e:KMT0}.  So, $\p_t \w  f \to 0$ in $L^1$ by the dominated convergence. 

Next,
\[
 f v \cdot \n_x \w  = \e_2 f v \cdot  \frac{  \n_x \th_{\rho,r_0} }{ ( \th_{\rho,r_0}  + \e_2)^2}.
 \]
We argue as previously, $ f v \cdot \n_x \w \to 0$ almost everywhere, and 
\[
| f v \cdot \n_x \w| \leq |v| f  \frac{|\n_x \th_{\rho,r_0} | }{\th_{\rho,r_0}} \leq |v| \frac{f}{\sqrt{ \th_{\rho,r_0}} }  |\n_x \sqrt{\th_{\rho,r_0}} |   \leq |v|^2 f +  \frac{f}{\th_{\rho,r_0}} \in L^1,
\]
since $ |\n_x \sqrt{\th_{\rho,r_0}} | \leq C \| D^2 \th_{\rho,r_0}\|^{1/2}_\infty < \infty$.
So, $ f v \cdot \n_x \w \to 0$ in $L^1$. 

This concludes the proof.
\end{proof}

\subsection{Entropy law} Our next goal is to prove the entropy equality for weak solutions in the class  \eqref{e:regweak}. To do that we formally have to substitute  $\b(x) = x \log x$ in the renormalized equation. Since such $\b$ is not smooth we have to perform an extra approximation procedure, which we will do next.

\begin{proposition}\label{p:renormsharp}
Suppose $\b \in C([0,\infty))$, $\b', \b'' \in L^\infty_\loc( (0,\infty))$, $\b(0) = 0$, and 
\[
\sup_{0<x<X} x |\b''(x)| <C(X), \quad \forall X.
\]
Then the renormalization \eqref{e:renorm-0} holds distributionally. 
\end{proposition}

\begin{proof}
Note that for all  $0<x<X$,
\[
|\b'(x) - \b'(1)| \leq A |\log x|.
\]
So, $|\b'(x)| \leq A|\log x| + B$. Integrating once again over $[0,x]$, and using that $\b(0) = 0$,
\[
|\b(x)| \leq A x|\log x| + B x.
\]
According to  \lem{l:LlogL}, $\b(f) \in  L^\infty_t L^1_{x,v}$. 

Let us consider a tapering of $\b$: for $\e>0$,  define
\begin{equation*}\label{}
\b_\e(x) = \left\{\begin{split}
\b(x), &\quad \e<x \\
h_\e(x), &\quad 0\leq  x\leq \e,
\end{split}\right.
\end{equation*}
where $h_\e$ is any $C^\infty$-function on $[0,\e]$ which connects smoothly with $\b(x)$ at $x=\e$ and such that 
\begin{equation}\label{e:he}
|h_\e(x)| \lesssim x |\log x|, \quad |h'_\e(x)| \leq |\log x|, \quad |h''_\e(x)| \leq \frac{1}{x}, \quad \text{for } x \leq \e.
\end{equation}
Then we have pointwise convergences on any $0\leq x<X$,
\begin{equation}\label{e:pointb}
\b_\e(x) \to \b(x), \quad x \b'_\e(x) \to x \b'(x)
\end{equation}

Since $f\in L^\infty$, we  can regard $\b$ is bounded at infinity as well. With that we can plug in $\b_\e$ into \eqref{e:renorm-0}.

\begin{equation}\label{e:renorm-e}
\begin{split}
\p_t \b_\e(f) & + v \cdot \n_x \b_\e(f) + \n_v( \cA(f) \b_\e(f))  =   n \st_\rho(f \b_\e'(f) - \b_\e(f))\\
&  +\st_\rho \D_v \b_\e(f) - \st_\rho |\n_v f|^2 \b_\e''(f).
\end{split}
\end{equation}
It is now a matter of checking that the equation converges distributionally to \eqref{e:renorm-0}, as $\e \to 0$. 

Indeed, due to \eqref{e:he}, \eqref{e:pointb}, $\b_\e(f) \to \b(f)$ in $L^1_{x,v}$ for each time $t<T$ by the  dominated convergence. Moreover, by  \lem{l:LlogL}, $v \b(f), v\b_\e(f) \in  L^\infty_t L^1_{x,v}$ with the common dominant $| v\b_\e(f)| \leq |v| f |\log f|$. So, again, $v \b(f) \to v\b_\e(f) $ in $L^1_{t,x,v}$, and hence, $ v \cdot \n_x \b_\e(f)  \to  v \cdot \n_x \b(f) $ distributionally. Similarly, $f \b'(f), f\b'_\e(f) \in  L^\infty_t L^1_{x,v}$ with $f\b'_\e(f) \leq f |\log f| \in L^1$. So, we have $n \st_\rho(f \b_\e'(f) - \b_\e(f)) \to n \st_\rho(f \b'(f) - \b(f))$ in $L^1$. Similarly, $\st_\rho \D_v \b_\e(f) \to \st_\rho \D_v \b(f)$ distributionally. 

Let us look into the last diffusion term $ \st_\rho |\n_v f|^2 \b_\e''(f)$.  Due to \eqref{e:he} this term has a common integrable dominant $\st_\rho \frac{|\n_v f|^2}{f}$. We also have $\b_\e''(f) \to \b''(f)$ pointwise, where $f \neq 0$.  On the set $\{ f = 0\}$, however, $\st_\rho |\n_v f|^2 =0$ almost everywhere. So, pointwise almost everywhere $\st_\rho |\n_v f|^2 \b_\e''(f) \to  \st_\rho |\n_v f|^2 \b''(f)$. Hence, the convergence also holds in $L^1_{t,x,v}$ by the dominated convergence. 

Lastly, let us look into the alignment term $\cA(f) \b_\e(f) =  \rmw_\rho(u)\b_\e(f)  - v \st_\rho \b_\e(f) $. The convergence $v \st_\rho \b_\e(f) \to v \st_\rho \b(f)$ in $L^1$ follows from the above discussion already. Pointwise we have $\rmw_\rho(u)\b_\e(f) \to  \rmw_\rho(u)\b(f)$. Due to \eqref{e:he} , we also have $|\rmw_\rho(u)\b_\e(f)| \lesssim |\rmw_\rho(u)| f |\log f| \leq |\rmw_\rho(u)|^2 f + f|\log f|^2$, which thanks to   \lem{l:LlogL} belongs to $L^\infty_t L^1_{x,v}$. So, the convergence $\rmw_\rho(u)\b_\e(f) \to  \rmw_\rho(u)\b(f)$ holds in $L^1$ by the dominated convergence. 
\end{proof}

Let us note that since the approximation $\b_\e(f) \to \b(f)$ holds in $L^1_{x,v}$ for each time $t<T$, we have the weak formulation of \eqref{e:renorm-0} with initial and terminal times: for any $\f \in C^\infty_0([0,T) \times \domain)$,
\begin{equation}\label{e:renorm-int}
\begin{split}
\int_{\domain \times \{t\} } & \b(f) \f \dv\dx - \int_{\domain\times \{0\} } \b(f_0) \f \dv\dx - \int_0^t \int_{\domain} \b(f) \p_s \f \dv\dx \ds \\
=   \int_0^t \int_{\domain}  & \b(f)  v \cdot \n_x \f  \dv\dx \ds+  \int_0^t \int_{\domain} \cA(f) \b(f) \n_v \f  \dv\dx \ds \\
+   \int_0^t \int_{\domain} &  n \st_\rho (f \b'(f) - \b(f)) \f  \dv\dx \ds  + \int_0^t \int_{\domain}   \st_\rho   \b(f) \D_v \f  \dv\dx \ds \\
 -  \int_0^t \int_{\domain} & \st_\rho  |\n_v f|^2 \b''(f)  \f  \dv\dx \ds.
\end{split}
\end{equation}

To derive a ``$\b$-law" from \eqref{e:renorm-int}, i.e. an equation on $\ddt \int_{\domain  }  \b(f)  \dv\dx$ we must take $\f=1$. To do that carefully, we approximate $1$ by $\f_R(v) = \f(v/R)$ where $\f =1$ on the ball $|v|<1$ and $\f\in C^\infty_0(|v|<2)$. We only check the non-trivial terms, namely the alignment and diffusion. Let us start with the alignment
\[
\left|\int_0^t \int_{\domain} \cA(f) \b(f) \n_v \f_R  \dv\dx \ds \right| \leq \frac1R \int_0^t \int_{R<|v|<2R} |\cA(f)| |\b(f)|  \dv\dx \ds
\]
Since $|\cA(f)| |\b(f)| \lesssim  \jap{v} f \jap{\log f}^2 + |\rmw_\rho(u)|^2 f  \in L^1$, the integral vanishes as $R\to \infty$. Next,
\[
\left| \int_0^t \int_{\domain}   \s  \b(f) \D_v \f_R  \dv\dx \ds \right| \leq \frac{1}{R^2} \int_0^t \int_{R<|v|<2R}  R^2 |\b(f)|  \dv\dx \ds = \int_0^t \int_{R<|v|<2R} |\b(f)|  \dv\dx \ds \to 0.
\]
Finally, since $|\b''(f)| \lesssim 1/f$, $\st_\rho  |\n_v f|^2 \b''(f)\in L^1$, and hence
\[
  \int_0^t \int_{\domain}  \st_\rho  |\n_v f|^2 \b''(f)  \f_R  \dv\dx \ds \to   \int_0^t \int_{\domain}  \st_\rho  |\n_v f|^2 \b''(f)   \dv\dx \ds.
  \]
\begin{proposition}\label{p:b-law}
Under the assumptions of \prop{p:renormsharp} the integral $\int_{\domain  }  \b(f)  \dv\dx$  is absolutely continuous in time and 
\begin{equation}\label{e:b-law}
\ddt \int_{\domain  }  \b(f)  \dv\dx =  \int_{\domain}   n \st_\rho (f \b'(f) - \b(f))  \dv\dx - \int_{\domain}  \st_\rho  |\n_v f|^2 \b''(f)  \dv\dx .
\end{equation}
\end{proposition}

Now we can finally set $\b(x) = x \log x$ and  obtain from \eqref{e:b-law}
\begin{equation}\label{ }
\ddt \int_{\domain  }  f \log f  \dv\dx =  \int_{\domain}   n \st_\rho  f  \dv\dx - \int_{\domain}  \st_\rho  \frac{|\n_v f|^2}{f}   \dv\dx .
\end{equation}
Combining it with the energy equality obtained previously in \eqref{e:eneq} we obtain
\begin{equation*}\label{}
\begin{split}
\ddt \int_{\domain}\cH & =2 \int_{\domain}  n  \st_\rho  f  \dv\dx - \int_{\domain}   \st_\rho \frac{| \n_v f|^2}{f}   \dv\dx +  (u,\ave{u}_\rho)_{\k_\rho} -  \int_{\domain} \st_\rho |v|^2 f \dv \dx 
\end{split}
\end{equation*}
We can justify integration by parts to express the first integral as
\[
 \int_{\domain}  n  \st_\rho  f  \dv\dx = -  \int_{\domain}    \st_\rho v \cdot \n_v f  \dv\dx 
\]
because $\st_\rho |v || \n_v f| \leq \st_\rho |v|^2 f + \st_\rho  \frac{| \n_v f|^2}{f}  \in L^1$. So, the right hand side of the entropy law converts into the partial Fisher information and the alignment energy exactly as in \eqref{e:entropylaw}.

\subsection{Super-solutions and comparison principle}

Let us notice that the renormalization has been obtained in close association of the solution $f$ with the coefficients $\cA(f)$ and $\st_\rho$, i.e. for strictly non-linear equation. In case if  $\rmw_\rho(u) \in L^\infty(\O)$ such association is not so crucial any more, since the paring $g \cA(f)$ makes distributional sense for any $g\in L^1_\loc$.  This allows us to at least partially disassociate the solution $f$ from the coefficients and state a more general and very instrumental renormalization result for smooth perturbations of $f$.

\begin{proposition}\label{p:renorm-super} Suppose $\rmw_\rho(u) \in L^\infty(\O\times [0,T))$ and  $\chi \in C^2_\loc ( \domain \times [0,T) )$ is a sub-solution to \eqref{e:FPA}  
\begin{equation}\label{e:FPAsuper}
\p_t \chi + v \cdot \n_x \chi +\n_v \cdot (  \cA(f)\chi ) \leq \st_\rho \D_v  \chi.
\end{equation}
 Then $g = f - \chi$ is a renormalized super-solution to \eqref{e:FPA}: for any  $\b \in C^{2}$ with $\b' \geq 0$,  
\begin{equation}\label{e:renorm-super}
\p_t \b(g)  + v \cdot \n_x \b(g) + \n_v( \cA(f) \b(g))  \geq   n \st_\rho (g \b'(g) - \b(g)) + \st_\rho(\D_v  \b(g) -  |\n_v g |^2 \b''(g)).
\end{equation}
\end{proposition}
\begin{proof}
The proof is just a repeat of \prop{p:renorm} where we prerenormalize the $f$-solution exactly as before, and combine with the $\chi$-equation \eqref{e:FPAsuper} tested with $\b'( (\w f)_{\e_1} - \chi)$:
\begin{multline*}
\p_t  \b((\w f)_{\e_1} - \chi)  + [(v \cdot \n_x(\w f))_{\e_1} - v \cdot \n_x \chi ]\b'( (\w f)_{\e_1}-\chi) +  [ (\w \n_v ( \cA(f)  \n_v f ))_{\e_1} - \n_v(\cA(f) \chi)]  \b'( (\w f)_{\e_1} - \chi ) 
\\=  [ (\w \st_\rho \D_v f)_{\e_1} - \st_\rho \D_v \chi] \b'( (\w f)_{\e_1} - \chi)
+ R_{\e_1, \e_2}  \b'( (\w f)_{\e_1} - \chi ).
\end{multline*}

The analysis of the alignment force becomes easier in this case due to the assumption of boundedness of the alignment operator $\cA(f)$. 
\end{proof}

The main utility of renormalization of super-solutions for us  is to conclude that  $g_-  = - \min\{ g, 0\}$ is a sub-solution. 

\begin{corollary}\label{}
Under the assumptions of \prop{p:renorm-super}, $g_-$ is a distributional sub-solution to 
\begin{equation}\label{e:renorm-}
\p_t g_-  + v \cdot \n_x g_- + \n_v( \cA(f) g_-)  \leq \st_\rho \D_v g_-.
\end{equation}
\end{corollary}
\begin{proof}
Technically, $g_- = - \b(g)$, where $\b(x) =  x \one_{x<0}$, which is not a smooth function. To make this rigorous, we consider an approximating sequence of $C^\infty$ functions $\b_\e$ so that $\b_\e(x) = 0$ for $x \geq 0$, $\b_\e(x) = x$, for $x < - \e$, $\b_\e' \geq 0$, and so that $\b_\e'' \leq 0$ and $|x\b''_\e(x)| \to 0$ pointwise and boundedly. These conditions imply in particular that
\[
g \b_\e'(g) - \b_\e(g) \to 0,
\]
in $L^1_\loc$. The last term of the renormalized inequality \eqref{e:renorm-super} is non-negative, $-  \st_\rho|\n_v g |^2 \b_\e''(g) \geq 0$. And the rest of the terms converge distributionally by dominated convergence. So, the limiting function $\b(g)$ satisfies
\[
\p_t \b(g)  + v \cdot \n_x  \b(g) + \n_v( \cA(f)  \b(g))  \geq \st_\rho \D_v  \b(g),
\]
which after reversing the sign implies \eqref{e:renorm-}.
\end{proof}

From this we can deduce a weak maximum principle.

\begin{corollary}[Weak Comparison Principle]\label{c:wmp}
Under the assumptions of \prop{p:renorm-super}, suppose that $(f - \chi)_- \in L^1( \domain \times [0, T])$, and $f_0 \geq \chi_0$.  Then $f \geq \chi$ everywhere on $\domain \times [0, T]$.
\end{corollary}
\begin{proof} We simply need to make sure that we can integrate the renormalized equation \eqref{e:renorm-} over the entire domain. This is done in the same way as in the proof of \prop{p:b-law} by testing \eqref{e:renorm-} with a truncated function $\f_R(v)$. Since $g_-\in L^1$ all the terms vanish in the limit as $R\to 0$ leading to 
\[
\int_\domain g_-(t) \dv \dx \leq \int_\domain g_-(0) \dv \dx = 0,
\]
for all $t \leq T$. The result follows.
\end{proof}

\subsection{Proof of \prop{p:Gauss} } \label{ss:propGauss}
We now make remarks about the extension of \prop{p:Gauss} as stated to the class of weak solutions. In fact the proof presented in \cite{S-EA} goes ad verbatim as it relies mostly on weak formulation and makes use of the  weak Harnack inequality proved in \cite{GI2021} for weak solutions as well.  Let us bring the attention to  \cite[Lemma 7.8]{S-EA} which is the only spot in the proof relying on the classical comparison principle for strong solutions. We claim that 
if
\begin{equation}\label{e:fellip}
f(t,x,v) \geq \d \one_{\{ |x| < r, \ |v| <R\}}, \quad 0\leq t <\t,
\end{equation}
where $\t>0$ is some small time, then
\[
f(t,x,v) \geq \frac{\d}{4}\one_{\{ |x - t v | < r/2, \ |v| <R/2\}},
\]
for a certain much longer time period $t<t_1$. We rely on the fact that $f$ is a super-solution of \eqref{e:FPA}, and for a properly chosen $A>0$,  the barrier function
\[
\chi(t,x,v) = - At + \frac{\d}{2}  \left( \frac12 - \frac{|x-tv|^2}{r^2} - \frac{|v|^2}{R^2} \right)
\]
is a (classical) sub-solution to \eqref{e:FPA} for $t<t_1$. This makes $f - \chi$ a weak super-solution to \eqref{e:FPA}.  Moreover, $(f - \chi)_-$ is non-zero where $f < \chi$, which is only possible on the compact ellipsoid $\frac12 - \frac{|x-tv|^2}{r^2} - \frac{|v|^2}{R^2} \geq 0$, where of course both $f$ and $\chi$ are uniformly bounded. So, $(f - \chi)_- \in L^1$. \cor{c:wmp} implies $f \geq  \chi $ everywhere and this proves the needed bound from below.

\subsection{Propagation of higher moments}
As another application of the renormalization we show propagation of higher moments. Note that if the initial condition $f_0$ belongs to $L^p_m$ for some $p<\infty$, $m>q$, then the existence part of \thm{t:weak} guarantees that at least one solution starting from $f_0$ will have finite  moment $L^p_m$. However, if the data is vacuous it is not known whether it will coincide with the given "old" solution in $L^1_q \cap L^\infty$. One can show that at least for protocols of $(2,\infty)$-type, the higher moment is propagated by any solution starting from such data.

\begin{lemma}\label{l:moments}
Let $\cC$ be a regular local protocol of type $(2,\infty)$. Then any weak solution in the class \eqref{e:regweak} with $f_0 \in L^p_m$, for $1\leq p <\infty$,  $m > 0$, will propagate its moment, $f\in L^\infty([0,T); L^p_m)$, for any $T>0$.
\end{lemma}

We will make use of tapered moments as a way to approximate true moments which are not a priori known to be a bounded.  We consider
\begin{equation}\label{e:weights}
\w_{m,R} (v) = \jap{v}^m  \jap{v/R}^{-Q}, \quad R >1, \ Q > m + n +1.
\end{equation}
 The important feature of this particular tapered moment is that it  controls its higher weighted gradients, it's hierarchial-in-$m$, and it has doubling property, all independent of $R$:
\begin{align}
(1+|v|^k) |\n^k_v \o_{m,R}(v) | & \leq C_k \o_{m,R}(v), \quad k\in \N, \label{e:whigh} \\
 |\n^k_v \w_{m,R}(v) | & \leq C_k \o_{m-k,R}(v), \quad k\in \N, \label{e:whier}\\
c \leq \frac{\w_{m,R}(v')}{\w_{m,R}(v'')} \leq C,& \text{ whenever } \frac12 \leq \frac{|v'|}{|v''|} \leq 2. \label{e:wdouble}
\end{align}

\begin{proof}
Before we begin we make one observation.  Since weak solution in the class stated in \eqref{e:regweak} have integrability of $\st_\rho \n_v f$ one can move one derivate from $\D_v \phi$ in the definition of weak solutions \eqref{e:weak} back on $f$. Then since all the remaining test function appear in derivative of first order against integrable functions, by elementary approximation, the equation \eqref{e:weak} extends to locally Lipschitz bounded test functions. By this observation we can test with an algebraically decaying in $v$ functions.

So, we will start from the renormalization \eqref{e:renorm} with $\b(x) = x^p$, and test it with $\f = \w_{m,R}$ (note that $\b$ can be tapered at infinity without altering the composition $\b(f)$ thanks to $f\in L^\infty$).  Denote
\[
j_{p,m,R} = \int_{\domain}   \w_{m,R}  f^p  \dv\dx .
\]
We obtain from  the boundedness of $\rmw_\rho$ and \eqref{e:whier} - \eqref{e:whigh},
\begin{equation*}\label{}
\begin{split}
j_{p,m,R} (t) - j_{p,m,R} (0)  = &    \int_0^t \int_{\domain} \st_\rho  f^p  \D_v \w_{m,R} \dv \dx \ds  - p \int_0^t \int_{\domain} \st_\rho   f^p v  \cdot\n_v   \w_{m,R} \dv \dx \ds \\
& - p(p-1) \int_0^t  \int_{\domain} \st_\rho |\n_v f|^2 f^{p-2}  \w_{m,R} \dv\dx \ds \\
&+ (p-1) \int_0^t  \int_{\domain} \st_\rho  f^{p}  \n_v\cdot (v \w_{m,R}) \dv\dx \ds\\
&  + \int_0^t  \int_{\domain}  f^p \rmw_\rho \cdot \n_v  \w_{m,R}   \dv\dx \ds \\
& \lesssim \int_0^t j_{p,m,R} \ds  - p(p-1) \int_0^t  \int_{\domain} \st_\rho |\n_v f|^2 f^{p-2}  \w_{m,R} \dv\dx \ds
\end{split}
\end{equation*}

Ignoring the dissipation term for a moment we obtain by the \GL, $j_{p,m,R} (t) \leq j_{p,m,R}(0) e^{ct}$. Letting $R \to \infty$ we obtain by monotone convergence,
\[
 \int_{\domain}  \jap{v}^m f^p  \dv\dx =\lim_{R \to \infty} j_{p,m,R} \leq C.
\]
\end{proof}

\section{Global hypoellipticity} \label{s:hypoell}

We now address parabolic regularization for weak solutions and provide a proof of \prop{p:instantreg}. According to \prop{p:Gauss} weak solutions gain positivity instantly, and thanks to regularity \ref{i:r2}, \eqref{e:sreg1} all the coefficients become smooth and the diffusion coefficient $\s$ becomes bounded away from zero by the locality \eqref{e:stbelow}. From this point on, no specific structure of the coefficients is necessary, nor their association with $f$ itself. So, we will address regularization in the context of a  linear Fokker-Planck equation
\begin{equation}\label{e:FPgen}
\p_t f + v \cdot \n_x f = \n_v (\rmA \n_v f) + \n_v \cdot (\rmb f).
\end{equation}
Here, $\rmA = \rmA(x,v,t)\in \R^{n} \times \R^n$ is a given matrix, and $\rmb = \rmb(x,v,t) \in \R^n$ is a field satisfying
\begin{equation}\label{e:Ab1}
\l \I \leq \rmA(x,v,t) \leq \L \I,  \quad (x,v,t) \in \domain\times [0,T)
\end{equation}
and for any multi-indeces $\bk,\bl \geq 0$,
\begin{equation}\label{e:Ab2}
\| \p^{\bk}_{x} \p^{\bl}_v \rmA\|_\infty < \infty, \quad   \| \jap{v}^{-1} \p^{\bk}_{x}  \rmb\|_\infty + \| \p^{\bk}_{x} \p^{\bl + 1}_v \rmb\|_\infty  <\infty.
\end{equation}
Thus, the drift $\rmb$ can grow  linearly at $v$-infinity unless at least one derivative in $v$ is applied.

First let us denote weighted semi-norms by
\[
\rmh^{k,l}_q(f) = \sum_{|\bk| = k, |\bl| = l}  \int_\domain  \jap{v}^q  |\p^{\bk}_{x} \p^{\bl}_v f |^2 \dv\dx.
\]
According to our definition of weighted Sobolev spaces \eqref{e:Sobdef}, we can write 
\begin{equation}\label{e:Sobdef}
H^{m}_q(\domain) =  \left\{ f :  \rmh(f) = \sum_{2 k + l \leq 2m}   \rmh^{k,l}_{q - 2k - l}(f) <\infty \right\}.
\end{equation}
It is worth noting that the regressive weights are only necessary for models with the $v$-growing part of $b$ dependent on $x$, see \cite{Villani-hypo} for the homogeneous case.

\begin{proposition}\label{p:reg}
Let $f\in L^\infty([0,T); L^1\cap L^2) \cap C([0,T); \cD')$ be a weak solution to \eqref{e:FPgen} satisfying \eqref{e:Ab1} - \eqref{e:Ab2}, and with initial condition in class $f_0 \in L^2_q$, for some $q \geq 0$. Then, $f \in L^\infty([0,T); L^2_q)$, and for any $m \in \N$, there exist  constants $C_m,\k>0$ such that 
\begin{equation}\label{e:reg}
\|f(t) \|_{H^m_q} + \|\p_t f(t)\|_{H^{m-1}_{q-3}}\leq \frac{C_m}{t^\k}.
\end{equation}
\end{proposition}

Clearly, \prop{p:reg} implies \prop{p:instantreg}. 
The proof of \prop{p:reg} will be performed in two stages. The first is devoted to reduction to the case when $f$ is already smooth and decaying sufficiently fast for all $t\geq 0$.  Once $f$ is regular, we perform a priori estimates and obtain the required bound independent of the regularity of the initial condition. 

The proof is in the spirit of global hypoellipticity methods summarized for the classical Fokker-Planck case, for example,  in \cite{Villani-hypo}. The advantage of this method, as opposed to local  methods such as Di Giorgi, is that it puts the solution directly into a  weighted space, thus controlling the growth of derivatives by design. Such control is crucial for the relaxation analysis in particular. 

We first address the classical energy  inequality and stability of weak solutions.

\begin{lemma} Every weak solution $f\in L^\infty([0,T); L^1\cap L^2) \cap C([0,T); \cD')$ automatically belongs to the class $\n_v f \in L^2_{t,x,v}$  and satisfies the following energy inequality
\begin{equation}\label{e:einAb}
\|f(t)\|_2^2 + \frac{\l}{2} \int_0^t \|\n_v f\|_2^2 \ds \leq \|f(0)\|_2^2 + c  \int_0^t \| f\|_2^2 \ds.
\end{equation}
\end{lemma}
\begin{proof}
 Let us mollify the equation
 \begin{equation}\label{e:FPmol}
\p_t f_\e + (v \cdot \n_x f)_\e = \n_v (\rmA \n_v f)_\e + \n_v \cdot (\rmb f)_\e,
\end{equation}
where all the gradients inside the mollifications are viewed as distributional.  Let us test with $f_\e \w_{0,R}$, with the weight $\w_{0,R}$ defined in \eqref{e:weights} being used as a temporary measure to control the growth of the drift $\rmb$:
\begin{multline*}
\frac12 \ddt  \int_\domain | f_\e |^2 \w_{0,R} \dv \dx+ \int_\domain (v \cdot \n_x f)_\e f_\e \w_{0,R} \dv \dx \\= - \int_\domain (\rmA \n_v f)_\e \n_v( f_\e \w_{0,R}) \dv \dx-  \int_\domain(\rmb f)_\e \cdot \n_v (f_\e \w_{0,R}) \dv\dx
\end{multline*}
For the transport term on the left, we use a cancellation and write
\begin{multline*}
\int_\domain (v \cdot \n_x f)_\e f_\e \w_{q,R} \dv \dx = \int_\domain [(v \cdot \n_x f)_\e - v \cdot \n_x f_\e] f_\e \w_{q,R} \dv \dx\\
 \leq  \int_\domain | f_\e |^2 \w_{0,R} \dv \dx +  \int_\domain |(v \cdot \n_x f)_\e - v \cdot \n_x f_\e|^2 \w_{0,R} \dv \dx.
\end{multline*}
We have
\[
(v \cdot \n_x f)_\e - v \cdot \n_x f_\e = - \int_\domain w \cdot \n_y \chi_\e(y,w) f(x-y,v-w)\dw \dy, 
 \]
 which represents another mollification with some $\tilde{\chi}_\e$, $\tilde{f}_\e = f \ast \tilde{\chi}_\e$. Thus, by \lem{l:wappr},
 \[
 \int_\domain |(v \cdot \n_x f)_\e - v \cdot \n_x f_\e|^2 \w_{0,R} \dv \dx \lesssim   \int_\domain | \tilde{f}_\e |^2 \w_{0,R} \dv \dx \leq \|f\|_2^2.
 \]
 
 Moving to the dissipation term,
 \begin{equation*}\label{}
\begin{split}
- &\int_\domain (\rmA \n_v f)_\e \n_v( f_\e \w_{0,R}) \dv \dx = -\int_\domain (\rmA \n_v f)_\e \n_v f_\e \w_{0,R} \dv \dx + \int_\domain (\rmA \n_v f)_\e f_\e \n_v \w_{0,R} \dv \dx \\
& \leq - \int_\domain [(\rmA \n_v f)_\e - \rmA \n_v f_\e] \n_v f_\e \w_{0,R} \dv \dx - \int_\domain \rmA \n_v f_\e \n_v f_\e \w_{0,R} \dv \dx \\
& +  \int_\domain  |\n_v f_\e| f_\e \w_{0,R} \dv \dx\\
& \leq - \frac{\l}{2} \int_\domain  |\n_v f_\e|^2 \w_{0,R} \dv \dx +  \|f\|_2^2 +\int_\domain | (\rmA \n_v f)_\e - \rmA \n_v f_\e|^2 \w_{0,R} \dv \dx.
\end{split}
\end{equation*}
 But,
\[
 (\rmA \n_v f)_\e - \rmA \n_v f_\e  = \int_\domain  [\rmA(x-y,v-w) - \rmA(x,v)] f(x-y,v-w) \n_w \chi_\e(y,w) \dw \dy.
\]
 Since $|\rmA(x-y,v-w) - \rmA(x,v)| \lesssim |y| + |w|$, the latter is bounded by $\tilde{f}_\e$, hence 
 \[
 \int_\domain | (\rmA \n_v f)_\e - \rmA \n_v f_\e|^2 \w_{0,R} \dv \dx \lesssim  \|f\|_2^2 .
 \]
 
 Lastly, for the drift term we have
\begin{equation*}\label{}
\begin{split}
 \int_\domain(\rmb f)_\e \cdot \n_v (f_\e \w_{0,R}) \dv\dx & =  \int_\domain(\rmb f)_\e \cdot \n_v f_\e \w_{0,R} \dv\dx +  \int_\domain(\rmb f)_\e f_\e  \n_v \w_{0,R} \dv\dx\\
 & \leq  \int_\domain [ (\rmb f)_\e - \rmb f_\e] \cdot \n_v f_\e \w_{0,R} \dv\dx +  \int_\domain \rmb f_\e \cdot \n_v f_\e \w_{0,R} \dv\dx \\
 &+  \int_\domain(\rmb f)_\e f_\e  \w_{-1,R} \dv\dx\\
 & \lesssim   \int_\domain | (\rmb f)_\e - \rmb f_\e |^2  \w_{0,R} \dv\dx + \frac{\l}{4}  \int_\domain  |\n_v f_\e|^2 \w_{0,R} \dv \dx\\
 &  - \frac12  \int_\domain  \n_v \cdot ( \rmb  \w_{0,R} ) f^2_\e \dv\dx + \|f\|_2^2.
\end{split}
\end{equation*}
 Since $ \n_v \cdot ( \rmb  \w_{0,R} )  \leq C$, the penultimate term is also bounded by $\|f\|_2^2$. Let us examine the last commutator,
 \begin{equation*}\label{}
\begin{split}
 | (\rmb f)_\e - \rmb f_\e | & =  \left| \int_\domain  [\rmb(x-y,v-w) - \rmb(x,v)] f(x-y,v-w)  \chi_\e(y,w) \dw \dy \right| \\
 & \leq  \int_\domain ( |v| |y| + |w|) | \chi_\e(y,w)| f(x-y,v-w) \dw \dy \leq \e \jap{v} \tilde{f}_\e.
\end{split}
\end{equation*}
Thus, for the drift term we have
\[
 \int_\domain(\rmb f)_\e \cdot \n_v (f_\e \w_{0,R}) \dv\dx \leq  \frac{\l}{4}  \int_\domain  |\n_v f_\e|^2 \w_{0,R} \dv \dx +  \|f\|_2^2 + \e \int_\domain  |\tilde{f}_\e|^2 \w_{1,R} \dv \dx.
\]
Additionally, notice that 
\[
 \e \int_\domain  |\tilde{f}_\e|^2 \w_{1,R} \dv \dx \leq \e R \int_\domain  |\tilde{f}_\e|^2 \jap{v/R}^{-Q+1} \dv \dx
 \leq  \e R \|f\|_2^2 .
 \]
 
 Putting together the obtained estimates we arrive at
\begin{equation}\label{e:eineR}
\frac12 \ddt  \int_\domain | f_\e |^2 \w_{0,R} \dv \dx \leq - \frac{\l}{4}  \int_\domain  |\n_v f_\e|^2 \w_{0,R} \dv \dx + (1 + \e R) \|f\|_2^2.
\end{equation}
Thus,
\[
 \int_\domain | f_\e(t)  |^2 \w_{0,R} \dv \dx +  \frac{\l}{4} \int_0^t \int_\domain  |\n_v f_\e|^2 \w_{0,R} \dv \dx \leq Ct(1+ \e R) +  \int_\domain | f_\e(0)  |^2 \w_{0,R} \dv \dx \lesssim 1 + \e R.
 \]
For a fixed $R$, let us let $\e\to 0$. We obtain from uniform boundedness of the sequence $\n_v f_\e$ in the weighted $L^2$ space that the limit $\n_v f$ belongs to the same space and satisfies
$\int_0^t  \int_\domain  |\n_v f |^2 \w_{0,R} \dv \dx \lesssim 1$,
uniformly in $R$. Letting $R \to \infty$, by Fatou's Lemma, we conclude that $
\int_0^t  \int_\domain  |\n_v f |^2  \dv \dx \lesssim 1$.
 
 Going back to \eqref{e:eineR} integrating in time, we can them pass to the limits as $\e \to 0$ and $R\to \infty $ consecutively, to obtain \eqref{e:einAb}.
\end{proof}

\begin{corollary} Weak solutions to \eqref{e:FPgen} are unique in class $ L^\infty([0,T); L^1\cap L^2) \cap C([0,T); \cD')$, and for any two solutions $f,g$ one has
\begin{equation}\label{e:contr}
\sup_{t<T} \| f- g \|_2 \leq C_T \|f_0 - g_0 \|_2.
\end{equation}
\end{corollary}

 In order to prove \prop{p:reg} it would be convenient to know a priori that $f$ is already smooth just so we can differentiate the right hand side of the equation.  To do that we need to have an approximated sequence of smooth solutions and ensure that the estimates on that sequence obey the bound \eqref{e:reg} independent of approximation parameter. So, let us consider a sequence of smoothed and compactly supported initial data $f^{(n)}_0 \to f_0$ in $L^1 \cap L^2$. For such data the equation \eqref{e:FPgen} is globally well-posed in any $H^m_q$-space. This follows from the standard analysis. Indeed,  the regularized equation
 \begin{equation}\label{e:FPgenR}
\p_t f + v \cdot \n_x f = \n_v (\rmA \n_v f) + \frac{1}{R} \D_x f + \n_v \cdot (\w_{0,R} \rmb f),
\end{equation}
 is a generator of an analytic $C_0$-semigroup in any classical $H^s$-space, and solutions from compactly supported data decay at infinity exponentially fast as follows from the bounds on the Green's function, see for example \cite{Krylov-book}. Thus, the solutions will remain in any $H^M_Q$ at all times. The fact that they remain in $H^M_Q$ uniformly in $R$ on any finite time interval $[0,T)$ will follow from our Sobolev estimate \eqref{e:hd} performed similarly  on \eqref{e:FPgenR}.  The classical compactness argument produces a solution to  \eqref{e:FPgen} in the class $f^{(n)} \in L^\infty H^M_Q$, $\n_v f^{(n)} \in L^2 H^M_Q$ on any $[0,T)$.  If we prove \eqref{e:reg} uniformly in $n$, i.e. independent of regularity of the initial data, then by the contraction bound \eqref{e:contr}, $f^{(n)} \to f$ in $L^2$, and by weak lower semi-continuity $\|f\|_{H^m_q} \leq \liminf_n \|f^{(n)}\|_{H^m_q}$, and \prop{p:reg} would be proved.
 
In conclusion, while proving \prop{p:reg}we can assume a priori that the solution belongs to $f \in H^M_Q$, for $M \gg m$ and $Q \gg q$.

 \begin{proof}[Proof of \prop{p:reg}] By the above reasoning, we can formally assume that $f \in H^M_Q$, for $M \gg m$ and $Q \gg q$, so that all the estimates below are justified.
 
Let us introduce the dissipation terms
\[
\rmd^{k,l}_q = \rmh^{k,l+1}_{q}, \quad \rmd = \sum_{2 k + l \leq 2m}   \rmd^{k,l}_{q - 2k - l},
\]

In what follows, all the constants $c_1,c_2,\ldots$ depend only on the regularity of $\rmA, \rmb$.

\begin{lemma} \label{l:n} We have 
\begin{equation}\label{ }
\dot{\rmh}^{0,0}_q \leq -c_1 \rmd^{0,0}_q + c_2 \rmh^{0,0}_q.
\end{equation}
Consequently, the $L^2_q$-momentum $\rmh^{0,0}_q$ remains bounded on $[0,T)$.
\end{lemma}
\begin{proof}  The proof is trivial by testing the equation with $f \jap{v}^q$.
\end{proof}

\begin{lemma} We have 
\begin{equation}\label{e:hd}
\dot{\rmh} \leq -c_3 \rmd + c_4 \rmh.
\end{equation}
\end{lemma}
\begin{proof}  
 Let us fix  a pair of indexes $2k + l \leq 2m$, apply $  \p_x^\bk \p_v^\bl  $ to \eqref{e:FPgen} and test with  $\p_x^\bk \p_v^\bl f \jap{v}^{q - 2k - l}$. We have
\begin{multline*}
\dot{\rmh}^{k,l}_{q - 2k - l} + \int_\domain  \p_x^\bk \p_v^\bl( v \cdot \n_x f) \p_x^\bk \p_v^\bl f \jap{v}^{q - 2k - l} \dv \dx \\
 = -  \int_\domain   \p_x^\bk \p_v^\bl   (\rmA \n_v f) \cdot \n_v ( \p_x^\bk \p_v^\bl f  \jap{v}^{q - 2k - l} ) \dv \dx  -   \int_\domain    \p_x^\bk \p_v^\bl (\rmb f) \cdot \n_v  ( \p_x^\bk \p_v^\bl f  \jap{v}^{q - 2k - l} )  \dv \dx.
\end{multline*}
Let us denote the terms according to their placement in the equation above,
\[
\dot{\rmh}^{k,l}_{q - 2k - l} + X =  D + B.
\]

Let us start with the dissipative term $D$, denoting the commutator
\[
[ \p_x^\bk \p_v^\bl, \rmA \n_v] f =   \p_x^\bk \p_v^\bl   (\rmA \n_v f) -   \rmA \n_v  ( \p_x^\bk \p_v^\bl f) ,
\]

\begin{equation*}\label{}
\begin{split}
D = & -  \int_\domain    \rmA \n_v  ( \p_x^\bk \p_v^\bl f) \cdot \n_v ( \p_x^\bk \p_v^\bl f  \jap{v}^{q - 2k - l} ) \dv \dx  - \int_\domain  [ \p_x^\bk \p_v^\bl, \rmA \n_v] f  \cdot \n_v ( \p_x^\bk \p_v^\bl f  \jap{v}^{q - 2k - l} ) \dv \dx\\
 = &  -  \int_\domain    \rmA \n_v  ( \p_x^\bk \p_v^\bl f) \cdot \n_v ( \p_x^\bk \p_v^\bl f ) \jap{v}^{q - 2k - l} \dv \dx  -  \int_\domain    \rmA \n_v  ( \p_x^\bk \p_v^\bl f) \cdot  \p_x^\bk \p_v^\bl f \n_v  \jap{v}^{q - 2k - l}  \dv \dx \\  
 & -   \int_\domain  [ \p_x^\bk \p_v^\bl, \rmA \n_v] f  \cdot \n_v ( \p_x^\bk \p_v^\bl f )  \jap{v}^{q - 2k - l}  \dv \dx - \int_\domain  [ \p_x^\bk \p_v^\bl, \rmA \n_v] f  \cdot \p_x^\bk \p_v^\bl f \n_v  \jap{v}^{q - 2k - l}  \dv \dx\\
 = & :  D_1 + D_2 + D_3 + D_4. 
\end{split}
\end{equation*}
Thanks to \eqref{e:Ab1}, we have
\[
D_1 \leq - \rmd^{k,l}_{q - 2k - l}.
\]
Continuing to $D_2$,
\[
|D_2| \lesssim \int_\domain | \n_v  ( \p_x^\bk \p_v^\bl f) | | \p_x^\bk \p_v^\bl f |  \jap{v}^{q - 2k - l - 1}  \dv \dx \leq \e \rmd^{k,l}_{q - 2k - l} + c \rmh^{k,l}_{q - 2k - l},
\]
where $\e>0$ is small to be determined later. As to the commutator terms, let us expand
\[
[ \p_x^\bk \p_v^\bl, \rmA \n_v] f = \sum_{\substack{\bk' \leq \bk, \bl' \leq \bl \\ |\bk'|+|\bl'| < k + l}} \p_x^{\bk - \bk'} \p_v^{\bl - \bl'} \rmA \n_v   \p_x^{\bk'} \p_v^{\bl'}   f ,
\]
hence,
\[
 \int_\domain | [ \p_x^\bk \p_v^\bl, \rmA \n_v] f |^2 \jap{v}^{q - 2k - l}  \dv \dx \lesssim \sum_{\substack{\bk' \leq \bk, \bl' \leq \bl \\ |\bk'|+|\bl'| < k + l}} \int_\domain  |  \p_x^{\bk'} \p_v^{\bl' + 1}   f |^2  \jap{v}^{q - 2k - l}    \dv \dx
 \]
 In the sum above, if $l' < l$, then all the derivative are of order lower than $\p_x^\bk \p_v$, and consequently, all these terms are bounded by $\rmh$. For $l' = l$, we have $k'<k$, and thus, $2k'+l+1 \leq 2k +l -1<2m$. So, the indexes are within the allowed range, and as for the weight can be increased to accommodate for lower order derivative,  $\jap{v}^{q - 2k - l}\leq  \jap{v}^{q - 2k' - l-1}$. So, again, this term is bounded by  $\rmh$. Thus,
 \[
  \int_\domain | [ \p_x^\bk \p_v^\bl, \rmA \n_v] f |^2 \jap{v}^{q - 2k - l}  \dv \dx \lesssim \rmh,
 \]
and consequently,
\[
|D_3| \leq \e \rmd^{k,l}_{q - 2k - l} + c \rmh.
\]
With the obtained estimate on the commutator $D_4$ is even of smaller order,
\[
|D_4| \lesssim \rmh.
\]

In conclusion,
\[
D \leq - (1-\e) \rmd^{k,l}_{q - 2k - l} + c \rmh.
\]

Moving to the $B$-term, we write
\begin{equation*}\label{}
\begin{split}
-B =&    \int_\domain    \p_x^\bk \p_v^\bl (\rmb f) \cdot \n_v  ( \p_x^\bk \p_v^\bl f  \jap{v}^{q - 2k - l} )  \dv \dx\\
= &    \int_\domain    \p_x^\bk \p_v^\bl (\rmb f) \cdot \n_v  ( \p_x^\bk \p_v^\bl f )  \jap{v}^{q - 2k - l}  \dv \dx +     \int_\domain    \p_x^\bk \p_v^\bl (\rmb f) \cdot  \p_x^\bk \p_v^\bl f \n_v  ( \jap{v}^{q - 2k - l} )  \dv \dx.
\end{split}
\end{equation*}
Notice that 
\[
| \p_x^\bk \p_v^\bl (\rmb f)\n_v  ( \jap{v}^{q - 2k - l} ) | \lesssim \jap{v}^{q - 2k - l}  \sum_{\bk' \leq \bk, \bl' \leq \bl} |\p_x^{\bk'} \p_v^{\bl'} f|,
\]
and therefore, the second term is bounded by $\rmh$.  As to the first term, we write 
\begin{multline*}
 \int_\domain    \p_x^\bk \p_v^\bl (\rmb f) \cdot \n_v  ( \p_x^\bk \p_v^\bl f )  \jap{v}^{q - 2k - l}  \dv \dx =  \int_\domain     \rmb \p_x^\bk \p_v^\bl f \cdot \n_v  ( \p_x^\bk \p_v^\bl f )  \jap{v}^{q - 2k - l}  \dv \dx \\
 +  \int_\domain  [ \p_x^\bk \p_v^\bl (\rmb f) -   \rmb \p_x^\bk \p_v^\bl f ] \cdot \n_v  ( \p_x^\bk \p_v^\bl f )  \jap{v}^{q - 2k - l}  \dv \dx : = B_1 + B_2.
\end{multline*}
After integration by parts in $B_1$,
\[
B_1 = - \frac12 \int_\domain   \n_v \cdot  ( \jap{v}^{q - 2k - l}  \rmb)  |  \p_x^\bk \p_v^\bl f |^2   \dv \dx \leq \rmh.
\]
For $B_2$, we need to make use of regressive weights. Let us expand the commutator,
\[
 [ \p_x^\bk \p_v^\bl (\rmb f) -   \rmb \p_x^\bk \p_v^\bl f ] = \sum_{\substack{\bk' \leq \bk, \bl' \leq \bl \\ |\bk'|+|\bl'| < k + l}} \p_x^{\bk - \bk'} \p_v^{\bl - \bl'} \rmb \  \p_x^{\bk'} \p_v^{\bl'}   f .
\]
The part where $l'<l$ involves globally bounded coefficients $ \p_x^{\bk - \bk'} \p_v^{\bl - \bl'} \rmb$, and hence this part gives rise to the term
\[
\lesssim  \int_\domain |  \p_x^{\bk'} \p_v^{\bl'}   f | | \n_v  ( \p_x^\bk \p_v^\bl f ) | \jap{v}^{q - 2k - l}  \dv \dx \leq \e \rmd^{k,l}_{q - 2k - l} + c \rmh.
\]
In the case $\bl' = \bl$, $\bk'<\bk$, we have
\[
| \p_x^{\bk - \bk'}  \rmb \  \p_x^{\bk'} \p_v^{\bl}   f | \lesssim \jap{v} | \p_x^{\bk'} \p_v^{\bl}   f |,
 \]
and the corresponding integral is bounded by 
\[
\lesssim  \int_\domain |  \p_x^{\bk'} \p_v^{\bl}   f | | \n_v  ( \p_x^\bk \p_v^\bl f ) | \jap{v}^{q - 2k - l + 1}  \dv \dx \leq \e \rmd^{k,l}_{q - 2k - l} +  \int_\domain |  \p_x^{\bk'} \p_v^{\bl}   f |^2 \jap{v}^{q - 2(k-1) - l}  \dv \dx.
\]
Since $|\bk'| \leq k-1$, the weight is consistent with the order of the derivative, and hence the last term is bounded by $\rmh$.

In conclusion,
\[
B \leq 2\e \rmd^{k,l}_{q - 2k - l} + c \rmh.
\]

Lastly, the transport term $X$. Canceling the component when all derivatives fall on $f$, we have
\[
X = \int_\domain [ \p_x^\bk \p_v^\bl, v\cdot \n_x] f \, \p_x^\bk \p_v^\bl f \jap{v}^{q - 2k - l} \dv \dx.
\]
Let us expand the commutator. Here we isolate three cases, $l =0,1$ and $l  \geq 2$. When $l=0$, then obviously,
$[ \p_x^\bk,  v\cdot \n_x] = 0$. When $l=1$, we have
\[
[ \p_x^\bk \p_v, v\cdot \n_x] f =\p_x^{ \bk +1}  f. 
\]
In this case,

\begin{equation*}\label{}
\begin{split}
X = &\int_\domain \jap{v}^{q-2k-1}  \p_x^{\bk+1} f \p_x^{\bk} \p_v f  \dv\dx = - \int_\domain \jap{v}^{q- 2k-1}  \p_x^{\bk+1} \p_v  f \p_x^{\bk} f  \dv\dx \\
&- \underbrace{ \int_\domain \p_v \jap{v}^{q- 2k-1}  \p_x^{\bk+1}  f \p_x^{\bk} f  \dv\dx}_{=0} \\
 = & - \int_\domain \jap{v}^{\frac{q- 2k}{2}-1}  \p_x^{\bk+1} \p_v  f   \jap{v}^{\frac{q- 2k}{2}} \p_x^{\bk} f  \dv\dx \leq \e  \rmd^{k+1,0}_{q - 2 k -2} + c_1 \rmh^{k,0}_{q - 2k} \leq \e  \rmd^{k+1,0}_{q - 2k -2} + c_1 \rmh.
\end{split}
\end{equation*} 
Notice that since $2k+1 < 2m$, then $2(k +1) \leq 2m$, and hence the dissipative term is within the allowed range of parameters.  

Let us consider the case when $l  \geq 2$.  Then
\[
X = \int_\domain \jap{v}^{q-2k-l}  \p_x^{\bk+1} \p_v^{\bl-1} f \p_x^{\bk} \p^\bl_v f  \dv\dx.
\]
Since $l \geq 2$, we can relieve the first $f$-component of one more derivative in $v$:
\begin{multline*}
X= -  \int_\domain \jap{v}^{q- 2k-l}  \p_x^{\bk+1} \p_v^{\bl-2} f \p_x^{\bk} \p^{\bl+1}_v f  \dv\dx -  \int_\domain \p_v \jap{v}^{q-2k-l}  \p_x^{\bk+1} \p_v^{\bl-2} f \p_x^{\bk} \p^\bl_v f  \dv\dx \\
\lesssim \rmh^{k+1,l-2}_{q-2k-l} + \d \rmd^{k,l}_{q-2k-l} +  \rmh^{k,l}_{q-2k-l}.
\end{multline*}
Note that $2(k+1) + l-2 = 2k + l \leq 2m$, so the indexes are within the allowed range. Thus,
\[
X \leq  \e \rmd^{k,l}_{q-2k-l} +  c \rmh.
\]

Summing up over all the indexes with $2k + l \leq m$, we obtain
\[
\dot{\rmh} \leq -(1 - c\e) \rmd + c \rmh,
\]
which for small $\d>0$ proves the result.  
 
\end{proof}

Next let us introduce the mixed terms
\[
\rmm(f) = \sum_{1\leq k \leq m} \rmm^{k}_{q - 2k}(f), \quad \rmm^{k}_{q-2k}(f) = \sum_{|\bk| = k}  \int_\domain  \jap{v}^{q-2k}  \p^{\bk}_{x} f  \cdot  \p^{\bk-1}_{x} \p_v f  \dv\dx, 
\]
where the convention for $\p^{\bk-1}_{x} \p_v$ is to use the same index for $-1$ in $x$ as $1$ for $v$.
We further split $\rmh$ into $x$-, $v$-, and $0$-order components
\begin{equation*}\label{}
\begin{split}
\rmh &= \rmx + \rmv + \rmn,\\
\rmx = \sum_{1\leq k\leq m}   \rmh^{k,0}_{q - 2k}(f), \quad \rmv & =   \sum_{\substack{2 k + l \leq 2m \\ l \geq 1}}   \rmh^{k,l}_{q - 2k - l}(f), \quad \rmn  = \rmh^{0,0}_q.
\end{split}
\end{equation*}
By \lem{l:n} the $0$-order term  $\rmn$ remains bounded on the entire interval $[0,T)$. For the kinetic term $\rmv$ we have the following bound
\begin{equation}\label{e:vlessd}
\rmv \leq \rmd.
\end{equation}
Indeed,
\[
\rmv =  \sum_{2 k + l' \leq 2m - 1}   \rmh^{k,l'+1}_{q - 2k - l' -1} \leq  \sum_{2 k + l' \leq 2m - 1}   \rmh^{k,l'+1}_{q - 2k - l' } = \sum_{2 k + l' \leq 2m - 1}   \rmd^{k,l'}_{q - 2k - l' } \leq \rmd.
\]

\begin{lemma}\label{ }
We have 
\begin{equation}\label{ }
\dot{\rmm} \leq - c_3 \rmx + c_4 \rmv + c_5 \rmv^{1/2} \rmd^{1/2} + c_6,
\end{equation}
for some $c_3,c_4,c_5,c_6 >0$ depending only on the regularity of $\rmA,\rmb$.
\end{lemma}
\begin{proof}
We start by observing that the total contribution of the $x$-transport term in the equation for $\rmm^{k}_{q - 2k}$ is the negative term $ - \rmh^{k,0}_{q - 2k}$. The contribution of the dissipation and the $\rmb$-drift will be denoted $D_{k,q}$ and $B_{k,q}$, respectively:
\[
\dot{\rmm}^{k}_{q - 2k} = - \rmh^{k,0}_{q - 2k} +D_{k,q} + B_{k,q}.
\]
Thus, summing over $1\leq k\leq m$, we obtain
\begin{equation}\label{e:mDB}
\dot{\rmm} = - \rmx + D + B.
\end{equation}

The rest of the proof is devoted to estimating the remainder terms $D,B$. 

So, we fix $1\leq k\leq m$, and start with the dissipation term
\[
\begin{split}
D_{k,q} &= \int_\domain \jap{v}^{q-2k} \p_x^{\bk}( \rmA \n_v f)  \cdot  \p^{\bk-1}_{x} \p_v \n_v f  \dv\dx +   \int_\domain \n_v \jap{v}^{q-2k} \p_x^{\bk}( \rmA \n_v f)  \cdot  \p^{\bk-1}_{x} \p_v f  \dv\dx \\
&+ \int_\domain \n_v \jap{v}^{q-2k}  \p^{\bk}_{x}  f  \cdot  \p^{\bk-1}_{x} \p_v (\rmA \n_v f)  \dv\dx + \int_\domain \jap{v}^{q-2k} \n_v  \p^{\bk}_{x}  f  \cdot  \p^{\bk-1}_{x} \p_v (\rmA \n_v f)  \dv\dx\\
& : = D^1 + D^2 + D^3 + D^4.
\end{split}
\]
In $D^1$, the second integrand $ \p^{\bk-1}_{x} \p_v \n_v f $ has the indexing $2(k-1) + 2 = 2 k$, which is consistent with the power of the weight. So, expanding the $\p_x^\bk$ partials in the first integrant, we naturally single out the all-on-$f$ term $\rmA \n_v \p_x^{\bk} f$, which results in a bound $(\rmd_{q-2k}^{k,0})^{1/2} \rmv^{1/2} \leq  \rmd^{1/2} \rmv^{1/2}$. The rest is bounded by $\rmv$
\[
D^1 \lesssim \rmd^{1/2} \rmv^{1/2} + \rmv.
\]
Next, the term $D^2$ is of lower order and is bounded by, after moving one derivative in $x$ on $\p^{\bk-1}_{x} \p_v f $,
\[
D^2 \leq \int_\domain  \jap{v}^{q-2k - 1} |\p_x^{\bk -1}( \rmA \n_v f) | | \p^{\bk}_{x} \p_v f |  \dv\dx  \lesssim \rmd^{1/2} \rmv^{1/2} .
\]
For $D^3$ we absorb the pure $x$-partials into  the diffusion available in budget of \eqref{e:mDB},
\[
D^3 \leq \e \rmx + c \rmv.
\]
Finally, arguing similarly 
\[
D^4 \lesssim \rmd^{1/2} \rmv^{1/2} .
\]

We conclude so far,
\begin{equation}\label{e:mDB}
\dot{\rmm} = - (1 -\e) \rmx +  \rmd^{1/2} \rmv^{1/2} + \rmv + B.
\end{equation}
It remains to estimate $B$. Let us write 
\[
B_{k,q} = \int_\domain \jap{v}^{q-2k} \p_x^{\bk}\n_v \cdot ( \rmb f)    \p^{\bk-1}_{x} \p_v  f  \dv\dx + \int_\domain  \jap{v}^{q-2k}  \p^{\bk}_{x}   f  \p^{\bk-1}_{x} \p_v \n_v \cdot (\rmb f)  \dv\dx.
\]
At this point we observe that when all the derivatives fall on $f$, one obtains a full derivative in $v$,
\begin{equation*}\label{}
\begin{split}
& \int_\domain \jap{v}^{q-2k} \rmb \cdot  \n_v (\p_x^{\bk} f)    \p^{\bk-1}_{x} \p_v  f  \dv\dx + \int_\domain  \jap{v}^{q-2k}  \p^{\bk}_{x}   f \rmb \cdot \n_v (\p^{\bk-1}_{x} \p_v f)  \dv\dx \\
 = & \int_\domain  \jap{v}^{q-2k} \rmb \cdot \n_v ( \p^{\bk}_{x}   f \p^{\bk-1}_{x} \p_v f ) \dv\dx = -  \int_\domain  \n_v \cdot( \jap{v}^{q-2k} \rmb ) \p^{\bk}_{x}   f \p^{\bk-1}_{x} \p_v f  \dv\dx\\
  \lesssim & \int_\domain  \jap{v}^{q-2k} | \p^{\bk}_{x}   f \p^{\bk-1}_{x} \p_v f  | \dv\dx \leq \e \rmx + \rmv.
\end{split}
\end{equation*}

 Next batch consists of terms when at least one derivative in $v$ falls on $\rmb$. Then $\p_v \rmb $, which we generically denote $\rmc$, becomes a uniformly bounded smooth coefficient:
 
 \begin{equation*}\label{}
\int_\domain \jap{v}^{q-2k} \p_x^{\bk}( \rmc f)    \p^{\bk-1}_{x} \p_v  f  \dv\dx + \int_\domain  \jap{v}^{q-2k}  \p^{\bk}_{x}   f  \p^{\bk-1}_{x} \p_v ( \rmc f)  \dv\dx: = I + II
\end{equation*} 
Expanding $\p_x^{\bk}( \rmc f) $ we go down to zero-order terms and therefore the estimate this time involves the bounded  neutral term $\rmn$:
\begin{equation}\label{ }
I \lesssim   ( \rmx + \rmn)^{1/2} \rmv^{1/2} \lesssim \e \rmx + c+  \rmv.
\end{equation}
Next,
\[
II = \int_\domain  \jap{v}^{q-2k}  \p^{\bk}_{x}   f  \p^{\bk-1}_{x} (\p_v  \rmc f)  \dv\dx + \int_\domain  \jap{v}^{q-2k}  \p^{\bk}_{x}   f  \p^{\bk-1}_{x} ( \rmc  \p_v  f)  \dv\dx: = II_1 + II_2.
\]
We trivially have
\[
II_2 \lesssim \rmx^{1/2} \rmv^{1/2} \leq \e \rmx + \rmv.
\]
In the term $II_1$, however, we lost all derivatives in $v$ to be used on $f$. So, this term bears only $\rmx$ information, which on its face is large compared to dissipation. To avoid this we use the fact that the second integrand in this term carries derivatives up to $k-1$, which allows to use an interpolation inequality to extract a fraction of $\rmx$ only. First, we split
\begin{equation*}\label{}
\begin{split}
II_1 & = \int_\domain  \jap{v}^{q-2k}  \p^{\bk}_{x}   f  (\p^{\bk-1}_{x} \p_v  \rmc) f  \dv\dx + \sum_{0<\bk'\leq \bk-1} \int_\domain  \jap{v}^{q-2k}  \p^{\bk}_{x}   f  \p^{\bk-1 - \bk'}_{x} \p_v  \rmc  \p_x^{\bk'} f  \dv\dx \\
& \leq \rmx^{1/2} \rmn^{1/2} + \rmx^{1/2} \sum_{0<\bk'\leq \bk-1} \left( \int_\domain  \jap{v}^{q-2k}|\p_x^{\bk'} f|^2  \dv\dx \right)^{1/2}\\
\intertext{and using \eqref{e:interinxonly},}
& \leq \rmx^{1/2} \rmn^{1/2} + \rmx^{1/2} \left( \sum_{0<\bk'\leq \bk-1} \rmx^{\th_{\bk'}} \rmn^{1- \th_{\bk'}}  \right)^{1/2} \lesssim \e \rmx + C.
\end{split}
\end{equation*}
Thus, collecting the above
\[
I + II \leq \e \rmx + \rmv+ C.
\]

The next and final batch of terms consist of those that have all the $v$-derivatives on $f$, and since we already took care of all derivatives on $f$, this time at least one $x$-derivative must fall on $\rmb$:
\begin{equation*}\label{}
\begin{split}
III & = \sum_{0\leq \bk' < \bk} \int_\domain \jap{v}^{q-2k} \p_x^{\bk - \bk'} \rmb\,  \p_x^{\bk'} \p_v f \,   \p^{\bk-1}_{x} \p_v  f  \dv\dx +  \sum_{0\leq \bk' < \bk-1}  \int_\domain  \jap{v}^{q-2k}  \p^{\bk}_{x}   f \,  \p^{\bk-1 - \bk'}_{x} \rmb \, \p^{\bk'}_{x}\p_v^2 f  \dv\dx \\
& \leq \sum_{0\leq \bk' < \bk} \int_\domain \jap{v}^{q-2k+1} |\p_x^{\bk'} \p_v f | |   \p^{\bk-1}_{x} \p_v  f |  \dv\dx +  \sum_{0\leq \bk' < \bk-1}  \int_\domain  \jap{v}^{q-2k+1}  |\p^{\bk}_{x}   f || \p^{\bk'}_{x}\p_v^2 f  | \dv\dx \\
& \leq \rmv + \rmx^{1/2} \sum_{0\leq \bk' \leq \bk-2} \left( \int_\domain  \jap{v}^{q-2(k-2)-2}  | \p^{\bk'}_{x}\p_v^2 f  |^2 \dv\dx \right)^{1/2} \leq \rmv + \rmx^{1/2}\rmv^{1/2} \lesssim \e \rmx + \rmv.
\end{split}
\end{equation*}
All in all, we obtain
\[
B \leq  \e \rmx + \rmv + C.
\]
Thus,

\begin{equation}\label{e:mDB2}
\dot{\rmm} = - (1 -c \e) \rmx +  \rmd^{1/2} \rmv^{1/2} + \rmv + c
\end{equation}
as desired.

\end{proof}

\begin{lemma}\label{ }
We have
\begin{align}
\rmv & \lesssim \sum_{i=1}^N \rmx^{\th'_i} \rmd^{\th''_i}, \quad  0<\th'_i+\th''_i<1 \label{l:vxd}\\
\rmm & \lesssim \sum_{i=1}^N \rmh^{\th_i}, \quad 0< \th_i <1.  \label{l:mh}
\end{align}
\end{lemma}
\begin{proof}

Let us first verify \eqref{l:vxd}. For any $2k+l \leq 2m$, $l \geq 1$, we apply \eqref{e:intergen} with $K = k+ \frac{l}{2}$ and $L = l+1$, and using that the neutral term is bounded,
\[
\rmh^{k,l}_{q - 2k - l}\lesssim \sum_{i=1}^N  \left(  \int_{\domain}  |  D^{k+\frac{l}{2}}_{x}  f |^2 \jap{v}^{q-2k-l} \dv\dx \right)^{\th_i'} \left(   \int_{\domain} \sum_{l' \leq L} |  \p^{l+1}_{v}  f |^2  \jap{v}^{q-2k-l} \dv\dx \right)^{\th_i''},
\]
with $\th_i' + \th_i''<1$. So, if $l$ is even, then the $x$-integral on the right hand side consists of integer order derivatives and therefore,
\begin{equation}\label{e:hxd}
\rmh^{k,l}_{q - 2k - l}\lesssim \sum_{i=1}^N \rmx^{\th_i'} \rmd^{\th_i''}.
\end{equation}
If $l$ is odd, $l = 2p + 1$, and hence $k+p+1 \leq m$, we use interpolation in $x$ and distribute the weight accordingly,
\begin{equation*}\label{}
\begin{split}
 \int_{\domain} &  |  D^{k+p + \frac12}_{x}  f |^2 \jap{v}^{q-2k-l} \dv\dx  \leq \int_{\R^n} \left( \int_\T^n   |  \p_x^{k+p + 1} f |^2 \dx \right)^{1/2} \left( \int_\T^n   |  \p_x^{k+p }  f |^2 \dx \right)^{1/2}  \dv\\
  = & \int_{\R^n}  \left( \int_\T^n   |  \p_x^{k+p + 1} f |^2  \jap{v}^{q-2k-2p-2}\dx \right)^{1/2} \left( \int_\T^n   |  \p_x^{k+p }  f |^2 \jap{v}^{q-2k-2p}\dx \right)^{1/2}  \dv\\
 \leq &\left( \int_\domain   |  \p_x^{k+p + 1} f |^2  \jap{v}^{q-2k-2p-2} \dv \dx \right)^{1/2} \left( \int_\domain   |  \p_x^{k+p }  f |^2 \jap{v}^{q-2k-2p} \dv \dx \right)^{1/2}  
\end{split}
\end{equation*}
So, if $k=p=0$, then the above is bounded by $\lesssim \rmx^{1/2}$, 
and if $k+p >0$, then by $\lesssim \rmx$. In either case we arrive at the same bound \eqref{e:hxd}. 

Summing up over $2k+l\leq 2m$, $l\geq 1$, and reindexing the exponents, we obtain
\[
\rmv \lesssim   \sum_{i=1}^N   \rmx^{\th'_i} \rmd^{\th''_i},
\]
with $N$ depending only on the indexes.

As to the mixed term, we have 
\begin{equation*}\label{}
\begin{split}
 \int_\domain  \jap{v}^{q-2k}  \p^{\bk}_{x} f  \cdot  \p^{\bk-1}_{x} \p_v f  \dv\dx & \leq \left( \int_\domain   |  \p_x^{k} f |^2  \jap{v}^{q-2k} \dv \dx \right)^{1/2} \left( \int_\domain   |  \p_x^{k-1} \p_v  f |^2 \jap{v}^{q-2k} \dv \dx \right)^{1/2}  \\
 \intertext{ using \eqref{e:intergen} with $K = k$, $L = 2k$,}
 & \leq \rmh^{1/2} \sum_i ( \rmh_{q-2k}^{k,0} )^{\th_i' /2} ( \rmh_{q-2k}^{0,2k} )^{\th_i''/2} \leq \sum_i \rmh^{\th_i} ,
 \end{split}
\end{equation*}
where all $\th_i<1$. After adding up the terms and relabeling the exponents, we obtain
\[
\rmm \lesssim \sum_i \rmh^{\th_i}.
\]

\end{proof}

It remains to collect the obtained inequalities and prove the generation of bound lemma.
\begin{lemma}\label{l:generation}
Suppose time-dependent quantities $\rmh,\rmd,\rmx,\rmv, \rmn \geq 0$ and $\rmm$ on time interval $(0,T)$ satisfy the following set of inequalities
\begin{align}
\dot{\rmh} & \leq -c_1 \rmd + c_2 \rmh\\
\dot{\rmm} & \leq - c_3 \rmx + c_4 \rmv + c_5 \rmv^{1/2} \rmd^{1/2} + c_6\\
\rmh & = \rmx + \rmv + \rmn \\
\rmn & \leq \bar{\rmn}\\
\rmv & \leq  c_7 \sum_{i=1}^N \rmx^{\th'_i} \rmd^{\th''_i}, \quad  0<\th'_i+\th''_i<1 \label{e:vxd}\\
| \rmm |  & \leq c_8 \sum_{i=1}^N \rmh^{\th_i}, \quad 0< \th_i <1.  \label{e:mh}
\end{align}
Then
\begin{equation}\label{e:hreg}
\rmh \leq \frac{C}{t^{\k}},
\end{equation}
for all $t\in (0,T)$, and for some $\k,C>0$ depending on all the constants above.
\end{lemma}
\begin{proof}
We argue as in \cite[Lemma A.20]{Villani-hypo}: in order to prove \eqref{e:hreg} it is sufficient to show that for any $E>0$ the length of the time interval $I \ss (0,T)$ on which $E/2 < \rmh <E$ is estimated by $|I| \lesssim 1/ E^{\k}$, $\k >0$.   This will be done by showing the bound
\[
|I| E \lesssim \sum_{i=1}^M  |I|^{\d'_i} E^{\d''_i},
\]
for some $0 \leq \d'_i, \d''_i < 1$.  Then picking the index $i$ for which the summand is maximal we achieve
\[
|I| \lesssim \frac{1}{E^{(1-\d''_i) / (1-\d'_i)}}.
\]

We assume without loss of generality that $E$ is larger than some fixed constant, namely, $E > 1$ and $E> 4 \bar{n}$,  for otherwise the bound $|I| \lesssim 1/ E^{\k}$ is trivial.  Also, by rescaling all quantities by $e^{-c_2 t} \rmh$ we can assume that 
\[
\dot{\rmh}  \leq -c_1 \rmd .
\]
Denoting $I = (t_1,t_2)$ and integrating of the above over $I$ we have
\[
\int_I \rmd(t) \dt \lesssim \rmh(t_1)  \leq  E.
\]
Noting that $\rmx \leq E$, we have
\[
\int_I \rmv \dt \lesssim  \sum_{i=1}^N \int_I  \rmx^{\th'_i} \rmd^{\th''_i} \dt \leq \sum_{i=1}^N |I|^{1-\th'_i - \th''_i} \left(\int_I \rmx \dt \right)^{\th'_i} \left(\int_I \rmd \dt \right)^{\th''_i}\leq \sum_{i=1}^N |I|^{1-\th'_i - \th''_i} E^{\th'_i + \th''_i} = \sum_{i=1}^N |I|^{1-\th'''_i} E^{\th'''_i}, 
\]
where $0<\th'''_i = \th_i'+ \th_i'' <1$. 

Next, from the $\rmm$-equation, and \eqref{e:mh},
\begin{equation*}\label{}
\begin{split}
\int_I \rmx \dt & \lesssim |\rmm(t_1)| + |\rmm(t_2)|+\int_I \rmv \dt + \int_I \rmv^{1/2} \rmd^{1/2} \dt + \int_I 1 \dt \\
& \lesssim \sum_i (E^{\th_i}  +|I|^{1-\th'''_i} E^{\th'''_i}) +(\sum_{i=1}^N |I|^{1-\th'''_i} E^{\th'''_i})^{1/2} E^{1/2} + 1.
\end{split}
\end{equation*}
Note that $1$ can be thought of as $|I|^0 E^0$, so it fits within the allowed range of powers. Reindexing, we obtain
\[
\int_I \rmx \dt \lesssim \sum_i |I|^{\d_i'} E^{\d_i''}.
\]

Since $E> 4 \bar{n}$, we have $\rmx + \rmv \geq E/4$, and hence, from the above,
\[
|I| E \lesssim \int_I (\rmx + \rmv ) \dt \lesssim  \sum_i |I|^{\d_i'} E^{\d_i''},
\]
which finishes the proof.
\end{proof}

This finishes the proof of the required bound on $H^m_q$. It remains to observe that directly from the equation \eqref{e:FPgen},
\[
\|\p_t f(t)\|_{H^{m-1}_{q-3}} \leq C \|f(t)\|_{H^m_q}.
\]
This finishes the proof.

\end{proof}

\section{Further comments}

The existence and uniqueness of weak solutions can be extended to include the open space $\O = \R^n$, as it was done already in \cite{KMT2013} for the Cucker-Smale model. Solutions are assumed to have at least 2nd finite moments in both $v$ and $x$, $f\in L^1(1+ |x|^2 + |v|^2) \cap L^\infty$.  The renormalization follows in a similar fashion since the distributional formulation is understood relative to compactly supported test-functions. In particular, on any ball $B_R$ we can still use \lem{e:KMT} even though $C = C(R)$. We therefore see the whole content of \thm{t:weak} recoverable in these settings. The gain of Gaussian tails in \eqref{e:Gauss}, however, requires compactness of $\O$ at least at the stage of showing the spread in $x$ if we want to control the coefficients only using the initial entropy like in \prop{p:Gauss}. It is conceivable to complete the argument with Gaussian decay in both $v$ and $x$-direction.  

The main issue related to the open space is the lack of thickness of the flock, due to the fundamental restriction coming from finiteness of the mass $\int_{\R^n} \rho \dx = 1$. Even if the spread of positivity is established we inevitably confront the problem of decay at $x$-infinity. In particular, the content of \lem{l:spgap} becomes problematic as well as the relaxation for large data result of \prop{p:relax}.  It is conceivable that implementing, say, a quadratic confinement force $- \n_x U$, $U \sim |x|^2$ in the model, can help metigate the lack of compactness and stabilize long time behavior.  The global hypoellipticity of \prop{p:reg} would likely require use of Sobolev spaces with weights in $x$ as well.

Back on the torus $\O = \T^n$, with regard to other non-symmetric protocols such as $\cC_\g$, or symmetric but not type $(2,\infty)$ protocols, it is possible to establish relaxation and regularization results for perturbative data. Such results have already been discussed in \cite{S-EA}.  
 If the initial Fisher information $\cI(f_0)$ is small then by the \CK, the density is close to uniform,
 \begin{equation}\label{e:rhonearunif}
\| \rho - 1/|\O| \|_1 \leq \d_0,
\end{equation}
which is turn implies uniform thickness by \ref{i:th3}. So, at least on a short interval of time this will remain to be the case, and thus all the coefficients are in $C^\infty$ and $\s>c_0$, which brings us into the framework of global hypoellipticity. With the conclusions of \prop{p:reg}, we can justify the hypocoercivity analysis. It is shown in \cite[Propositions 4.16, 4.18]{S-EA} that the spectral gap for all core models will remain uniformly bounded from zero for data close to uniform \eqref{e:rhonearunif}. We then obtain an estimate of type \eqref{e:HILyap} which shows that $\cI(f)$ remains small on that same interval of time.  The argument then closes to show that the solution will never leave a neighborhood of the Maxwellian at least in the Fisher metric, and thus will exponentially relax to the momentum-centered Maxwellian $\mu_{\s_0,\bar{u}}$, where $\bar{u}$ depends on time, see \cite[Proposition 8.5]{S-EA} for a formal stability argument. The only remaining issue to be addressed is stabilization of the momentum $\bar{u}$ itself. As we demonstrated in \cite[Lemma 8.6]{S-EA}, $\bar{u}$ indeed stabilizes to some $u_\infty\in \R^n$ exponentially fast provided the protocol $\cC$ is conservative relative to the uniform distribution $\bar{\rho} = 1/ |\O|$:
\begin{equation}\label{e:consunif}
\st^*_{\bar{\rho}} = \st_{\bar{\rho}}.
\end{equation}
It is easy to verify that all our core models fulfill this condition. So, we obtain the following small data result. 

\begin{theorem}\label{t:mainrelaxfinal}
For all the core models \ref{Cg}, \ref{Cg+}, \ref{Cg/2}, \ref{Segg}, for any $0\leq \g\leq 1$, there exists a  $\d>0$ depending only on the parameters of the protocol $\cC$ such that all the conclusions of \thm{t:mainrelaxCS} hold for initial data satisfying
\begin{equation}\label{ }
\cI(f_0) \leq  \s_0 \d,
\end{equation}
and with $\bar{u} \to u_\infty$ exponentially fast.
\end{theorem}

Note that for non-conservative protocols $u_\infty$ is an emerging vector which may not coincide with the initial momentum.

\section{Appendix: interpolation in weighted Sobolev spaces}
For any any $p\in \N$ we denote for short $ \p_x^p f = ( \p_{x_1}^p f, \ldots, \p_{x_n}^p f)$, and $|\p_x^p f |^2 = \sum_i |\p_{x_i}^p f |^2$. Similar notation will be used in $v$-variable. It is clear that $|\p_x^p f |^2$ represents a linear combination of $p$th order derivatives.

 Let us first note the following simple interpolation inequality:  for $k,l \geq 0$ and $K,L >0$ such that $\frac{k}{K} + \frac{l}{L} <1$ and any $|\bk| = k$, $|\bl | = l$, we have
\begin{equation}\label{e:interpol}
\int_\domain |\p_x^{\bk} \p_v^{\bl} f|^2 \dv \dx \leq  \left(  \int_{\domain}  |  \p^{K}_{x}  f |^2 \dv\dx \right)^{\frac{k}{K}} \left(   \int_{\domain} |  \p^{L}_{v}  f |^2 \dv\dx \right)^{\frac{l}{L}} \left(  \int_{\domain}  |  f |^2 \dv\dx \right)^{1 -\frac{k}{K} - \frac{l}{L}}.
\end{equation}
Indeed, denoting by $\dxi$ the counting measure over $\Z^n$,
\begin{equation*}\label{}
\begin{split}
\int_\domain |\p_x^{\bk} \p_v^{\bl} f|^2 \dv \dx & = \int_{\Z^n \times\R^n} \Pi_{i=1}^n |\xi_i|^{2k_i} |\eta_i|^{2l_i} |\hat{f}|^2 \deta \dxi \leq \int_{\Z^n \times\R^n}  |\xi|^{2k} |\eta|^{2l} |\hat{f}|^2 \deta \dxi \\
& \leq \int_{\Z^n \times\R^n}  |\xi|^{2k} |\hat{f}|^{2k/K} |\eta|^{2l}|\hat{f}|^{2l/L} |\hat{f}|^{2(1 - k/K - l/K)}  \deta \dxi \\
& \leq  \left(  \int_{\Z^n \times\R^n} |\xi|^{2K} |\hat{f}|^2  \deta \dxi  \right)^{\frac{k}{K}} \left(   \int_{\Z^n \times\R^n}  |\eta|^{2L}|\hat{f}|^{2}\deta \dxi  \right)^{\frac{l}{L}} \left(  \int_{\Z^n \times\R^n}  |  f |^2 \deta \dxi \right)^{1 -\frac{k}{K} - \frac{l}{L}}\\
& \lesssim \left( \sum_i  \int_{\Z^n \times\R^n}  |\xi_i|^{2K} |\hat{f}|^2  \deta \dxi  \right)^{\frac{k}{K}} \left(  \sum_i  \int_{\Z^n \times\R^n}  |\eta_i|^{2L}|\hat{f}|^{2}\deta \dxi  \right)^{\frac{l}{L}} \\
 & \hspace{3in}\times \left(  \int_{\Z^n \times\R^n}  |  f |^2 \deta \dxi \right)^{1 -\frac{k}{K} - \frac{l}{L}}\\
& =   \left(  \int_{\domain}  |  \p^{K}_{x}  f |^2 \dv\dx \right)^{\frac{k}{K}} \left(   \int_{\domain} |  \p^{L}_{v}  f |^2 \dv\dx \right)^{\frac{l}{L}} \left(  \int_{\domain}  |  f |^2 \dv\dx \right)^{1 -\frac{k}{K} - \frac{l}{L}}.
\end{split}
\end{equation*}

When a weight is involved that depends only on $v$, $\w=\w(v)$, then the above interpolation extends trivially in the case when $\bk = 0$ by application of \eqref{e:interpol} in $x$-variable only, and the \HI \ in $v$:
\begin{multline}\label{e:interinxonly}
\int_\domain  |\p_x^{\bk} f|^2  \w \dv \dx \leq \int_{\R^n} \left(  \int_{\T^n}  |  \p^{K}_{x}  f |^2 \dx \right)^{\frac{k}{K}}  \left(  \int_{\T^n}  |  f |^2 \dx \right)^{1 - \frac{k}{K}} \w \dv \\ 
\leq \left(  \int_{\domain}  |  \p^{K}_{x}  f |^2 \w \dv \dx \right)^{\frac{k}{K}}  \left(  \int_{\domain}  |  f |^2 \w \dv \dx \right)^{1 - \frac{k}{K}} .
\end{multline}

It is clear that the above inequalities extend to any fractional $k,K$ as well if we replace the derivatives with Fourier multipliers $D_x^k f$  with symbol $ |\xi|^k$. In fact the computations above already use that symbol even for integer order derivatives. It is important to keep in mind though that if $k$ is integer, then since $\w$ is not involved in Fourier transform, we have
\[
 \int_{\domain}  |  D_x^{k}  f |^2 \w \dv \dx  \sim  \int_{\domain}  | \p_x^{k}  f |^2 \w \dv \dx. 
\]
In other words we can go back to the classical derivatives. 

\begin{lemma}\label{ }
Suppose $\w$ is a doubling weight,
\begin{equation}\label{ }
\begin{split}
c \leq \frac{\w(v')}{\w(v'')} \leq C, & \quad \frac12 \leq \frac{|v'|}{|v''|} \leq 2\\
\w(v) \sim 1, & \quad |v| \leq 1.
\end{split}
\end{equation}
Let $k,l \geq 0$ and $K,L >0$ such that $\frac{k}{K} + \frac{l}{L} <1$,  and let $|\bk| = k$, $|\bl | = l$. Then there exists constants $N \in\N$, $0\leq \th_i',\th_i'',\th_i''' \leq 1$ with $\th_i'+\th_i''+\th_i''' = 1$ and $\th_i'''>0$, that depend only on $k,K,l,L$,  such that 
\begin{multline}\label{e:intergen}
\int_\domain  |\p_x^{\bk} \p_v^{\bl} f|^2  \w \dv \dx \\ \leq \sum_{i=1}^N  \left(  \int_{\domain}  |  \p^{K}_{x}  f |^2 \w \dv\dx \right)^{\th_i'} \left(   \int_{\domain} \sum_{l' \leq L} |  \p^{l'}_{v}  f |^2 \w \dv\dx \right)^{\th_i''} \left(  \int_{\domain}  |  f |^2 \w \dv\dx \right)^{\th_i'''}.
\end{multline}
Here, $K$ can be fractional with $\p_x^K$ being interpreted as above.
\end{lemma}

\begin{proof}
We argue by  induction on $\mu$.  For the base case when $\mu = 0$, the inequality has been proved above.

Let us now assume that  $\mu \geq 1$ and the inequality holds for the indices up to $\mu-1$. We will use the Littlewood-Paley decomposition of the kinetic $v$-space. Fix $\phi \in C^\infty$, $\phi(r) = 1$ for $r<1/2$, and $\phi(r) = 0$ for $r>1$. Define, $\psi^2(r) = \phi(r/2) - \phi(r)$, and $\psi^2_p = \psi^2(r/2^p)$, $\psi^2_{-1} = \phi$. Then $\sum_{p=-1}^\infty \psi^2_p = 1$. We have
\begin{equation*}\label{}
\begin{split}
\int_\domain  |\p_x^{\bk} \p_v^\bl f|^2  \w \dv \dx&  = \sum_{p=-1}^\infty \int_\domain   |\p_x^{\bk} (  \psi_p \p_v^\bl f) |^2  \w \dv \dx \\
& \leq  \sum_{p=-1}^\infty \int_\domain   |\p_x^{\bk}  \p_v^\bl (  \psi_p f) |^2  \w \dv \dx + \sum_{\bl' < \bl} \sum_{p=-1}^\infty \int_\domain   |\p_v^{\bl - \bl'} \psi_p|^2  |\p_x^{\bk}  \p_v^{\bl'}  f |^2  \w \dv \dx
\end{split}
\end{equation*}
Note that $ |\p_v^{\mu - m} \psi_p|^2$ is a geometrically decreasing sequence of  functions supported on the dyadic shell $\{2^{p-1}\leq |v|\leq 2^{p+1}\}$. The shells overlap at most by two, so the last sum adds up to 
\[
\lesssim  \sum_{\bl' < \bl}   \int_{\domain}  |\p_x^{\bk}  \p_v^{\bl'}  f |^2  \w \dv \dx.
\]
By the induction hypothesis all these terms are bounded as desired. So, let us focus on the main term. Using \eqref{e:interpol}, and the doubling property of the weight,
\begin{equation*}\label{}
\begin{split}
& \sum_{p=-1}^\infty \int_\domain   |\p_x^{\bk}  \p_v^\bl (  \psi_p f) |^2  \w \dv \dx  \sim \sum_{p=-1}^\infty  \w(2^p) \int_\domain   |\p_x^{\bk}  \p_v^\bl (  \psi_p f) |^2  \dv \dx\\
 \lesssim & \sum_{p=-1}^\infty  \w(2^p)\left(  \int_{\domain} \psi_p^2  |  \p^{K}_{x}  f |^2 \dv\dx \right)^{\frac{k}{K}} \left(   \int_{\domain} |  \p^{L}_{v} (\psi_p  f) |^2 \dv\dx \right)^{\frac{l}{L}} \left(  \int_{\domain} \psi_p^2   |  f |^2 \dv\dx \right)^{1 -\frac{k}{K} - \frac{l}{L}}
\end{split}
\end{equation*}
Distributing the weights $\w(2^p)$ inside the integrals, we use again the doubling property noting that each integral is over the $p$-th dyadic shell,
\[
\lesssim  \sum_{p=-1}^\infty \left(  \int_{\domain} \psi_p^2  |  \p^{K}_{x}  f |^2 \w \dv\dx \right)^{\th'} \left(   \int_{\domain} |  \p^{L}_{v} (\psi_p  f) |^2  \w \dv\dx \right)^{\th''} \left(  \int_{\domain} \psi_p^2   |  f |^2 \w \dv\dx \right)^{\th'''},
\]
and by the \HI,
\[
\begin{split}
\lesssim &  \left( \sum_{p=-1}^\infty \int_{\domain} \psi_p^2  |  \p^{K}_{x}  f |^2 \w \dv\dx \right)^{\th'} \left(  \sum_{p=-1}^\infty \int_{\domain} |  \p^{L}_{v} (\psi_p  f) |^2  \w \dv\dx \right)^{\th''} \left(\sum_{p=-1}^\infty  \int_{\domain} \psi_p^2   |  f |^2 \w \dv\dx \right)^{\th'''}\\
\lesssim & \left( \int_{\domain}  |  \p^{K}_{x}  f |^2 \w \dv\dx \right)^{ \th' } \left(  \sum_{l' \leq L} \int_{\domain} |  \p^l_v  f |^2  \w \dv\dx \right)^{\th''} \left( \int_{\domain}   |  f |^2 \w \dv\dx \right)^{\th'''},
\end{split}
\]
which has the required form.

\end{proof}

\begin{lemma}\label{l:wappr}
Suppose $\w$ is a doubling weight and $f_\e$ denotes convolution with the standard compactly supported mollifier $\chi_\e$. Then
\begin{align}
\|f_\e\|_{L^2(\w)} & \leq C\|f\|_{L^2(\w)} \label{e:wappr} \\
\|f - f_\e\|_{L^2(\w)} & \to 0, \text{ as } \e \to 0.\label{e:wapprconv}
\end{align}
\end{lemma}
\begin{proof}
We have 
\begin{equation*}\label{}
\begin{split}
\|f_\e\|_{L^2(\w)}^2  & \leq \int_\domain \chi_\e(y,w) \int_\domain  \w(v) |f(x-y,v-w)|^2 \dv \dx  \dw \dy \\
& = \int_\domain \chi_\e(y,w) \int_\domain  \w(v) |f(x,v-w)|^2 \dv \dx  \dw \dy.
\end{split}
\end{equation*}
If $|v|\leq 1$, then $|v -w| \leq 2$ for small $\e$, and thus, $\w(v) \sim \w(v-w) \sim 1$. Otherwise, if $|v| >1$, then $\w(v) / \w(v-w) \sim 1$ by the doubling property. Hence,
\[
\|f_\e\|_{L^2(\w)}^2 \lesssim  \int_\domain \chi_\e(y,w) \int_\domain  \w(v-w) |f(x,v-w)|^2 \dv \dx  \dw \dy = \|f\|_{L^2(\w)}^2.
\]

The convergence \eqref{e:wapprconv} follows from the standard Lebesgue theory.

\end{proof}


\begin{thebibliography}{10}

\bibitem{ABFHKPPS}
G.~Albi, N.~Bellomo, L.~Fermo, S.-Y. Ha, J.~Kim, L.~Pareschi, D.~Poyato, and
  J.~Soler.
\newblock Vehicular traffic, crowds, and swarms: {F}rom kinetic theory and
  multiscale methods to applications and research perspectives.
\newblock {\em Math. Models Methods Appl. Sci.}, 29(10):1901--2005, 2019.

\bibitem{AZ2024}
Francesca Anceschi and Yuzhe Zhu.
\newblock On a spatially inhomogeneous nonlinear {F}okker-{P}lanck equation:
  {C}auchy problem and diffusion asymptotics.
\newblock {\em Anal. PDE}, 17(2):379--420, 2024.

\bibitem{Axel97}
Robert Axelrod.
\newblock {\em The Complexity of Cooperation: Agent-Based Models of Competition
  and Collaboration}.
\newblock Princeton University Press, 1997.

\bibitem{Ben2005}
E.~Ben-Naim.
\newblock Opinion dynamics: Rise and fall of political parties.
\newblock {\em Europhys. Lett.}, 69:671--677, 2005.

\bibitem{BCC2011}
Fran\c{c}ois Bolley, Jos\'{e}~A. Ca\~{n}izo, and Jos\'{e}~A. Carrillo.
\newblock Stochastic mean-field limit: non-{L}ipschitz forces and swarming.
\newblock {\em Math. Models Methods Appl. Sci.}, 21(11):2179--2210, 2011.

\bibitem{CCFS2008}
A.~Cheskidov, P.~Constantin, S.~Friedlander, and R.~Shvydkoy.
\newblock Energy conservation and {O}nsager's conjecture for the {E}uler
  equations.
\newblock {\em Nonlinearity}, 21(6):1233--1252, 2008.

\bibitem{Choi2016}
Young-Pil Choi.
\newblock Global classical solutions of the {V}lasov-{F}okker-{P}lanck equation
  with local alignment forces.
\newblock {\em Nonlinearity}, 29(7):1887--1916, 2016.

\bibitem{cet}
P.~Constantin, W.~E, and E.~S. Titi.
\newblock Onsager's conjecture on the energy conservation for solutions of
  {E}uler's equation.
\newblock {\em Comm. Math. Phys.}, 165(1):207--209, 1994.

\bibitem{CS2007a}
F.~Cucker and S.~Smale.
\newblock Emergent behavior in flocks.
\newblock {\em IEEE Trans. Automat. Control}, 52(5):852--862, 2007.

\bibitem{CS2007b}
F.~Cucker and S.~Smale.
\newblock On the mathematics of emergence.
\newblock {\em Jpn. J. Math.}, 2(1):197--227, 2007.

\bibitem{ds-1}
C.~De~Lellis and L.~Sz{\'e}kelyhidi, Jr.
\newblock The {E}uler equations as a differential inclusion.
\newblock {\em Ann. of Math. (2)}, 170(3):1417--1436, 2009.

\bibitem{DesVill2000}
Laurent Desvillettes and C\'{e}dric Villani.
\newblock On the spatially homogeneous {L}andau equation for hard potentials.
  {I}. {E}xistence, uniqueness and smoothness.
\newblock {\em Comm. Partial Differential Equations}, 25(1-2):179--259, 2000.

\bibitem{FranPoli2006}
Marco Di~Francesco and Sergio Polidoro.
\newblock Schauder estimates, {H}arnack inequality and {G}aussian lower bound
  for {K}olmogorov-type operators in non-divergence form.
\newblock {\em Adv. Differential Equations}, 11(11):1261--1320, 2006.

\bibitem{DS2021}
Helge Dietert and Roman Shvydkoy.
\newblock On {C}ucker-{S}male dynamical systems with degenerate communication.
\newblock {\em Anal. Appl. (Singap.)}, 19(4):551--573, 2021.

\bibitem{DiPL1989}
R.~J. DiPerna and P.-L. Lions.
\newblock Ordinary differential equations, transport theory and {S}obolev
  spaces.
\newblock {\em Invent. Math.}, 98(3):511--547, 1989.

\bibitem{DFT2010}
Renjun Duan, Massimo Fornasier, and Giuseppe Toscani.
\newblock A kinetic flocking model with diffusion.
\newblock {\em Comm. Math. Phys.}, 300(1):95--145, 2010.

\bibitem{Shvart2019}
Torquato S Shvartsman SY Krajnc~M. Dutta~S, Djabrayan~NJ.
\newblock Self-similar dynamics of nuclear packing in the early drosophila
  embryo.
\newblock {\em Biophys J.}, 117(4):743--750, 2019.

\bibitem{Edel2001}
Leah Edelstein-Keshet.
\newblock Mathematical models of swarming and social aggregation.
\newblock 2001.

\bibitem{Glassey}
Robert~T. Glassey.
\newblock {\em The {C}auchy problem in kinetic theory}.
\newblock Society for Industrial and Applied Mathematics (SIAM), Philadelphia,
  PA, 1996.

\bibitem{GIMV2019}
Fran\c{c}ois Golse, Cyril Imbert, Cl\'{e}ment Mouhot, and Alexis~F. Vasseur.
\newblock Harnack inequality for kinetic {F}okker-{P}lanck equations with rough
  coefficients and application to the {L}andau equation.
\newblock {\em Ann. Sc. Norm. Super. Pisa Cl. Sci. (5)}, 19(1):253--295, 2019.

\bibitem{GI2021}
Jessica Guerand and Cyril Imbert.
\newblock Log-transform and the weak {H}arnack inequality for kinetic
  {F}okker-{P}lanck equations.
\newblock {\em J. Inst. Math. Jussieu}, 22(6):2749--2774, 2023.

\bibitem{HL2009}
S.-Y. Ha and J.-G. Liu.
\newblock A simple proof of the {C}ucker-{S}male flocking dynamics and
  mean-field limit.
\newblock {\em Commun. Math. Sci.}, 7(2):297--325, 2009.

\bibitem{HT2008}
S.-Y. Ha and E.~Tadmor.
\newblock From particle to kinetic and hydrodynamic descriptions of flocking.
\newblock {\em Kinet. Relat. Models}, 1(3):415--435, 2008.

\bibitem{HaXZh2018}
Seung-Yeal Ha, Qinghua Xiao, and Xiongtao Zhang.
\newblock Emergent dynamics of cucker--smale particles under the effects of
  random communication and incompressible fluids.
\newblock {\em Journal of Differential Equations}, 264(7):4669--4706, 2018.

\bibitem{HST2020}
Christopher Henderson, Stanley Snelson, and Andrei Tarfulea.
\newblock Self-generating lower bounds and continuation for the {B}oltzmann
  equation.
\newblock {\em Calc. Var. Partial Differential Equations}, 59(6):Paper No. 191,
  13, 2020.

\bibitem{Hormander1967}
Lars H\"{o}rmander.
\newblock Hypoelliptic second order differential equations.
\newblock {\em Acta Math.}, 119:147--171, 1967.

\bibitem{IM2021}
Cyril Imbert and Cl\'{e}ment Mouhot.
\newblock The {S}chauder estimate in kinetic theory with application to a toy
  nonlinear model.
\newblock {\em Ann. H. Lebesgue}, 4:369--405, 2021.

\bibitem{IMS2020}
Cyril Imbert, Cl\'{e}ment Mouhot, and Luis Silvestre.
\newblock Gaussian lower bounds for the {B}oltzmann equation without cutoff.
\newblock {\em SIAM J. Math. Anal.}, 52(3):2930--2944, 2020.

\bibitem{Isett2018}
P.~Isett.
\newblock A proof of {O}nsager's conjecture.
\newblock {\em Ann. of Math. (2)}, 188(3):871--963, 2018.

\bibitem{KMT2013}
T.~K. Karper, A.~Mellet, and K.~Trivisa.
\newblock Existence of weak solutions to kinetic flocking models.
\newblock {\em SIAM. J. Math. Anal.}, 45:215--243, 2013.

\bibitem{Kolm1934}
A.~Kolmogoroff.
\newblock Zuf\"{a}llige {B}ewegungen (zur {T}heorie der {B}rownschen
  {B}ewegung).
\newblock {\em Ann. of Math. (2)}, 35(1):116--117, 1934.

\bibitem{Krylov-book}
N.~V. Krylov.
\newblock {\em Lectures on elliptic and parabolic equations in {H}\"{o}lder
  spaces}, volume~12 of {\em Graduate Studies in Mathematics}.
\newblock American Mathematical Society, Providence, RI, 1996.

\bibitem{MT2011}
S.~Motsch and E.~Tadmor.
\newblock A new model for self-organized dynamics and its flocking behavior.
\newblock {\em J. Stat. Phys.}, 144(5):923--947, 2011.

\bibitem{MT2014}
S.~Motsch and E.~Tadmor.
\newblock Heterophilious dynamics enhances consensus.
\newblock {\em SIAM Rev.}, 56(4):577--621, 2014.

\bibitem{Mouhot2005}
Cl\'{e}ment Mouhot.
\newblock Quantitative lower bounds for the full {B}oltzmann equation. {I}.
  {P}eriodic boundary conditions.
\newblock {\em Comm. Partial Differential Equations}, 30(4-6):881--917, 2005.

\bibitem{MP2018}
P.~B. Mucha and J.~Peszek.
\newblock The {C}ucker-{S}male equation: singular communication weight,
  measure-valued solutions and weak-atomic uniqueness.
\newblock {\em Arch. Ration. Mech. Anal.}, 227(1):273--308, 2018.

\bibitem{Darwin}
L.~Perea, P.~Elosegui, and G.~Gomez.
\newblock Extension of the {C}ucker-{S}male control law to space flight
  formations.
\newblock {\em Journal of Guidance, Control, and Dynamics}, 32:526 -- 536,
  2009.

\bibitem{ShuT2019}
R.~Shu and E.~Tadmor.
\newblock Flocking hydrodynamics with external potentials.
\newblock {\em Arch. Ration. Mech. Anal.}, 238(1):347--381, 2020.

\bibitem{S-book}
Roman Shvydkoy.
\newblock {\em Dynamics and analysis of alignment models of collective
  behavior}.
\newblock Ne\v cas Center Series. Birkh\"auser/Springer, Cham, [2021]
  \copyright 2021.

\bibitem{S-hypo}
Roman Shvydkoy.
\newblock Global hypocoercivity of kinetic {F}okker-{P}lanck-alignment
  equations.
\newblock {\em Kinet. Relat. Models}, 15(2):213--237, 2022.

\bibitem{S-EA}
Roman Shvydkoy.
\newblock Environmental averaging.
\newblock {\em EMS Surv. Math. Sci.}, 11(2):277--413, 2024.

\bibitem{S-GAC}
Roman Shvydkoy.
\newblock Generic alignment conjecture for systems of {C}ucker-{S}male type.
\newblock {\em J. Evol. Equ.}, 24(2):Paper No. 19, 18, 2024.

\bibitem{Tadmor-notices}
Eitan Tadmor.
\newblock On the mathematics of swarming: emergent behavior in alignment
  dynamics.
\newblock {\em Notices Amer. Math. Soc.}, 68(4):493--503, 2021.

\bibitem{TV2000}
G.~Toscani and C.~Villani.
\newblock On the trend to equilibrium for some dissipative systems with slowly
  increasing a priori bounds.
\newblock {\em J. Statist. Phys.}, 98(5-6):1279--1309, 2000.

\bibitem{VZ2012}
T.~Vicsek and A.~Zefeiris.
\newblock Collective motion.
\newblock {\em Physics Reprints}, 517:71--140, 2012.

\bibitem{Villani-optimal}
C.~Villani.
\newblock {\em Topics in optimal transportation}, volume~58 of {\em Graduate
  Studies in Mathematics}.
\newblock American Mathematical Society, Providence, RI, 2003.

\bibitem{Villani-hypo}
C\'{e}dric Villani.
\newblock Hypocoercivity.
\newblock {\em Mem. Amer. Math. Soc.}, 202(950):iv+141, 2009.

\end{thebibliography}

\end{document}